\documentclass[12pt]{amsart}
\usepackage{graphicx}
\usepackage{tikz-cd}
\usepackage{hyperref}
\usepackage{graphicx}
\usepackage{epsfig}

\usepackage{xcolor}
 \definecolor{awesome}{rgb}{1.0, 0.13, 0.32}
 \definecolor{notred}{rgb}{0.57, 0.36, 0.51}
\newtheorem{thm}{Theorem}[section]

\newtheorem{defn}[thm]{Definition}
\newtheorem{lemma}[thm]{Lemma}

\newtheorem{cor}[thm]{Corollary}
\newtheorem{remark}[thm]{Remark}
\newtheorem{example}[thm]{Example}

\usepackage{amsmath}
\usepackage{amsxtra}
\usepackage{amscd}
\usepackage{amsthm}
\usepackage{amsfonts}
\usepackage{amssymb}
\usepackage{eucal}
\newcommand{\barQ}{$\bar{\rm Q}$}
\newcommand{\barD}{\bar D}
\newcommand{\bmb}{\left( \begin{array}{rr}}
\newcommand{\enm}{\end{array}\right)}

\newcommand{\cQ}{\mathcal Q}

\newcommand{\g}{{\mathfrak{g}}}

\newcommand{\gl}{{\mathfrak{gl}}}

\newcommand{\D}{{\mathcal{D}}}
\newcommand{\cH}{{\mathcal{H}}}
\renewcommand{\sl}{{\mathfrak{sl}}}

\newcommand{\C}{{\mathbb C}}
\newcommand{\Z}{{\mathbb Z}}
\newcommand{\Q}{{\mathcal Q}}
\newcommand{\R}{{\mathbb R}}

\newcommand{\bx}{{\mathbf x}}

\newcommand{\al}{{\alpha}}

\newcommand{\lL}{{\Lambda}}
\newcommand{\K}{{\mathcal K}}

\newcommand{\half}{\frac{1}{2}}

\numberwithin{equation}{section}

\begin{document}

\title{Macdonald Duality and the proof of the Quantum Q-system conjecture}
\author{Philippe Di Francesco} 
\address{Department of Mathematics, University of Illinois MC-382, Urbana, IL 61821, 
U.S.A. e-mail: philippe@illinois.edu and 
Institut de physique th\'eorique, Universit\'e Paris Saclay, 
CEA, CNRS, F-91191 Gif-sur-Yvette, FRANCE\hfill
\break e-mail: philippe.di-francesco@cea.fr
}
\author{Rinat Kedem} 
\address{Department of Mathematics, University of Illinois MC-382, Urbana, IL 61821, 
U.S.A. \hfill
\break e-mail: rinat@illinois.edu}

\begin{abstract} The $SL(2,\Z)$-symmetry of 
Cherednik's spherical double affine Hecke algebras in Macdonald theory includes a distinguished generator
which acts as a discrete time evolution of Macdonald operators, which can also be interpreted as a torus Dehn twist in type $A$.
We prove for all twisted and untwisted affine algebras of  type $ABCD$ that the time-evolved $q$-difference Macdonald operators,
in the $t\to\infty$ $q$-Whittaker limit, form a representation of the associated discrete integrable quantum Q-systems, 
which are obtained, in all but one case, via the canonical quantization of suitable cluster algebras. 
The proof relies strongly on the duality 
property of Macdonald and Koornwinder polynomials, which allows, in the $q$-Whittaker limit, for a unified description of the quantum Q-system variables
and the conserved quantities as limits of the time-evolved Macdonald operators and the Pieri operators, respectively. The latter are identified with relativistic $q$-difference Toda Hamiltonians.
A crucial ingredient in the proof is the use of the ``Fourier transformed" picture, in which we compute time-translation operators and prove that they commute with the Pieri operators or Hamiltonians.
We also discuss the universal solutions of Koornwinder-Macdonald eigenvalue and Pieri equations, 
for which we prove a duality relation, which simplifies the proofs further.
\end{abstract}

\maketitle
\date{\today}
\section{Introduction}

\subsection{Overview}

The purpose of this paper is to determine, in a uniform manner,  a certain discrete integrable structure associated with the finite-dimensional representations of quantum affine algebras or Yangians of affine algebras $\g$, and to understand the relation of this structure with Koornwinder-Macdonald theory and  spherical double affine Hecke algebras (sDAHA) for all $\g$ of non-exceptional type. 

Specifically, the characters of a special subset of representations of these algebras, known as KR-modules \cite{ChariMoura,KR}, satisfy recursion relations which can be viewed as discrete time evolution equations. These  evolutions, called Q-systems \cite{HKOTY,HKOTT}, are known to have an integrable structure for a subset of algebras $\g$. When $\g=A_{N-1}^{(1)}$, 
the integrable structure was shown to be a Coxeter-Toda system \cite{DFK10,GSVct}, and this was generalized  \cite{Williams} to $\g=D_N^{(1)}, A_{2N-1}^{(2)}, D_{N+1}^{(2)}$ using the refactorization map \cite{HKKR}. In type $\g=B_N^{(1)}$ the system was shown to be integrable \cite{Panu} using a Goncharov-Kenyon type dimer model \cite{GoncharovKenyon}. Generally, discrete integrability was a conjecture and does seem to be related to refactorization maps.

The combinatorial structure of Q-system evolution equations is such that, in almost all cases, they can be interpreted as mutations in a cluster algebra of geometric type \cite{clusK,clusDFK}. This  structure can be encoded in the form of  a quiver, which, in type $A_{N-1}^{(1)}$ appears in various contexts such as $K$-theoretic Coulomb branches \cite{SS19}, gauge theories \cite{CNV}, factorization maps \cite{Williams} and shifted quantum affine algebras or Yangians \cite{BFN,FTshiftedquantum,DFKqt}.

Cluster algebras have a natural quantization \cite{BZ}, and in \cite{krKR,Simon} it was shown that this quantization is associated to the grading of the so-called fusion product of KR-modules  \cite{FL}. That is, the character of the graded fusion product can be expressed as a constant term identity involving the corresponding solutions of the quantum Q-system.

The Hamiltonians associated with the quantum Q-system in type $A_{N-1}^{(1)}$ are the $q$-difference (or relativistic) 
quantum Toda Hamiltonians of \cite{Etingof}, and therefore the graded characters of fusion products of fundamental representations in this type can be identified with $q$-Whittaker functions of $U_q(\sl_N)$ \cite{DFK15}. More generally, the quantum Q-system generators can be identified with the algebra of creation and annihilation operators of generalized $q$-Whittaker functions or graded characters. 
This algebra was identified in \cite{DFK15,DFKcur,DFKqt} with the $q$-Whittaker limit $t\to\infty$ of the spherical double affine Hecke algebra of type $A$ \cite{Cheredbook} as follows.
There is a natural $SL(2,\Z)$-symmetry of the sDAHA, and the generator $\tau_+\in SL(2,\Z)$ acts on the functional 
representation as conjugation by Cherednik's Gaussian. The subset of the sDAHA generators consisting of all 
$\tau_+$-translates of the commuting Macdonald difference operators can be identified, in the $q$-Whittaker limit, 
as the set of solutions of the quantum Q-system. This gives a functional representation of the quantum Q-system as an algebra of 
difference operators. In particular, single $\tau_+$-translates of Macdonald operators are the $q$-Whittaker limits of the Kirillov-Noumi raising 
operators for Macdonald polynomials \cite{kinoum}.

There are generalizations of DAHAs and their functional representation to the other classical affine root systems \cite{Cheredbook,CheredDAHA}.
The corresponding Macdonald operators have generalized Macdonald polynomials as common eigenfunctions.
The latter can also be obtained as specializations of the Koornwinder polynomials, which depend on several extra parameters $(a,b,c,d)$ \cite{Koornwinder,YamaYana,vandiej} (see Table \ref{tableone}). These  are eigenfunctions of the Koornwinder difference operator, the first in a family of $N$ commuting difference operators \cite{vandiej}, where $N$ is the rank of the finite sub-algebra $R$. These are elements in the functional representation \cite{Noumireps} of a DAHA depending on $(a,b,c,d)$.

Quantum Q-systems can be defined for all affine algebras.
It is natural to expect that some selected commuting family of difference operators together with their $\tau_+$-translates will, in the  $q$-Whittaker limit, satisfy these non-commutative evolution equations. For a subset of algebras $\g$, we presented a conjectural family of such difference operators in \cite{DFKconj}. In this paper, we generalize and prove these conjectures for all non-exceptional $\g$ using duality and the Fourier transform.

In type $A_{N-1}^{(1)}$, we present a simpler, alternative proof to that of \cite{DFK15}, that the $\tau_+$-translates of the Macdonald difference operators, in the $q$-Whittaker limit, satisfy the quantum Q-system. An essential ingredient in the simplification of the proof is the use of the Macdonald duality property \cite{macdo}.
The Macdonald difference operators
are elements of $\C_t(x)[\Gamma_i]$, written in terms of the generators of the quantum torus
\begin{equation}\label{qtorx}
\mathbb T_x:= \langle x_i, \Gamma_i\rangle_{i=1,...,N}, \qquad \Gamma_i x_j = q^{\delta_{ij}} x_j \Gamma_i 
\end{equation}
acting on functions of $x=(x_1,...,x_N)$. The monic, symmetric Macdonald polynomials $P_\lambda(x)$ are the common eigenfunctions of the Macdonald difference operators, depending on two sets of parameters, the variables $x$ as well as integer partitions $\lambda=(\lambda_1\geq \lambda_2\geq \cdots \geq \lambda_N)$. 
Encoding the partitions using the variables $s=(s_1,...,s_N)$ with $s_i=\lL_i\, t^{\rho_i}$, $\Lambda_i=q^{\lambda_i}$ and $\rho$ is the half-sum of positive roots ($\rho_i=N-i$),
the duality of Macdonald theory is a symmetry property of the properly normalized eigenfunctions of the Macdonald operators under the interchange of the variables $x$ and $s$ \cite{macdo,vDselfdual,Cherednik95}. 

This symmetry can  be used to relate the Macdonald operator eigenvalue equations to the Pieri rules, which express the effect of multiplication of the Macdonald polynomial by elementary symmetric functions, as the action of $q$-difference operators
in the variables $s$ or $\Lambda$. These difference operators are Pieri operators, and can be written in terms of the quantum torus
\begin{equation}\label{qtorlambda}
\mathbb T_\Lambda:=\langle\Lambda_i
, T_i\rangle_{i=1,...,N},\qquad T_i \Lambda_j = q^{\delta_{ij}} \Lambda_j T_i .
\end{equation} 

The eigenvalue equations for Macdonald operators 
with eigenfunctions $P_\lambda(x)$ can be thought of as a special case of the sDAHA ``Fourier transform",
which is a map from operators $f(x)$ to operators $\bar f(\lL)$, such that
\begin{equation}\label{FT} f(x)\,P_\lambda(x)=\bar f(\lL)\, P_\lambda(x) \end{equation}
for all $\lambda$. In particular, the Fourier transforms of the Macdonald operators are the elementary symmetric functions in the variable $s$, and the transform of the
Pieri operators are the elementary symmetric functions in $x$.

In the $q$-Whittaker limit, switching to the ``Fourier transformed" $\mathbb T_\Lambda$ picture results in drastic simplifications.
Indeed, the transforms of the Macdonald operators in the $q$-Whittaker limit are the limits of the elementary symmetric function in $s$, which are the leading monomials in 
$\Lambda_i$ due to the dependence of $s$ on $t$. Because of this, 
the $\tau_+$-translates of the Fourier-transformed operators are simpler. 
In the functional representation as difference operators in $x$-variables, the action of $\tau_+$ is expressed as conjugation by the Cherednik Gaussian 
$\gamma(x)$ \cite{Cheredbook}, which is independent of $t$. In the Fourier-transformed picture in $\Lambda$-variables, we find the explicit operator 
$g(\Lambda)$ which represents the $\tau_+$-translation,  generates the discrete time evolution of the quantum Q-system, and commutes with the Hamiltonians. 

The operator $g(\Lambda)$ is interpreted as a Dehn twist \cite{SS} from a quantum higher Teichmüller theory interpretation in type $A$, 
and can be expressed as a product of quantum dilogarithms. It was shown to be a specialization of the Baxter $Q$-operator of \cite{SS} for the Q-system quiver of type $A_{N-1}^{(1)}$. 
The Baxter $Q$-operator is the generating function of commuting $q$-difference Toda Hamiltonians \cite{Etingof}, which coincide with the $q$-Whittaker limit of the Macdonald Pieri rules.

The above discussion for type $A_{N-1}^{(1)}$ can be extended to the affine algebras, twisted or untwisted, of the form $\g=X_N^{(r)}$, with $X$ a classical Lie algebra of type $ABCD$ and $r=1,2$, in the notation of \cite{Kac,HKOTT}. This includes the cases  presented as
conjectures in \cite{DFKconj}.  In this paper, we prove the relation between the the $q$-Whittaker limit of the $SL(2,\Z)$-translates of appropriately chosen $q$-difference Koornwinder-Macdonald operators, for each algebra $\g$,  with the solutions of the type $\g$ quantum Q-system systems.

At the same time, we find a uniform description of the integrable structure of the quantum Q-systems as discrete dynamical systems with commuting integrals of motion. These are versions of the $q$-difference Toda Hamiltonians for the various root systems.
The classical Q-systems associated with the affine root systems of types $A^{(1)}_N, D^{(1)}_N, A^{(2)}_{2N-1},  D^{(2)}_{N+1}$
are known \cite{Williams} to be the evolution equations of the refactorization maps 
of types $A,D,C,B$, respectively. 
As such, their integrable structure is given by a classical Toda-type lattice \cite{HKKR,Resh03}.
The integrable structure of the Q-systems for $\g=B_N^{(1)},C_N^{(1)}$ and $A^{(2)}_{2N}$, also given by a Toda-type Hamiltonian, but do not arise from factorization dynamics.

Here, we show that the unifying structure in the case of the quantum Q-systems 
is given, instead, by the Koornwinder-Macdonald operators \cite{vDselfdual,Sahi,macdoroot} via the duality relating the eigenvalue equations to the Pieri rules, which become $q$-difference Toda equations \cite{Etingof,gotsy} in the $q$-Whittaker limit $t\to \infty$.

In this limit, the symmetry in the duality relations is broken, and as a result
the Fourier transform of the Macdonald operators acting on $\Lambda$-space is greatly simplified. 
The rationale of our proofs is to identify the Fourier transforms of the $SL(2,\Z)$-translates of the Macdonald operators in the $q$-Whittaker limit and to transform the relations occurring in $\lL$-space back to $x$-space.

We also use another formulation of duality, based on the notion of the ``universal solution" $P(x;s)$ to the Koornwinder-Macdonald eigenvalue equations (c.f. \cite{Cherednik95,cherednikqwhittaker} and the asymptotically-free basic Harish-Chandra series of \cite{Stokman}). In \cite{ShiraishiNoumi}  such solutions are explicitly computed in type $A_{N-1}^{(1)}$ . In the Koornwinder case and its specializations, this suggests the definition of series solutions for the eigenvalue equation in the formal variables $(x,s)$. These truncate to the polynomial eigenfunctions when $\lambda$ is specialized to $\g$-partitions corresponding to dominant integral weights of $\g$, where $s$ is a discrete variable. Duality relates the universal solution $P(x;s)$ of the eigenvalue equation, to the universal solution $Q(s;x)$ of the Pieri equations, and formalizes the Fourier transform.

\subsection{Summary of the main results}

All the results here refer to all affine algebras $\g$ of type $A_{N-1}^{(1)}, B_N^{(1)}, C_N^{(1)}, D_N^{(1)}, A_{2N-1}^{(2)}, D_{N+1}^{(2)}$ and $A_{2N}^{(2)}$. The subsequent labeling of theorems refers to both Section 2 (dedicated to the case of $A_{N-1}^{(1)}$ for pedagogical reasons)
and Section 4 which addresses the cases $B_N^{(1)}, C_N^{(1)}, D_N^{(1)}, A_{2N-1}^{(2)}, D_{N+1}^{(2)}$ and $A_{2N}^{(2)}$. 

For each such algebra $\g$, we define a preferred set of $N$ commuting $q$-difference operators $\{D_a^{(\g)}(x;q,t),\ a\in [1,N]\}$, chosen via specializations of linear combinations of van Diejen's higher Koornwinder operators, augmented by Rains and Macdonald operators for each algebra. The operators are chosen so that their eigenvalues form a basis for the space spanned by the fundamental characters of the underlying finite subalgebras $R^*$ listed in Table \ref{tableone}, with $R^*=A_{N-1}$ in the case of $\g=A_{N-1}^{(1)}$.

Once an appropriate family of commuting difference operators is constructed, we define the $q$-Whittaker limit $t\to\infty$ of the family, $\{D_{a}^{(\g)}(x;q)\}$, as well as their $\tau_+^n\in SL(2,\Z)$-translates, $\{D_{a;n}^{(\g)}(x;q), n\in\Z\}$. 
These operators satisfy the $\g$-type quantum Q-systems, \eqref{Acommutation}, \eqref{AQsys} or \eqref{qsys1}--\eqref{qsys3}. The main results of this paper may be stated as the following Theorems:

\newtheorem*{Atheorem}{Theorems \ref{Atheorem}, \ref{Qsysconj}}
\begin{Atheorem}
For each $\g$, the family of $q$-difference operators $\{D_{a;n}^{(\g)}(x;q), a\in[1,N], n\in\Z\}$ defined in Section \ref{sec:genqW} satisfy the corresponding quantum Q-system relations.
\end{Atheorem}
These operators can be viewed as generalized raising or lowering operators, and in particular, we prove
\newtheorem*{raiseconj}{Theorem \ref{raiseconj}}
\begin{raiseconj}
The $q$-difference operators $\{D_{a,i}^{(\g)}(x;q) : i=0,1\}$ satisfy eigenvalue and raising operator properties:
\begin{eqnarray*}
D_{a;0}^{(\g)}(x;q)\, \Pi_\lambda^{(\g)}(x)&=&\lL^{\omega_a^*}\,\Pi_\lambda^{(\g)}(x),\\
D_{a;1}^{(\g)}(x;q)\, \Pi_\lambda^{(\g)}(x)&=& \lL^{\omega_a^*}\, \Pi_{\lambda+\omega_a}^{(\g)}(x),
\end{eqnarray*}
where $\omega_a,\omega_a^*$ are fundamental weights of $R,R^*$ the finite root lattices of $\g,\g^*$.
\end{raiseconj}

The proof uses the $t\to\infty$ limit of the ``Fourier transform" \eqref{FT} (sometimes called $q$-Whittaker transform), which relies on the completeness of eigenfunctions of the $q$-Whittaker difference operators. Let $\Pi^{(\g)}_\lambda(x)$ be the set of common eigenfunctions of the operators $D_{a}^{(\g)}(x)=D_{a;0}^{(\g)}(x)$, i.e. limits of Macdonald polynomials. The Fourier transform relates operators $f(x)$ in $\mathbb T_x$ to operators $f(\lL)$ in $\mathbb T_\lL$ via $f(x)\, \Pi^{(\g)}_\lambda(x)=\bar f(\lL)\, \Pi^{(\g)}_\lambda(x)$.
Then we define the $q$-difference operators acting on the index variables $\lL$, $\barD_{a,n}^{(\g)}(\lL)$, starting from the initial (Macdonald eigenvalue) condition $D_{a;0}^{(\g)}(x)\,\Pi_\lambda^{(\g)}(x)=\bar D_{a;0}^{(\g)}(\lL)\,\Pi_\lambda^{(\g)}(x)$, and further constrained to obey the $\g$ quantum Q-system with the opposite multiplication. 
Finally we set out to prove that the operators $\barD_{a,n}^{(\g)}(\lL)$ are the Fourier transforms of the operators $D_{a;n}^{(\g)}(x)$.

To this end, we first find the explicit time-translation operators $n\mapsto n+1$ (or $n+2$) acting on the operators $\barD_{a,n}^{(\g)}$ for each $\g$: 
\newtheorem*{goperators}{Theorems \ref{gfunctionA}, \ref{longshort}}
\begin{goperators}
There is a unique, up to scalar multiple, function $g=g^{(\g)}(\Lambda)$ in the generators of the quantum torus $\mathbb T_\Lambda$, such that
$$\barD_{a,n+t_a} = q^{-\omega_a^*\cdot \omega_a t_a/2} g \barD_{a,n} g^{-1},$$
where $t_a=1$ for all labels $a$ except for the so-called short labels $N$ in type $B_{N}^{(1)}$, $[1,N-1]$ in type $C_N^{(1)}$ and $[1,N]$ in type $A_{2N}^{(2)}$, in which case $t_a=2$. Here $\omega_a$ ($\omega_a^*$) are fundamental weights of $R$ ($R^*$).
\end{goperators}
The $g$ operators are listed in \eqref{gtypeA} and \eqref{variousg}. They can be expressed as explicit products of quantum dilogarithms. In the case of $A_{N-1}^{(1)}$ the  operator $g^{-1}$ is a particular evaluation of the Baxter $Q$-operator of \cite{SS}.

Next, we use the duality in Koornwinder-Macdonald theory, 
which relates the eigenvalue equation of the Koornwinder-Macdonald operators for $\g$ to the Pieri rules for the dual $\g^*$, 
where all cases are self-dual except $B_N^{(1)*}=C_N^{(1)}$ and $C_N^{(1)*}=B_N^{(1)}$. 
In the $q$-Whittaker limit, the Pieri operators are interpreted as (relativistic) $q$-difference Toda Hamiltonians, and the
polynomials $\Pi_\lambda^{(\g)}(x)$ as $q$-Whittaker functions. 
The First Pieri operators for $\Pi_\lambda^{(\g)}(x)$ are listed in \eqref{typeAHamiltonian} and \eqref{Hamil}. They express the effect of multiplying $\Pi_\lambda^{(\g)}(x)$ with the first elementary symmetric function in $x$, $e_1(x)$ in type $A_{N-1}^{(1)}$ or $\hat e_1(x)=\sum_{i=1}^N( x_i + x_i^{-1})$ for all other $\g$. The $N$ commuting Pieri operators $H_a^{(\g)}(\Lambda)$, which can be thought of as $q$-difference Toda Hamiltonians, commute with the time-translation operator $g^{(\g)}(\Lambda)$. Thus, we have 
\newtheorem*{integrability}{Corollaries \ref{identiftoda}, \ref{integrability}}
\begin{integrability}
The Pieri operators $\{H_a^{(\g)}(\Lambda), a\in[1,N]\}$ are algebraically independent conserved quantities of the opposite quantum Q-system for type $\g$.
\end{integrability}
The origin of the operators $g^{(\g)}(\Lambda)$ is clarified by the following theorem:
\newtheorem*{lemtwo}{Theorems \ref{fourierduathm}, \ref{lemtwo}}
\begin{lemtwo}
The operators $g^{(\g)}(\Lambda)$ are the Fourier transforms of Cherednik's Gaussian operator $\gamma(x)^{t_1}$, with $\gamma(x)$ as in Equation \eqref{gamma} the functional representation of the $SL(2,\Z)$ generator $\tau_+$:
$$
\gamma(x)^{t_1} \Pi_\lambda(x) = g^{(\g)}(\Lambda) \Pi_\lambda(x), \qquad t_1=\left\{ \begin{matrix} 2 &  {\rm if}\  \g=A_{2N}^{(2)}, C_N^{(1)},\\ 1 &\ {\rm otherwise .}
\end{matrix}\right.
$$
\end{lemtwo}

Using these properties, and a uniqueness argument involving series solutions of the first Pieri rule, we finally prove
\newtheorem*{mainthm}{Theorems \ref{mainthm}, \ref{foutranthm}}
\begin{mainthm}
For all $\g$, 
\begin{equation*}
D_{a;n}^{(\g)}(x)\,\Pi_\lambda^{(\g)}(x)=\barD_{a;n}^{(\g)}(\lL)\,\Pi_\lambda^{(\g)}(x),\qquad (n\in\Z)
\end{equation*}
valid for any $\g$-partition $\lambda$ and any root label $a\in [1,N]$.
\end{mainthm}
In other words, the operators $\barD_{a;n}^{(\g)}(\lL)$ are the Fourier transforms of the Whittaker limits of Macdonald operators
$D_{a;n}^{(\g)}(x)$.
This completes the proof of the main claim of this paper (Theorems \ref{Atheorem}, \ref{Qsysconj} and \ref{raiseconj}), 
and the conjectures of \cite{DFKconj}, that the $q$-difference operators $D_{a,n}^{(\g)}(x;q)$ satisfy the quantum Q-system of type $\g$.

We consider duality in terms of the universal solutions of the eigenvalue and Pieri equations, $P(x;s)$ and $Q(s;x)$. These are continuations of the Koornwinder or Macdonald polynomials to arbitrary values of the parameters $\lambda$, with $s=q^\lambda t^\rho$ and $x=q^\mu t^{\rho^*}$ with $\rho$ a function of $(a,b,c,d)$ as in Section \ref{secduabcd}. The series have the form $P(x;s)=q^{\lambda\cdot\mu}\,\sum_{\beta\in Q_+} c_\beta(s) \, x^{-\beta}$ and
$Q(s;x)=q^{\lambda\cdot\mu}\,\sum_{\beta\in Q_+^*} \bar c_\beta(x) \, s^{-\beta}$, with normalizations $c_0(s)=\bar c_0(x)=1$, and $Q_+, Q_+^*$ are the positive root lattices of the root systems $R, R^*$. Under specialization of the parameters $(a,b,c,d)$, universal solutions of the Koornwinder equation specialize $\g$-Macdonald solutions, and whenenever $\lambda$ corresponds to any dominant integral $R$-weight, the series truncates to a polynomial. They also specialize to Chalykh's Baker-Akhiezer quasi-polynomials \cite{Chalykh} when parameters $a,b,c,d,t$ (resp. $t$) are specialized to negative (half-)integer (resp. integer) powers of $q$, while $\lambda$ remains arbitrary.


We establish the following relation between the series $P(x;s)$ and $Q(x;s)$ , and a simple subsequent reformulation of duality.

\newtheorem*{UKconj}{Theorem \ref{UKconj}, Corollary \ref{UGconj}, Theorem \ref{DUAthm}}
\begin{UKconj}
 The universal Koornwinder functions $Q^{(a,b,c,d)}(s;x)$ and $P^{(a,b,c,d)}(x;s)$  and their $\g$-specializations
are related via
$$
Q^{(a,b,c,d)}(s;x)=\frac{P^{(a,b,c,d)}(x;s)}{\Delta^{(a,b,c,d)}(x)},\ \ Q^{(\g)}(s;x)=\frac{P^{(\g)}(x;s)}{\Delta^{(\g)}(x)}
$$
with $\Delta^{(a,b,c,d)}$ as in \eqref{deltabcd}, and $\Delta^{(\g)}$ as in \eqref{deltaG}. Moreover the duality of Koornwinder and $\g$-Macdonald polynomials extends to universal solutions as follows
$$Q^{(a^*,b^*,c^*,d^*)}(x;s)=Q^{(a,b,c,d)}(s;x), \ \ Q^{(\g^*)}(x;s)=Q^{(\g)}(s;x) .
$$
\end{UKconj}

\subsection{Outline of the paper}

This paper is organized as follows. 

We first revisit the $A$ type in Section \ref{Atype} as an illustration of the concepts used for other types in the remainder of the paper. 
After recalling the definition of Macdonald operators and polynomials, and showing how eigenvalue equations relate to Pieri rules via duality, we discuss the $q$-Whittaker limit $t\to\infty$ and the quantum Q-system (Sects. \ref{sec:21}-\ref{sec:24}).
Our main character is the time translation operator, whose adjoint action allows to advance Macdonald operators $D_a(x)$ in discrete time, thus producing operators $D_{a;n}(x)$. In $x$ space it takes the form of the (scalar) Gaussian operator $\gamma(x)$ (Sect. \ref{sec:25}). To prove that the Macdonald operators obey the quantum Q-system we switch to $\lL$ space: in Sect. \ref{sec:26}, using the eigenvalue $\bar D_{a;0}(\lL)$ of the Macdonald operator $D_a(x)$ as initial data, we {\it define} the candidate Fourier duals 
$\bar D_{a;n}(\lL)$ of the translated Macdonald operators as solutions of the opposite quantum Q-system for this initial data. In Sect. \ref{sec:27} we derive the associated explicit time translation operator $g(\lL)$.
Finally the latter is shown to be the Fourier transform of the Gaussian $\gamma(x)$ 
in Sect. \ref{sec:29} leading to the proof of the main Theorem \ref{Atheorem} for type $A$ by identifying $D_{a;n}(x)$ with the Fourier transform of $\bar D_{a;n}(\lL)$.
A key ingredient is the commutation of $g(\lL)$ with the first Pieri operator $H_1(\lL)$, 
shown in Sect. \ref{sec:28}. This allows to identify the Pieri operators as the commuting conserved 
quantities of the quantum Q-system, also known as the relativistic Toda Hamiltonians of type $A$ 
(Sect. \ref{Identification}). Finally Section \ref{universAsec} is devoted to a reformulation of the duality
properties, Fourier transform, and proofs in terms of the universal solution of the Macdonald eigenvalue equation considered by \cite{ShiraishiNoumi}. 

The remainder of the paper focusses on the other types. In Section \ref{macdoG}, we define suitable families of commuting Macdonald operators, borrowing from various existing constructions. The first approach uses the known specialization scheme of the Koornwinder operator to Macdonald operators, completed by van Diejen and Rains into a set of commuting difference operators (Sects. \ref{sec:macdopol} and \ref{sec:macdop}, and details in Appendix \ref{appA}). The duality of Koornwinder polynomials descends to a duality between Macdonald operators and Pieri rules for dual types. Sect. \ref{sec:qwhittak} describes the $q$-Whittaker limit $t\to\infty$, and provides detailed definitions of the translated limiting Macdonald operators $D_{a;n}^{(\g)}$ for all types. Some new subtleties arise for non-$A$ types, in particular the distinction between 
long and short labels $a$ for which the time translation has to be defined separately (type $A$ only has long labels).
Section \ref{sec:proof} defines the quantum Q-systems for all types and presents the main results of this paper,
Theorems \ref{Qsysconj} and \ref{raiseconj} (Sect. \ref{sec:41}), and their proof (Sect. \ref{sec:42} complemented by Appendices \ref{appB} and \ref{appC}), along the same lines as for the case of type $A$. We define candidate Fourier transforms $\bar D_{a;n}^{(\g)}(\lL)$ of the translated Macdonald operators by use of the opposite $\g$-quantum Q-system, and construct the time translation operators  $g^{(\g)}(\lL)$ explicitly. The latter commute with the Toda Hamiltonians, identified as the conservation laws of the quantum Q-systems, thus allowing us to identify $g^{(\g)}(\lL)$ with the Fourier transform of the Gaussian operator $\gamma^{(\g)}(x)$. 
Finally in Sect. \ref{sec:43} we prove that the $\bar D_{a;n}^{(\g)}(\lL)$ are the Fourier transforms of the 
translated Macdonald operators $D_{a;n}^{(\g)}(x)$, from which the main results follow. 

Section \ref{sec:universal} is an extension of the constructions of Section \ref{universAsec}. 
In Sect. \ref{sec:51}
we introduce universal solutions $P(x,s)$ to the Koornwinder-Macdonald eigenvalue equations in the form of power series of  the variables $x^{-\al_i},s^{-\al_i^*}$ where $\al_i,\al_i^*$ are the simple roots of suitable lattices. $P(x,s)$ has the remarkable 
property that it reduces to Koornwinder-Macdonald polynomials upon specializing $s=q^\lambda t^\rho$ for $\lambda$ a
($\g$-) partition. In Sect. \ref{sec:52} we extend the duality of Koornwinder-Macdonald polynomials to some 
relation between $P(x,s)$ and its dual $P^*(s,x)$. Finally in Sect. \ref{sec:53} we show how this new formulation 
allows to rewrite the proofs of this paper more elegantly.

We finally gather some concluding remarks in Section \ref{sec:conclusion}. We comment on the three additional ``companion" quantum Q-systems obtained as a by-product of our study and the associated representation by $q$-difference operators, which were not part of the original setting (Sect. \ref{sec:threenew}). We discuss path models for the various universal functions of the paper in Sect. \ref{sec:paths}. We interpret the universal function results in terms of $q$-Whittaker functions in Sect. \ref{sec:whitak}.
Sect. \ref{sec:summary} summarizes our results and lists related open questions.
\medskip

\noindent{\bf Acknowledgments.} 
We thank I. Cherednik, G. Schrader, A. Shapiro, J. Shiraishi and C. Stroppel for useful discussions, and G. Barraquand for pointing out reference \cite{Rains}.
We are also thankful to the referee who raised a number of interesting issues and suggested further important references. 
This work was supported by the following grants: National Science Foundation grants DMS 18-02044 and DMS-1937241; 
NSF Grant No. 1440140 while the authors were in residence at the Mathematical Sciences Research Institute in Berkeley, California in 2021; Simons Foundation Fellowship grants 613580 and 617036 and 
Simons Foundation grants MP-TSM-00002262 and MP-TSM-00001941.
PDF is supported by the Morris and Gertrude Fine Endowment.
RK thanks the Institut Henri Poincar\'e and the Institut de Physique Th\'eorique-CEA Paris Saclay for their hospitality.

\section{Duality property in type $A^{(1)}_{N-1}$}\label{Atype}

The main result of \cite{DFK15} is a theorem which states that the $q$-Whittaker limit of the A-type Macdonald operators and their $\tau_+$-translates satisfy the quantum Q-system of type $A_{N-1}^{(1)}$. To prove this, we used the explicit form of these difference operators in $x$-space. In this section we will give a different method for proving the same theorem by using Macdonald's duality relation and the Fourier transform of these operators.

We use the notation for the roots and weights of $\gl_N$ as follows. Let $\{e_1,...,e_N\}$ be the standard basis of $\R^N$, then the simple roots are $\{\alpha_i=e_i-e_{i+1}; i=1,...,N-1\}$, the $\gl_N$ fundamental weights (including  $\omega_N$) are $\{\omega_i = e_1 + \cdots + e_i, i=1,...,N\}$, the positive roots are $R_+=\{e_i-e_j, 1\leq i<j\leq N\}$, and $\rho=\sum_{i=1}^N (N-i)e_i$. For $u,v\in \R^N$, we use the notation $u\cdot v$ for the standard scalar product. We consider $q\in\C^*$, $|q|<1$.

\subsection{Macdonald operators and polynomials}\label{sec:21} Let $x=(x_1,...,x_N)$ be formal variables, and define the $q$-difference operators
$\D_a(x; q,t)$  be the  $a$-th Macdonald $q$-difference operator in type $A^{(1)}_{N-1}$ \cite{macdo}: 
\begin{equation}\label{DA}
\D_a(x;q,t) = \sum_{\substack{ I\subset\{1,...,N\}\\ |I|=a}} \prod_{\substack{j\notin I\\ i\in I}}\left(
\frac{t x_i -x_j}{x_i-x_j} \right)\prod_{i\in I} \Gamma_i, \qquad a\in\{1,...,N\}, \quad \Gamma_i x_j = q^{\delta_{ij}} x_j \Gamma_i.
\end{equation}
These difference operators form a commuting family and preserve the space of symmetric polynomials in $x$. 
The common eigenvectors are the monic symmetric Macdonald polynomials $P_\lambda(x)$, where $\lambda=(\lambda_1\geq \lambda_2\geq \cdots \geq \lambda_N\geq 0)$ is an integer partition:
\begin{equation}\label{eigenvalue}
\D_a(x;q,t) P_\lambda(x) = t^{-{a\choose2}}e_a(s) P_\lambda(x)\qquad (a=1,2,...,N),
\end{equation}
where, $e_a(s)$ is the $a$th elementary symmetric function in $s=(s_1,...,s_N)$, with $s_i = t^{N-i} q^{\lambda_i}$. Macdonald polynomials form a basis for the space of symmetric polynomials of $x$.

\subsection{Duality}\label{Duality}
The Macdonald polynomials $P_\lambda(x)$ satisfy a remarkable duality property, which is a symmetry under the interchange of the variables $x$ and $s$. Under the specialization $x=t^\rho q^\mu$ with $\mu$ an integer partition,  Macdonald showed that \cite{macdo}
\begin{equation}\label{dualpoA}
\frac{P_\lambda(t^\rho q^\mu)}{P_\lambda(t^\rho)}=\frac{P_\mu(t^\rho q^\lambda)}{P_\mu(t^\rho)}
\end{equation}
for all integer partitions $\lambda,\mu$. 
\begin{remark} The duality property has a formal generalization to arbitrary values of $\mu,\lambda$, see the discussion on the universal function in Section \ref{universAsec}.
\end{remark}

The denominators in Equation \eqref{dualpoA} can be obtained as specializations of an infinite product expression.
Let
\begin{equation}
\Delta(x) = \prod_{n\geq 1} \prod_{\alpha\in R_+} \frac{1-q^n x^{-\alpha}}{1-t^{-1}q^n x^{-\alpha}}=
\prod_{\beta\in \hat R_{++}} \frac{1-x^{-\beta}}{1-t^{-1}x^{-\beta}},\label{DeltaA}
\end{equation}
where the second product extends over the set $\hat R_{++}$ of strictly positive affine roots $\beta=n\delta+\sum_i \beta_i\al_i$, $n\geq 1,\beta_i\in \Z_+$, and we use the shorthand notation $x^\beta=q^{-n} \prod_i (x_{i}/x_{i+1})^{\beta_i}$.
Then  \cite{macdo}\footnote{Our definition of the infinite product $\Delta$ differs slightly from
that of Macdonald's $\Delta^+$, but is better suited for taking the limit $t\to\infty$ below.}
\begin{equation}\label{normPA}
P_\lambda(t^\rho)=t^{\lambda\cdot\rho}\,\frac{\Delta(t^\rho)}{\Delta(t^\rho q^\lambda)}.
\end{equation}

Next we use Macdonald's evaluation homomorphism $u_\mu$ defined on functions of $x$ by $u_\mu(f(x))=f(q^\mu t^\rho)$. It maps the generators of quantum torus $\mathbb T_x$ to
$u_\mu(x_i)=q^{\mu_i}t^{N-i}$ and $u_\mu(\Gamma_i)=e^{\partial_{\mu_i}}$ which acts as the translation $\mu_i\mapsto \mu_i+1$ while leaving $\mu_j,j\neq i$ unchanged.
Applying $u_\mu$ for $\mu$ an integer partition on  the eigenvalue equation \eqref{eigenvalue}, and using duality \eqref{dualpoA} gives
\begin{eqnarray*}
e_a(t^\rho q^\lambda) P_\mu(t^\rho q^\lambda)&=& t^{{a\choose2}}\, P_\mu(t^\rho)\,u_\mu\left(\D_a(x;q,t)\right) \,P_\mu(t^\rho)^{-1} \,P_\mu(t^\rho q^\lambda)\\
&=& \left\{t^{{a\choose2}}\, t^{\rho\cdot\mu}\,\Delta(t^\rho q^\mu)^{-1}\, u_\mu\left(\D_a(x;q,t)\right) \, \Delta(t^\rho q^\mu)\, t^{-\rho\cdot\mu}\right\} \,
P_\mu(t^\rho q^\lambda).
\end{eqnarray*}
Let us finally interchange the partition labels $\lambda\leftrightarrow \mu$ in the above equation, and note moreover that since
the polynomial $P_\lambda(x)$ is entirely determined by its values at the discrete specializations $x=t^\rho q^\mu$ with $\mu$ taking values in integer partitions, we can conclude that $P_\lambda(x)$ satisfies the Pieri equations:
\begin{eqnarray}
\cH_a(\Lambda;q,t)\, P_\lambda(x)&=& e_a(x)\, P_\lambda(x) \qquad (a=1,2,...,N)\label{PieriA} ,
\end{eqnarray}
with 
\begin{eqnarray}
\cH_a(\Lambda;q,t)&:=&
t^{{a\choose 2}}t^{\rho\cdot \lambda} \Delta(s)^{-1} u_\lambda\left(\D_a(x;q,t)\right) \Delta(s) t^{-\rho\cdot\lambda}\label{PieriHA}.
\end{eqnarray} 
The Pieri rules for the Macdonald polynomials express the multiplication of $P_\lambda(x)$ by an elementary symmetric function $e_a(x)$ as a linear combination
of Macdonald polynomials with shifted partitions. 

\begin{remark}\label{leftrightrem}
Note that in \eqref{PieriHA} the evaluation $u_\lambda$ maps $x_i\to q^{\lambda_i}t^{N-i}=s_i$ and 
$\Gamma_i\to e^{\partial_{\lambda_i}}$. The same evaluation map sends the generators of the quantum torus $\mathbb T_\lL$ to respectively $\lL_i\to q^{\lambda_i}$ and $T_i\to e^{\partial_{\lambda_i}}$.
By a slight abuse of notation, we simply express the operator $\cH_a(\Lambda;q,t)$ in terms of the generators $\lL_i$ and $T_i$ 
and omit the mention of the evaluation map from now on. This amounts to evaluating the left action of the quantum torus $\mathbb T_\lL$ 
on functions of the variables $(\lL_1,...,\lL_N)$. There is also a right action of the evaluation of $\mathbb T_\lL$ on the basis of Macdonald polynomials $P_\lambda$
in which $\lL_i$ acts diagonally $P_\lambda \,\lL_i=P_\lambda\,q^{\lambda_i} $ and $T_i$ by a shift $P_\lambda \,T_i=P_{\lambda+e_i}$. We adopt a single notation using the quantum torus generators, to denote either left or right action.
\end{remark}

Explicitly the commuting  operators $H_a(\Lambda;q,t)$
acting on functions of $\Lambda$ or $s$ are
\begin{equation}
\cH_a(\Lambda;q,t) =
\sum_{I\subset [1,N]\atop |I|=a}
\prod_{i\in I,\ j\not \in I\atop j<i} \frac{t^{i-j-1}\lL_{j}-\lL_i}{t^{i-j}\lL_{j}-\lL_i}\ \frac{t^{i-j+1} \lL_j-q \lL_{i}}{t^{i-j}\lL_j-q\lL_{i}}\, 
\prod_{i\in I}T_i,\qquad T_i \Lambda_j = q^{\delta_{ij}} \Lambda_j T_i.
\end{equation}
They can be thought of as Macdonald Hamiltonians of type $A_{N-1}^{(1)}$: In the $q$-Whittaker limit, they become relativistic $q$-difference
Toda Hamiltonians (see Sect. \ref{Identification} below).

\subsection{The $q$-Whittaker limit}\label{sec:23}
By a slight abuse of languge, we call the limit $t\to\infty$ of the difference operators and eigenfunctions above the $q$-Whittaker limit (as opposed to the standard $t\to0$).
The $q$-Whittaker functions\footnote{This is a slight abuse of  language, as these are strictly speaking $q^{-1}$-Whittaker functions. $\Pi_\lambda(x)$ is interpreted as a (class 1) $q$-Whittaker function where $x$ is the representation index, and $\lambda$ the argument.} are the limits of the Macdonald polynomials \cite{cherednikqwhittaker}:
$$\Pi_\lambda(x) =  \lim_{t\to\infty} P_\lambda(x).$$
In this limit we use the $t$-independent variables $\Lambda_i=q^{\lambda_i}$ instead of $s_i$. The symmetry between the eigenvalue equation and the Pieri rules under exchange of $x$ and $s$ is lost in this limit. However, the limit of the difference equations still makes sense.
The $q$-Whittaker functions satisfy the difference equations
\begin{eqnarray}\label{DWhittaker}
D_a(x;q) \Pi_\lambda(x) &=& \Lambda^{\omega_a} \Pi_\lambda(x), \qquad \Lambda^{\omega_a} = \Lambda_1 \cdots \Lambda_a;
\label{eigenvalueW}\\
H_a(\Lambda;q) \Pi_\lambda(x) &=& e_a(x) \Pi_\lambda(x).\label{pieriW}
\end{eqnarray}
Here, the difference operators are
\begin{equation}\label{qDiffOp}
D_{a}(x;q) =\lim_{t\to\infty} t^{a(a-N)} \D_{a}(x;q,t) 
=
 \sum_{\substack{ I\subset\{1,...,N\}\\ |I|=a}} 
 \prod_{j\notin I\atop i\in I}
\frac{x_i}{x_i-x_j} \prod_{i\in I} \Gamma_i,%
\end{equation}
and the Pieri operators are
\begin{equation}\label{typeAHamiltonian}
H_a(\lL;q)=\lim_{t\to\infty} \cH_{a}(\lL;q,t)=\sum_{I\subset[1,N]\atop |I|=a}
\prod_{i\in I\atop i-1\notin I} \left(1-\frac{\Lambda_i}{\Lambda_{i-1}}\right) \, \prod_{i\in I}T_{i}
\end{equation}
with the convention that $\Lambda_0:=\infty$. Note the simplified form of these equations: 
the Pieri operators are polynomials in $\Lambda^{-\alpha_i}$ where $\alpha_i$ are the simple roots, rather than rational functions. Moreover,
the eigenvalues of operators $D_a(x;q)$ are monomials in $\Lambda$, because $e_a(s)$ is replaced by its leading term in $t$, $\Lambda^{\omega_a}$, where $\omega_a$ are the fundamental weights.

\subsection{Type $A^{(1)}_{N-1}$ quantum Q-system}\label{sec:24} 
The main result of \cite{DFK15} is that the set of all $\tau_+$-translations of the operators $D_a(x;q)$ satisfy the quantum Q-system equations\footnote{This was referred to as the type $A_{N-1}$ Q-system in e.g. \cite{DFK15}.} of type $A^{(1)}_{N-1}$. To define this system, let $\{ \Q_{a;k}: a\in\{1,...,N\},k\in\Z\}$ be invertible elements in some non-commutative algebra over $\C(q)$, subject to two types of relations: A $q$-commutation relation
\begin{equation}\label{Acommutation}
\Q_{a;k}\Q_{b;k+i} = q^{i\min(a,b)} \Q_{b;k+i}\Q_{a;k}, \qquad k\in\Z,\ a,b\in[1,N],\ |i|\leq 1,
\end{equation}
and a cluster exchange-type relation, which can be thought of as  a discrete time evolution in the  variable $k$:
\begin{equation}\label{AQsys}
q^a \Q_{a;k+1}\Q_{a;k-1}=\Q_{a;k}^2-\Q_{a-1;k}\Q_{a+1;k},\quad a\in[1,N],\ k\in\Z,
\end{equation}
subject to the boundary conditions $\Q_{0;k}=1, \Q_{N+1;k}=0$. The algebra generated by the set $\{\Q_{a;k}: a\in[1,N], k\in\Z\}$ and their inverses is a subalgebra in a certain quantum cluster algebra \cite{DFKnoncom}.

Define the generalized $q$-Whittaker difference operators, acting on the space of functions in the variables $\{x_1,...,x_N\}$ with coefficients in $\C(x_1,...,x_N,q)$:
\begin{equation}\label{qDiffOphigher}
D_{a;k}(x;q) = \sum_{\substack{ I\subset\{1,...,N\}\\ |I|=a}} \prod_{i\in I} x_i^k
 \prod_{j\notin I}\left(
\frac{x_i}{x_i-x_j} \right)\prod_{i\in I} \Gamma_i, \qquad a\in[1,N],\ k\in\Z.
\end{equation}
When $k=0$ these are just the $q$-Whittaker difference operators \eqref{DWhittaker}.
The main result of \cite{DFK15} is that these operators provide a functional representation of the quantum Q-system:
\begin{thm}[\cite{DFK15}] \label{Atheorem} The $q$-difference operators $D_{a;k}(x;q)$ of \eqref{qDiffOphigher} satisfy the quantum Q-system relations \eqref{Acommutation} and \eqref{AQsys}.
\end{thm}
The original proof of the theorem consisted of working directly with the difference operators $D_{a;k}(x;q)$.
Here, we present a simplification of the proof which uses the duality of Section \ref{Duality}. This method has the advantage that it can be generalized to other root systems.

\subsection{The action of $\tau_+$ on the difference operators}\label{sec:25}
The difference operators $D_a(x;q,t)$ are representations of elements of the spherical DAHA acting on the space of functions in $N$ variables. There is an action of $SL_2(\Z)$ on DAHA, and the $SL_2(\Z)$-generator $\tau_+$ acts on the functional representation by the adjoint action of Cherednik's Gaussian operator \cite{Cheredbook}
\begin{equation}\label{gamma}
\gamma(x)=e^{\sum_{i=1}^N \frac{\log(x_i)^2}{2\log(q)}}.
\end{equation}
In the $q$-Whittaker limit, the adjoint action of $\gamma(x)$ is well-defined. In particular, the difference operators \eqref{qDiffOp} are the $\tau_+$-translates of $D_a(x;q)$:
\begin{lemma}[\cite{DFKqt}]
\begin{equation}
D_{a;k}(x;q) = q^{-a k/2} \gamma(x)^{-k} D_{a}(x;q) \gamma(x)^k.
\end{equation}
\end{lemma}
The proof follows from
$$
\gamma(x)^{-1} \Gamma_i \gamma(x) = q^{\frac12} x_i \Gamma_i.
$$

\subsection{Fourier transform}\label{sec:26}

The $q$-Whittaker functions $\{\Pi_\lambda(x)\}$, with $\lambda$ ranging over integer partitions, form a complete basis of the space of symmetric polynomials. Suppose that a set of difference operators $\{D_a(x)\}$ in the variables of $\mathbb T_x$ satisfies
$$
D_a(x)\, \Pi_\lambda(x) = \bar D_a(\Lambda) \, \Pi_\lambda(x)
$$
for all $\lambda$, where the difference operators $\bar D_a(\Lambda)$ act in the variables of $\mathbb T_{\Lambda=q^\lambda}$. If the operators $\{\bar D_a(\lambda)\}$ satisfy certain relations $\mathcal R$, then the set of operators $\{D_a(x)\}$ satisfy the relations $\mathcal R^{\rm op}$ with the opposite multiplication. The operators $\bar D_a(\Lambda)$ are the ``Fourier transforms" of the operators $D_a(x)$.

The strategy is to define the operators $\bar D_{a;k}(\Lambda;q)$, starting with initial data $\{\bar D_{a;0},\bar D_{a;1}: a\in[1,N]\}$, by using the quantum Q-system relations with the opposite multiplication. We will then show that the corresponding Fourier-dual operators $D_{a;k}(x;q)$ are the difference operators \eqref{qDiffOphigher}, which therefore satisfy the quantum Q-system. The simplification of the proof in the Fourier transformed picture is due to the simple form of the initial data in $\mathbb T_\Lambda$, which allows us to compute the $\tau_+$ action on the Fourier transformed operators directly.

\subsubsection{Initial data}
To find appropriate initial data for the opposite quantum Q-system, we start with the eigenvalue equation \eqref{eigenvalueW}, which we write as
\begin{equation}\label{initdatazero}
D_{a;0}(x;q) \Pi_\lambda(x) = \bar D_{a;0}(\Lambda;q)\Pi_\lambda(x) ,\qquad \bar D_{a;0} (\Lambda;q) = \Lambda^{\omega_a}, \qquad a\in [1,N].
\end{equation}

For the operators $\bar D_{a;k}(\Lambda;q)\in \C_q(\Lambda)[T]$ to be well-defined from the opposite quantum Q-system for all $k$, we need another set of initial data, $\{\bar D_{a;1}(\Lambda;q),\ a\in[1,N]\}$, such that (we drop the arguments $\lL,q$ from now on):
\begin{equation}\label{oppcr}
\bar D_{a;0}\, \bar D_{b;1} = q^{-\min(a,b)}\, \bar D_{b;1}\,\bar D_{a;0},\qquad [\bar D_{a;1}, \bar D_{b;1}]=0,\qquad a,b\in[1,N].
\end{equation}
These are $2N$ relations on the quantum torus $\mathbb T_\Lambda$ of dimension $2N$, so they determine 
$\{\bar D_{a;1}\}$ up to scalar multiple, which commutes with $\mathbb T_\Lambda$, i.e. an element in $\C(q)$. We choose this constant to be 1:
\begin{equation}\label{DaonetypeA}
\bar D_{a;1} = \Lambda^{\omega_a} T^{\omega_a} = \Lambda_1\cdots \Lambda_a T_1 \cdots T_a.
\end{equation}

\subsubsection{The set $\{\bar D_{a;k}\}$}
Given the set of $2N$ initial data $\{\bar D_{a;k}, k\in\{0,1\}, a\in[1,N]\}$, which form an alternate basis for $\mathbb T_\Lambda$,  all $\bar D_{a;k}$ are uniquely defined by the requirement that they satisfy the opposite quantum Q-system relation:
\begin{equation}\label{oppositeQ}
q^{a}\,\bar {D}_{a;k-1}\,\bar D_{a;k+1}= \bar {D}^2_{a;k} - \bar D_{a+1;k}\,\bar D_{a-1;k} , \qquad \bar D_{0;k}=1,\ \bar D_{N+1;k}=0.
\end{equation}

\begin{remark}
The quantum Q-system variables are cluster variables in a quantum cluster algebra, and therefore, due to the Laurent property of these algebras, all $\bar D_{a;k}$ are in fact Laurent polynomials in the initial data \eqref{initdatazero} and \eqref{DaonetypeA}. As the latter are monomial in the
variables of the quantum torus $\mathbb T_\Lambda$, so are the $\bar D_{a;k}\in \C_q[\Lambda^{\pm1}, T^{\pm1}]$.
\end{remark}

The main theorem to be proven is
\begin{thm}\label{mainthm}
$$ \bar D_{a;k}(\Lambda;q)\, \Pi_\lambda(x) = D_{a;k}(x;q)\, \Pi_\lambda(x)$$
for all $a\in[1,N]$ and $k\in\Z$, with $D_{a;k}(x;q)$ the difference operators of Equation \eqref{qDiffOphigher}
\end{thm}
Since $\bar D_{a;k}(\Lambda;q)$ satisfy the opposite quantum Q-system relations, then the operators $D_{a;k}(x;q)$ satisfy the quantum Q-system. Theorem \ref{mainthm} therefore implies Theorem \ref{Atheorem}.

The proof of this theorem performed over the next several subsections as follows. First, in \ref{sec:27} we compute a ``time translation" operator $g(\Lambda)$, whose adjoint action on $\bar D_{a;k}(\Lambda;q)$ gives $\bar D_{a;k+1}(\Lambda;q)$ in accordance with the opposite Q-system recursion evolution. We then show in \ref{sec:28} that this time-translation commutes with the Pieri operator or Hamiltonian $H_1(\Lambda;q)$. In \ref{sec:unique} we prove the uniqueness of the solution of the Pieri equation $H_1(\Lambda,q)\,f=\hat e(x)\, f$. This implies, as shown in \ref{sec:29}, that $g(\Lambda)$ is the Fourier transform of the Gaussian $\gamma(x)$. The proof of Theorem \ref{mainthm}  follows (Section \ref{sec:theproof}) from a simple argument using the Fourier transform.

\subsection{The action of $\tau_+$ on $\bar D_{a;k}(\Lambda;q)$}\label{sec:27}
The key to proving theorem \ref{mainthm} is to find an element $g(\Lambda)$ in the completion of the space of rational functions in $\mathbb T_\Lambda$ which realizes the action of $\gamma(x)$ in the Fourier transformed picture\footnote{The element $g(\lL)$ is referred to as the ``Dehn twist" generator in the geometric formulation of Ref. \cite{SS}, which uses a different but related definition of the Fourier transform, under the name of Whittaker transform.}:
\begin{thm}\label{gfunctionA}
There is a unique, up to scalar multiple, function $g\equiv g(\Lambda)$ acting on the variables of $\mathbb T_\Lambda$, such that
\begin{equation}\label{gevolution}
\bar D_{a;k} = q^{-a k/2} \,g^k \,\bar D_{a;0} \,g^{-k}
\end{equation}
is a solution of the opposite quantum Q-system for all $a,k$.
\end{thm}

\begin{proof} The element $g$ is determined by the commutation relation \eqref{oppcr} and the subset of Q-system relations
\eqref{oppositeQ} corresponding to $k=1$. Indeed, assuming such  
an element $g$ exists, we first note that 
$g^m[{\bar D}_{a,0},{\bar D}_{b,0}]g^{-m}=q^{(a+b)m/2}[{\bar D}_{a,m},{\bar D}_{b,m}]=0$
for all $m\in \Z$. 
Similarly, conjugating \eqref{oppcr} and \eqref{oppositeQ} for $k=1$ with $g^m$ gives 
\begin{eqnarray*}&&\!\!\!\!\!\!g^m ({\bar D}_{a,0}\, {\bar D}_{b,1}-q^{{\rm Min}(a,b)} {\bar D}_{b,1}\,{\bar D}_{a,0})g^{-m}=
q^{(a+b)m/2}({\bar D}_{a,m}\, {\bar D}_{a,m+1}-q^{{\rm Min}(a,b)} {\bar D}_{a,m+1}\,{\bar D}_{a,m})=0\\
&&\!\!\!\!\!\!\!\!g^m ( q^a {\bar D}_{a,-1}{\bar D}_{a,1}\!\!-\!\!{\bar D}_{a,0}^2\!\!+\!\!{\bar D}_{a-1,0}{\bar D}_{a+1,0})g^{-m}=
q^{am} (q^a {\bar D}_{a,m-1}{\bar D}_{a,m+1}\!\!-\!\!{\bar D}_{a,m}^2\!\!+\!\!{\bar D}_{a-1,m}{\bar D}_{a+1,m})=0
\end{eqnarray*}
We therefore obtain the opposite of \eqref{Acommutation} and \eqref{AQsys} for all $m\in \Z$.

We now prove the existence of $g$ by construction.
Assume $g=g_T g_\Lambda$, where $g_T$ commutes with all $T_a$ and $g_\Lambda$ commutes with all $\Lambda_a$. Up to a scalar multiple, $g_T$ is  determined by the $N$ equations \eqref{gevolution} with $k=1$:
$$
q^{-a/2} \,g\,\bar  D_{a;0}\,g^{-1} = q^{-a/2} \,g_T \,\Lambda^{\omega_a} \,g_T^{-1} = \bar D_{a;1} =  \Lambda^{\omega_a} \,T^{\omega_a},\qquad a\in[1,N],
$$
because $g_\Lambda$ commutes with $\bar D_{a;0}=\Lambda^{\omega_a}$ by assumption.
This has the solution
\begin{equation}\label{gtvalue}
g_T =\prod_{i=a}^N e^{\frac{(\log T_a)^2}{2\log q}},
\end{equation}
where we used the identity
$$
e^{\frac{(\log T_a)^2}{2\log q}} \,\Lambda_a \,e^{-\frac{(\log T_a)^2}{2\log q}}  = \Lambda_a \,e^{\frac{(\log qT_a)^2-(\log T_a)^2}{2\log q}} = q^{1/2} \,\Lambda_a\, T_a.
$$

To find $g_\Lambda$, we use the Q-system \eqref{oppositeQ} with $k=1$:
\begin{equation}\label{Dtwoa}
\bar D_{a;2}= \Lambda^{\omega_a}\,T^{2\omega_a}\,(1-q \,\Lambda^{-\alpha_a}\,T^{-\alpha_a})  ,  \ a\in[1,N-1],\quad \bar D_{N;2}=\Lambda^{\omega_N}\,T^{2\omega_N}.
\end{equation}

To simplify the equations below, we set $\Lambda_{N+1}=0$ and $\Lambda_0=\infty$, i.e. we define $\alpha_0,\alpha_N$ by $\Lambda^{-\alpha_0}=\Lambda^{-\alpha_N}=0$.
The function $g_\Lambda$ is defined from
\begin{equation}\label{dtwotodone}
\bar D_{a;1}=q^{a/2}\,g^{-1}\,\bar  D_{a;2} \,g = q^{a/2} \,g_\Lambda^{-1} \,g_T^{-1} \,\bar D_{a;2}\,g_T\, g_\Lambda =  \Lambda^{\omega_a} \,T^{\omega_a}.
\end{equation}
Acting on \eqref{Dtwoa} by the adjoint action of $g_T^{-1}$, we have
$$
q^{a/2} \,g_T^{-1}\,\bar  D_{a;2} \,g_T = \Lambda^{\omega_a}\, T^{\omega_a} \,(1-\Lambda^{-\alpha_a}).
$$
Using 
$$
\prod_{n=0}^\infty (1-q^n \,\Lambda^{-\alpha_a}) \,T^{\omega_b} = T^{\omega_b} \,\prod_{n=0}^\infty (1-q^{n+\delta_{ab}} \,\Lambda^{-\alpha_a}) = T^{\omega_b} \, (1-\Lambda^{-\alpha_a})^{-1}\prod_{n=0}^\infty (1-q^n \,\Lambda^{-\alpha_a})
$$
we can choose
\begin{equation}\label{glambdavalue}
g_\Lambda = \prod_{a=1}^{N-1} (\Lambda^{-\alpha_a};q)_\infty^{-1},\qquad (a;b)_\infty = \prod_{n=0}^{\infty} (1-b^n a).
\end{equation}

We conclude that the function
\begin{equation}\label{gtypeA}
g=\left(\prod_{a=1}^N e^{\frac{(\log T_a)^2}{2\log q}}\right)\  \prod_{a=1}^{N-1}\left( \frac{\Lambda_{a+1}}{\Lambda_a};q\right)_\infty^{-1}
\end{equation}
satisfies the conditions of Theorem \ref{gfunctionA}.
\end{proof}

\subsection{Commutation with the first Pieri operator}\label{sec:28}
In general,  the translation operator $g$ commutes with each of the Pieri operators $H_a(\Lambda;q)$. It is sufficient for our purposes to show that it commutes with the first, $H_1(\Lambda;q)$:
\begin{thm}\label{commutationtypeA}
The function $g(\Lambda)$ of Theorem \ref{gfunctionA} commutes with the Pieri operator 
$H_1(\Lambda;q)$ of Equation \eqref{typeAHamiltonian}.
\end{thm}
\begin{proof}
From \eqref{typeAHamiltonian},
\begin{eqnarray*}
H_1(\Lambda;q) &=& \sum_{a=1}^{N} (1-\Lambda^{-\alpha_{a-1}}) T_{a} = \sum_{a=1}^N T_a - \sum_{a=1}^{N-1} \frac{\Lambda_{a+1}}{\Lambda_{a}} T_{a+1}.
\end{eqnarray*}
For all $a\in[1,N]$,
$$ g (1-\Lambda^{-\alpha_{a-1}})T_a g^{-1}= g_T(1-q^{-1}\Lambda^{-\alpha_{a}}) T_a g_T^{-1}
= (T_a - \Lambda^{-\alpha_a} T_{a+1}) ,
$$
again with the convention that $\Lambda^{-\alpha_0}=\Lambda^{-\alpha_N}=0$.
Summing over all $a=1,2,...,N$ results in $g H_1(\Lambda;q) g^{-1}=H_1(\Lambda;q)$.
\end{proof}

\subsection{Uniqueness of the Pieri solution}\label{sec:unique}
To prove the result of next section, we resort to a uniqueness argument regarding the solutions to the first Pieri equation
\eqref{firsthamil}.

Recall that $\Pi_\lambda(x)$ is eigenfunction
of the Macdonald operator: $D_1\,\Pi_\lambda= \lL_1\, \Pi_\lambda$. Writing $\Pi_\lambda(x)=x^\lambda p_\lambda(x)$,
and $D_1=\sum_{i} \phi_{i}(x) \,\Gamma_i$, the eigenvalue equation turns into
\begin{equation}\label{continuable}
(1-\sum_{i} \phi_{i}(x) \,\frac{\lL_i}{\lL_1} \, \Gamma_i)\,p_\lambda(x)=0 .
\end{equation}
This equation makes it easy to analytically continue 
$p_\lambda(x)$ to a function $\hat p_\lambda(x)$ with $\lambda\in \C^N$,
as the dependence on $\lL$ is explicit (in fact polynomial of the variables $\lL^{-\al_i}$, $\al_i$ the simple roots of $A_{N-1}$).  
Writing $\hat p_\lambda(x)$
as a series of the variables $x^{-\al_i}$: $\hat p_\lambda(x)=\sum_{\beta\in Q_+}
c_\beta(\lL) x^{-\beta}$, $Q_+$ the positive cone of the root lattice of $A_{N-1}$, 
\eqref{continuable} turns into a linear triangular system for the coefficients $c_\beta(\lL)$,
which are uniquely determined (with  $c_0(\lL)=1$), and rational functions of the $\lL^{-\al_i}$. 
Moreover, specializing $\lambda$ to an integer partition in $\hat p_\lambda(x)$ recovers $p_\lambda(x)$. 
Expanding $c_\beta(\lL)=\sum_{\delta\in Q_+} c_{\beta,\delta} \lL^{-\delta}$ allows to view $\Pi_\lambda(x)$ as the specialization of a series in $\lL$: 
$$x^\lambda \hat p_\lambda(x)=x^\lambda \sum_{\delta\in Q_+} \hat c_\delta(x)\, \lL^{-\delta},\ \ 
\hat c_\delta(x)=\sum_{\beta\in Q_+} c_{\beta,\delta} x^{-\beta} .$$

\begin{lemma}\label{uniquelem}
Assume we have a  (non necessarily polynomial) solution $\Theta_\lambda(x)$ of the first Pieri rule \eqref{firsthamil}, 
which admits a series expansion of the form $\Theta_\lambda(x)=x^\lambda \sum_{\beta\in Q_+} \tau_\beta(x)\, \lL^{-\beta}$ for $\lambda\in \C^N$.
Then we have, for $\lambda$ an integer partition:
$$\Pi_\lambda(x)=\frac{\hat c_0(x)}{\tau_0(x)} \Theta_\lambda(x), $$
where $\hat c_0(x)$ is the leading coefficient in the series $\hat p_\lambda(x)$ that specializes to 
$\Pi_\lambda(x)$ for integer partitions $\lambda$.
\end{lemma}
\begin{proof}
Write $\Theta_\lambda(x)=x^\lambda \theta_\lambda(x)$. By use of \eqref{typeAHamiltonian} for $a=1$, 
the Pieri equation turns into
$$\left(\sum_{i=1}^N x_i \left\{ (1-\lL^{-\al_i}) T_i -1\right\}\right) \theta_\lambda(x)=0 .$$
This is a linear triangular system for the coefficients $\tau_\beta(x)$, which are uniquely determined for 
$\beta\neq 0$, proportional to $\tau_0(x)$. The same holds for the coefficients $\hat c_\beta(x)$ 
of $\hat p_\lambda(x)$ in terms of $\hat c_0(x)$. We deduce that
$x^\lambda \, \hat p_\lambda(x) =\frac{\hat c_0(x)}{\theta_0(x)} \Theta_\lambda(x)$, and the Lemma follows
by specialization.
\end{proof}

\begin{remark}
The argument used in this section prefigures the reformulation in terms of universal solutions performed in Sect. \ref{universAsec} below, and extends to all types.
\end{remark}

\subsection{Fourier duality of $g$ and $\gamma$}\label{sec:29}
The final piece of information we need to prove Theorem \ref{mainthm} is that $g$ acting on the $q$-Whittaker functions is equal to the 
Gaussian $\gamma(x)$ acting on the same functions:
\begin{thm}\label{fourierduathm}
\begin{equation}
g(\Lambda)\, \Pi_\lambda(x) = \gamma(x) \, \Pi_\lambda(x).
\end{equation}
\end{thm}
\begin{proof}
Acting with $g(\Lambda)$ on both sides of the first Pieri rule
\begin{equation}\label{firsthamil}
H_1(\Lambda;q)\, \Pi_\lambda(x) = e_1(x)\, \Pi_\lambda(x)
\end{equation}
and noting that $g$ and $\gamma$ commute with both $H_1(\Lambda;q)$ and $e_1(x)$, we see 
that $\gamma^{-1} \,g\, \Pi_\lambda(x)$ obeys the {\it same} first Pieri rule. 
We use the uniqueness argument of the previous section for the solutions of \eqref{firsthamil}. 
Applying Lemma \ref{uniquelem} to  $\Theta_\lambda=\gamma^{-1} \,g\, x^\lambda\,\hat p_\lambda(x)$, 
the latter must be proportional to $\Pi_\lambda(x)$ when specialized to an integer partition $\lambda$.
To determine the proportionality constant, we compute the leading coefficient of the series expansion $x^{-\lambda}\,\gamma^{-1} \,g\,x^\lambda\,\hat p_\lambda(x)$. Since $g_\Lambda$ is a power series in $\{\Lambda^{-\alpha_i}\}$ with leading term 1, the leading term is determined by the action of $g_T$ on $x^\lambda\,\hat c_0(x)$. 

We claim that
$$
g_T x^\lambda = \gamma\, x^{\lambda} {g}'_T,
$$
where $g'_T$ acts only by shifting $\Lambda$ by powers of $q$, leaving the leading term $\hat c_0(x)$ unchanged. 

To see this, let $a,b$ be elements in an algebra such that $[a,[a,b]]=0$ and $[b,[a,b]]=c$ commutes with both $a$ and $b$. 
Then the Campbell-Hausdorff formula
implies:
\begin{equation}\label{CH}
e^a\, e^b \,e^{-a} \,e^{-b} = e^{a+b+\frac12[a,b] - \frac{1}{12}[b,[a,b]]}e^{-a-b+\frac12[a,b]+ \frac{1}{12}[b,[a,b]]} 
= e^{[a,b] +\frac{c}{2}}.
\end{equation}
Choosing $e^a=g_T$ and $e^b = x^\lambda$, i.e. $a=\sum_i\frac{\partial_{\lambda_i}^2}{2\log q }$ and $b=\sum_i\lambda_i\,\log x_i$ leads to 
$[a,b]
= \sum_i \frac{\log x_i\,\partial_{\lambda_i}}{\log q}$ and $e^{\frac12 c} = \gamma(x)^{-1}$, which commutes with $e^{[a,b]}$. 
We finally write
\begin{eqnarray*} e^a \,e^b= g_T\, x^\lambda=e^{\frac{c}{2}}\, e^{[a,b]} \,e^b\,e^a
=\gamma^{-1} e^{\sum_i \frac{\log x_i\,\partial_{\lambda_i}}{\log q}} e^{\sum_i \lambda_i \log x_i}\, g_T
=\gamma \, x^\lambda\, \prod_i T_i^{\frac{\log x_i}{\log q}} g_T=\gamma \, x^\lambda\,g_T'
\end{eqnarray*}
where the last step uses again \eqref{CH}, with $a=\sum_i \frac{\log x_i\,\partial_{\lambda_i}}{\log q}$, $b=\sum_i \lambda_i \log x_i$
hence $e^{[a,b]}=\gamma^2$ and $c=0$.

We conclude that $\gamma^{-1} \,g\, \Pi_\lambda(x)=\Pi_\lambda(x)$, and the theorem follows.
\end{proof}

\begin{remark} \label{ajim}
The relation $g(\Lambda)\, \Pi_\lambda(x) = \gamma(x) \, \Pi_\lambda(x)$ is equivalent to the recursion relation for 
the quantities denoted $J_\beta^\mu$ in \cite{FJMMfermio} (Theorem 3.1), which we denote as $J_\beta(x=q^\mu)$ below. These can be identified as the coefficients in the formal expansion of the series
$\widetilde \Pi_\lambda(x):=g_\Lambda\,\Pi_\lambda(x)=x^\lambda \sum_{\beta \in Q_+}  J_\beta(x)\Lambda^{-\beta}$.
Theorem \ref{fourierduathm} says that $g_Tg_\Lambda \Pi_\lambda(x)=\gamma(x)\Pi_\lambda(x)$, which implies
$(g_\Lambda g_T) \widetilde \Pi_\lambda(x)=\gamma(x)\, \widetilde \Pi_\lambda(x)$. As a recursion relation for coefficients, this means
$$J_\beta^\mu=\sum_\delta \frac{\gamma(q^\delta)}{(q,q)_{\beta-\delta}} x^{-\delta}J_\delta^\mu$$
with the notation $(q,q)_\al=\prod_i (q,q)_{\al_i}$,  
where we have used the series expansion $g_\Lambda=\sum_{\alpha\in Q_+} \frac{\Lambda^{-\alpha}}{(q;q)_\alpha}$,
as well as the relation $\gamma(x)^{-1} x^{-\lambda}g_T x^\lambda \Lambda^{-\delta}=\gamma(q^{\delta})(\Lambda x)^{-\delta}$.
A similar connection holds for certain other root systems, see Remark \ref{newajim}.
\end{remark}

\begin{cor}\label{identiftoda}
The (higher) $q$-difference Toda Hamiltonians $H_a(\lL)$ are algebraically independent conserved quantities of the 
opposite quantum Q-system \eqref{oppositeQ}.
\end{cor}
\begin{proof}
Multiply the Pieri equation $e_a(x)\Pi_\lambda(x)=H_a(\lL)\,\Pi_\lambda(x)$ on the left by $g \gamma^{-1}$, and use
Theorem \ref{fourierduathm} to rewrite:
$$g\gamma^{-1} \,e_a(x)\Pi_\lambda(x)=e_a(x)\Pi_\lambda(x)=g\,H_a(\lL)\,g^{-1}\, \Pi_\lambda(x)=H_a(\lL)\, \Pi_\lambda(x),$$
hence $g$ commutes with all $H_a(\lL)$, $a=1,2,...,N$. The latter are Laurent polynomials of the elements of the quantum torus 
$\mathbb T_\lL$, hence of the initial data $\bar D_{i;0},\bar D_{i;1}$ as well, which are invariant under any time translation
$(\bar D_{i;0},\bar D_{i;1})\mapsto (\bar D_{i;n},\bar D_{i;n+1})$, $n\in \Z$. These are independent conserved quantities 
of the opposite quantum Q-system that governs the $\bar D_{a;n}$, as any dependence between $H_a(\lL)$ would imply 
a dependence between the $e_a(x)$ by inverse Fourier transform.
\end{proof}

\subsection{Proof of Theorem \ref{mainthm}}\label{sec:theproof}
The proof of the main Theorem \ref{mainthm} now follows.
\begin{proof}
Acting on $D_{a;0}(x;q)\Pi_\lambda(x)$ by $\gamma(x)^{-k} \,g^k(\Lambda)$, we have
$$
\gamma^{-k}\,g^k D_{a;0}(x;q) \,\Pi_\lambda(x) = \gamma^{-k} \,D_{a;0}(x;q)\,g^k \,\Pi_\lambda(x) 
=  \gamma^{-k}\,D_{a;0}(x;q)\,\gamma^{k}\,\Pi_\lambda(x)= q^{a k/2} \,D_{a;k}(x;q) \,\Pi_\lambda(x).
$$
Using the eigenvalue equation \eqref{eigenvalueW}, this is equal to 
$$\gamma^{-k}\,g^k \,\bar D_{a;0}\,\Pi_\lambda(x) = g^k \,\bar D_{a;0}\,\gamma^{-k}\,\Pi_\lambda(x)  
= g^k \, \bar D_{a;0}\,g^{-k}\,\Pi_\lambda(x)= q^{a k/2} \,\bar D_{a;k}\,\Pi_\lambda(x),$$
which is the statement of Theorem \ref{mainthm}.
\end{proof}

\subsection{Pieri operators, conserved quantities and Toda Hamiltonians}
\label{Identification}
Let us compare the result of  Corollary \ref{identiftoda} with the conserved quantities of the (opposite) quantum
Q-system of \cite{DFKnoncom,DFK10,DFK15}. Those references use a different normalization of the Q-system variables, resulting in the system
\begin{eqnarray}
\bar Q_{a;n}\, \bar Q_{b;n+1}&=& v^{-\Lambda_{a,b}} \,\bar Q_{b;n+1}\, \bar Q_{a;n}\nonumber \\
v^{\Lambda_{a,a}} \,\bar Q_{a;n-1}\, \bar Q_{a;n+1}&=& (\bar Q_{a;n})^2-\bar Q_{a-1;n}\, \bar Q_{a+1;n}\label{otherquant}
\end{eqnarray}
with $\Lambda_{a,b}=\min(a,b)\big( N-\max(a,b)\big)$.
The precise relation with the operators $\bar D_{a;k}(\Lambda;q)$ is
\begin{equation}\label{dictio}
\bar Q_{a;n}=q^{-\frac{n \al^2}{2N}}\,(\bar D_{N;0})^{-\frac{a}{N}}\,\bar D_{a;n}
\end{equation}
with $v=q^{1/N}$. The conserved quantities $\bar C_m$ of the quantum system \eqref{otherquant}
are expressed in terms of the initial cluster $\{ \bar Q_{a;0},\bar Q_{a;1}\}$
as hard particle partition functions (i.e. generating polynomials of independent sets of vertices) on a graph (Figure 3 of \cite{DFK10})
with ordered vertices labeled $1,2,...,2N-1$, with a weight $\bar y_i$ per occupied vertex labeled $i$, where: 
\begin{eqnarray*}
\bar y_{2a-1}&=&\bar Q_{a-1;0}\,(\bar Q_{a;0})^{-1}\, (\bar Q_{a-1;1})^{-1}\, \bar Q_{a;1} \qquad (a=1,2,...,N),\\
\bar y_{2a}&=&-\bar Q_{a-1;0}\,(\bar Q_{a;0})^{-1}\, (\bar Q_{a;1})^{-1}\, \bar Q_{a+1;1} \qquad (a=1,2,...,N-1).
\end{eqnarray*}
Using \eqref{dictio},
$$\bar y_{2a-1}=v^{-\frac{1}{2}}\,T_a,\qquad \bar y_{2a}=-v^{-\frac{1}{2}} \frac{\lL_{a+1}}{\lL_a}\, T_{a+1}, $$
hence the resulting conserved quantities are related to the Hamiltonians \eqref{typeAHamiltonian} by $H_m(\lL)= v^{m/2}\, \bar C_m$.

In Ref. \cite{DFK15}, it was shown that the conserved quantities $\bar C_m$ can be interpreted as type $A$ (relativistic) 
$q$-difference Toda Hamiltonians \cite{Etingof}. This justifies the identification of the  Pieri operators \eqref{typeAHamiltonian} 
as $q$-difference Toda Hamiltonians.

\subsection{Universal solutions and duality}\label{universAsec}

The results of the previous sections may be rephrased in the more uniform context of universal solutions, obtained explicitly in the case of type A in \cite{ShiraishiNoumi}. The following discussion treats the two sets $x$ and $s$ on equal footing, as formal variables, and implies that the duality property is more general.

\subsubsection{Universal Macdonald and Pieri solutions and duality}
We now consider $s=t^{\rho}q^{\lambda}$ as a {\it formal} variable which we may specialize to
integer partitions $\lambda$, and similarly write $x=q^\mu t^\rho$.
If we consider the difference operators \eqref{DA} in the space $\C(q,t)[[\{x^{-\al_i}\}]][\Gamma_1,...,\Gamma_N],$
i.e. expanded as power series of the variables $x_{i+1}/x_i$ with $1\leq i \leq N-1$,
then there is a unique series solution $P(x;s)$ to the Macdonald eigenvalue equation
\begin{equation}\label{eig}
\mathcal D_1(x;q,t)\, P(x;s)= e_1(x)\, P(x;s)
\end{equation}
of the form
\begin{equation}\label{xseriessol}
P(x;s) = q^{\lambda\cdot \mu} \sum_{\beta\in Q_+} c_{\beta}(s;q,t) x^{-\beta}, \qquad c_0(s;q,t)=1. 
\end{equation}
Note that $q^{\lambda\cdot \mu}=x^\lambda\, t^{-\rho\cdot \lambda}  = s^\mu \,t^{-\rho\cdot \mu}$ is symmetric with respect to the interchange of $x$ and $s$. The uniqueness of the solution follows from the fact that \eqref{eig} is a nonsingular triangular system for the coefficients $c_\beta(s;q,t)$.

The series \eqref{xseriessol} is called the universal Macdonald solution because, under the specialization of the 
variables $s =q^{\lambda}\, t^{\rho} $ with $\lambda$ an integer partition, the function $t^{\rho\cdot \lambda}\,P(x;s)$ specializes to the symmetric Macdonald  polynomial $P_\lambda(x)$:
\begin{equation}\label{specializationA}
P_{\lambda}(x) = t^{\rho\cdot \lambda} P(x;q^\lambda\,t^\rho ),
\qquad \lambda=(\lambda_1\geq\cdots\geq\lambda_N),\quad \lambda_i\in \Z_{\geq 0}.
\end{equation}
That is, the infinite series \eqref{xseriessol} truncates to a finite number of terms, and the prefactor $t^{\rho\cdot \lambda} q^{\lambda\cdot \mu}=x^\lambda$ ensures that the function is a symmetric polynomial of $x_1,...,x_N$. 
All Macdonald polynomials $P_\lambda(x)$ are obtained as specializations of $P(x;s)$.
On the other hand, the series $P(x;s)$ also specializes to the Baker-Akhiezer quasi-polynomials introduced by Chalykh \cite{Chalykh}. In this case, the relevant specialization consists in taking $t=q^{-k}$, $k$ some positive integer. One can check that another truncation occurs leaving us with a finite sum.
\begin{example}
Let us illustrate the phenomenon of truncation in the simplest case of $A_1^{(1)}$ of rank $N=2$. The universal Macdonald series is expressed in terms of the 
variables $u=x_2/x_1$ and $v=\lL_2/\lL_1$ as
$$ P(x;s)=x_1^{\lambda_1}x_2^{\lambda_2} t^{-\lambda_1}\, 
\sum_{n=0}^\infty u^n \,\prod_{i=0}^{n-1} \frac{1- v q^i}{1-t^{-1}v q^i}\,\frac{1-t^{-1}q^{-i}}{1-q^{-i-1}}.$$
The first type of truncation occurs for $v=q^{\lambda_2-\lambda_1}=q^{-k}$ for some integer $k\geq 0$. We see that the first factor in the numerator 
of the coefficient in the series vanishes as 
soon as $n>k$ (for $i=k$).  The second type of truncation occurs when $t=q^{-k}$ for some integer $k\geq 0$ 
and arbitrary $\lambda$. In that case the second factor vanishes as soon as $n>k$.
\end{example}

Similarly, if we consider $\mathcal H_1(\Lambda;q,t)$ as a power series in $\{s_{i+1}/s_i\}$, 
then the equation  
\begin{equation}\label{HQ}
t^{-\rho\cdot \lambda}\mathcal H_1(\Lambda;q,t)t^{\rho\cdot \lambda}Q(s;x) = e_1(x) Q(s;x)
\end{equation}
has a unique ``universal Pieri" solution of the form
\begin{equation}\label{Qseries}
Q(s;x)
= q^{\lambda\cdot \mu} \sum_{\beta\in Q_+} \bar{c}_\beta(x;q,t) s^{-\beta}, \qquad \bar{c}_0(x;q,t)=1.
\end{equation}

An outcome of the study of \cite{ShiraishiNoumi} is a relation between the universal Macdonald and Pieri solutions, which we re-prove below.
\begin{thm}\label{macpiedelta} The universal Macdonald and Pieri solutions are related via
$$P(x;s)=\Delta(x)\, Q(s;x),$$
with $\Delta(x)$ as in \eqref{DeltaA}. 
\end{thm}
\begin{proof}
First note that $P(x;s)$ obey both Macdonald eigenvalue and Pieri equations \cite{Cherednik95,Chalykh}. 
Expanding the coefficients of $P(x;s)$ as series of $s^{-\al_i}$: $c_\beta(s)=\sum_{\delta \in Q_+} c_{\beta,\delta}s^{-\delta}$ allows to rewrite
$P(x;s)=q^{\lambda\cdot \mu}\sum_{\delta  \in Q_+} \hat c_\delta(x)\, s^{-\delta}$, where $\hat c_\delta(x)=\sum_{\beta \in Q_+}c_{\beta,\delta}x^{-\beta}$.
Now both $P(x;s)$ and $Q(s;x)$ can be viewed as two different solutions of the Pieri equation. By uniqueness, they must be proportional up to an $s$-independent factor. We deduce that $P(x;s)=\hat c_0(x) Q(s;x)$. To compute $\hat c_0(x)$, note that it is the limit of 
$q^{-\lambda\cdot \mu}P(x;s)=\hat c_0(x)+O(s^{-\al_i})$, when we take $|s_1|>>|s_2|>>\cdots |s_N|>>1$, so that all $s^{-\al_i}\to 0$. Using the $a$-th Macdonald eigenvalue equation \eqref{eigenvalue} and the explicit formula \eqref{DA}, and dividing by $\lL_1\lL_2\cdots \lL_a$, we have for $a=1,2,...,N$:
$$\left\{t^{-{a\choose 2}}\frac{e_a(s)}{\lL_1\lL_2\cdots \lL_a}-\sum_{I\subset [1,N]\atop |I|=a} \prod_{j\not\in I\atop i\in I} \frac{t x_i-x_j}{x_i-x_j} \frac{1}{\lL_1\lL_2\cdots \lL_a}\prod_{i\in I} \lL_i \Gamma_i\right\}\, \big(\hat c_0(x)+O(x^{\al_i})\big)=0.$$
In the limit $\lL_{i+1}/\lL_i\to 0$, all the terms tend to zero except for the leading term in the eigenvalue, and the term corresponding to $I=\{1,2,...,a\}$ in the sum:
\begin{equation}\label{eyh} \left\{1-\prod_{1\leq i\leq a<j\leq N} \frac{1-\frac{x_j}{t x_i}}{1-\frac{x_j}{x_i}} \,
\Gamma_1\Gamma_2\cdots \Gamma_a\right\}\,\hat c_0(x)=0.
\end{equation}
We verify that $\hat c_0(x)$ and $\Delta(x)$ of \eqref{DeltaA} both obey \eqref{eyh} for $a=1,2,...,N$. 
Their ratio must therefore be independent of $x$ as it is invariant under each $\Gamma_i$,
and is easily identified to 1 by taking the limit when all $x^{-\al_i}\to 0$ and using $c_{0,0}=1$. The Theorem follows.
\end{proof}

Repeating the argument of Section \ref{Duality} starting with \eqref{eig}, including the interchange of the labels $x$ and $s$,
we see that
$$t^{-\rho\cdot \lambda}\mathcal H_1(\Lambda;q,t)t^{\rho\cdot \lambda}\, Q(x;s)=e_1(x)\, Q(x;s) ,$$
where the expansion in $s$ of $Q(x;s)=P(s;x)/\Delta(s)$ has the form $q^{\lambda\cdot \mu}(1+\mathcal{O} (s^{-\alpha_i}))$. By uniqueness of the solution, 
\begin{equation}\label{symmetryA}
Q(x;s)=Q(s;x)\quad \Leftrightarrow \quad \Delta(s)\,P(x;s)=\Delta(x)\,P(s;x).
\end{equation}
The duality \eqref{symmetryA} states that the series $Q(s;x)$ in $s$ is equal to the series $Q(x;s)$ in $x$. The specialization of $\lambda$ and $\mu$ to integer partitions reduces to the Macdonald polynomial duality relation \eqref{dualpoA}. We may therefore interpret \eqref{symmetryA} as a universal extension of the duality of Macdonald theory.

\subsubsection{$q$-Whittaker limit}

The symmetry \eqref{symmetryA} is replaced in the $t\to\infty$ limit by 
\begin{equation}\label{limitofsymmetry}
\Pi(x;\Lambda) = \overline{\Delta}(x) \mathrm{K}(\Lambda;x),\qquad \overline{\Delta}(x)=\prod_{\alpha\in R_+} (q x^{-\alpha};q)_\infty,
\end{equation}
where the universal $q$-Whittaker function $\Pi(x;\Lambda)$ and Pieri solution $\mathrm{K}(\Lambda;x)$ read respectively
$$
\Pi(x;\Lambda)=\lim_{t\to\infty} t^{\rho \lambda} P(x;s), \quad \mathrm{K}(\Lambda;x) = \lim_{t\to\infty} t^{\rho \lambda} Q(s;x)
$$
and satisfy the limiting Macdonald eigenvalue equations and Pieri rules:
\begin{eqnarray}
D_a(x;q)\, \Pi(x;\Lambda)&=& \Lambda^{\omega_a}\, \Pi(x;\Lambda) \label{UeigenvalueW}\\
H_a(\Lambda;q) \, \mathrm{K}(\Lambda;x) &=& e_a(x)\mathrm{K}(\Lambda;x) \label{HKappa}
\end{eqnarray}

As in the case of the Macdonald function, the universal functions $\Pi(x;\Lambda)$ and $\mathrm{K}(\Lambda;x)$ are uniquely determined from the difference equations \eqref{UeigenvalueW} and \eqref{HKappa} as series of the form
\begin{eqnarray}
\Pi(x;\Lambda) &=& x^\lambda \sum_{\beta \in Q_+} {c}_\beta(\Lambda;q) x^{-\beta},\qquad {c}_0(\Lambda;q)=1,\label{piseries}\\
\mathrm{K}(\Lambda;x) &=& x^\lambda \sum_{\beta\in Q_+}\bar{c}_\beta(x;q) \Lambda^{-\beta},\qquad 
\bar{c}_0(x;q)=1.\label{kappaseries}
\end{eqnarray}

\begin{example}\label{pathAex}
Let us compute  $\mathrm{K}(\Lambda;x)$ explicitly as a solution of the first Pieri equation $H_1(\Lambda;q) \, \mathrm{K}(\Lambda;x)= e_1(x)\mathrm{K}(\Lambda;x)$ with $H_1(\lL;q)$ as in \eqref{typeAHamiltonian} for $a=1$,
by use of a path model. The equation reduces to the following triangular system for the coefficients $\bar{c}_\beta(x;q)$ from \eqref{kappaseries}:
$$\left\{\sum_{i=1}^N x_i (q^{\beta_{i-1}-\beta_i}-1)\right\}\bar{c}_\beta(x;q)=\sum_{i=1}^{N-1} x_{i+1} \,q^{\beta_i-\beta_{i+1}-1}\, \bar{c}_{\beta-e_i}(x;q),$$
with $\beta=\sum_{i=1}^{N-1} \beta_i e_i$ and $\beta_0=\beta_N=0$ by convention. 
The system is nonsingular for $x$ generic, and
we may express $\bar{c}_\beta(x;q)$ as a path model
partition function, namely as the sum over all paths $p$ from $0$ to $\beta$ in the positive quadrant $\Z_+^{N-1}$ of path weights $w(p)$. 
The path weight is defined as a product along the path of its vertex weights $w_v(a)=(\sum_{i=1}^n x_i (q^{a_{i-1}-a_i}-1))^{-1}$ per vertex $a$ visited and edge step weights $w_e(s)=x_{i+1} \,q^{b_i-b_{i+1}-1}$ per edge step $s=b-e_i\to b$:
$$\bar{c}_\beta(x;q)=\sum_{{\rm paths}\, p:0\to \beta} \quad \prod_{{\rm vertex}\, a\, {\rm of}\, p} w_v(a)\prod_{ {\rm edge}\, {\rm step}\, s\, {\rm of}\, p}
w_e(s)$$
This construction parallels that of Whittaker vectors performed in \cite{DKT}.
\end{example}

The Fourier transform may be restated in terms of the universal $q$-Whittaker function as follows:
$f(x)\, \Pi(x;\lL)={\bar f}(\lL)\, \Pi(x;\lL)$, and our main theorem \ref{mainthm} as $D_{a;k}(x;q)\, \Pi(x;\lL)=\bar D_{a;k}(\lL;q)\, \Pi(x;\lL)$, also a consequence
of $g(\lL) \, \Pi(x;\lL)=\gamma(x)\,  \Pi(x;\lL)$.

\section{Duality property for other classical root systems}\label{macdoG}

In this section we present a generalization of the methods of Section \ref{Atype} for the root systems corresponding to the affine algebras in the classical series, listed in Table \ref{tableone}. In these cases, the underlying double affine Hecke algebra is of $BC$-type, corresponding to finite-type Weyl groups of types $C_N$ or $D_N$, and the commuting $q$-difference operators and their eigenfunctions are Koornwinder/Macdonald operators and polynomials.
A key ingredient
is the duality property, which implies a relation between the Koornwinder/Macdonald eigenvalue equations and the Pieri rules  \cite{vDselfdual,Sahi,cheredMehta}.

Our goal is to choose a set of $N$ commuting difference operators for each  algebra $\g$ such that, in the $q$-Whittaker limit, these difference operators and their $SL(2,\Z)$-translates satisfy the $\g$-type quantum Q-system.
There are several known constructions of the difference operators. Koornwinder introduced a first order operator \cite{Koornwinder} depending on generic parameters $(q,t,a,b,c,d)$, which was extended by van Diejen to a complete set of $N$ commuting difference operators \cite{vandiej}. 
Upon specialization of the parameters $(a,b,c,d)$, these correspond to the root systems of $\g$. In some cases, van Diejen's operators must be supplemented by operators introduced by Macdonald \cite{macdoroot}, who used the structure of specific root systems to produce difference operators for miniscule co-weights.
We combine these constructions and define a set of $N$ commuting {\it $\g$-Macdonald} operators in Appendix \ref{appA}, whose eigenvalues form a basis for the space of fundamental characters of the finite algebra $R^*=R(\g^*)$ (see Table \ref{tableone}).

This section is organized as follows. We introduce the Koornwinder operators for generic values of $(a,b,c,d)$ in (Section~\ref{sec:macdopol}). The specialization of the parameters corresponding to $\g$ of Table \ref{tableone} are explained in Section~\ref{sec:macdop}.
Finally, the $q$-Whittaker limit of the $\g$-Macdonald operators is described in Section~\ref{sec:qwhittak}, together with their $\tau_+$-translates. We will prove that the $\tau_+$ translation is the discrete time evolution of the associated quantum Q-systems in Section \ref{sec:proof}.

\subsection{Koornwinder Operators, Polynomials and Pieri rules}
\label{sec:macdopol}

\subsubsection{The Koornwinder operators}
\label{kooropsec}

Let $\mathbb F$ be the field of rational functions in the indeterminates $a,b,c,d, q^{\half},t^{\half}$ and  $W$ the Weyl group of type $C_N$. It acts on functions in $\mathbb F(x_1,...,x_N)$
by permutations and inversions of the variables. The $W$-invariant space of Laurent polynomials in $x$ is generated by
 the elementary symmetric functions $\hat e(x)$, defined via the generating function
\begin{equation}\label{Chatdef}
{\hat E}(z;x):=\prod_{i=1}^N\prod_{\epsilon=\pm 1} (1+z x_i^\epsilon)=\sum_{k=0}^{2N} z^k\, {\hat e}_k(x),
\end{equation}
so that ${\hat e}_1(x)=\sum_{i=1}^n( x_i+x_i^{-1})$. The space $\mathbb F[x,x^{-1}]^W$ is preserved by the action of the Koornwinder operator, defined as follows.
\begin{defn}
The Koornwinder operator $\K_1^{(a,b,c,d)}(x;q,t)$ is the  $q$-difference operator acting on $\mathbb F[x,x^{-1}]$:
\begin{eqnarray} 
\K_1^{(a,b,c,d)}(x;q,t)&=&\sum_{i=1}^N\sum_{ \epsilon=\pm 1} \Phi_{i,\epsilon}^{(a,b,c,d)}(x) \, 
 \, (\Gamma_i^{\epsilon}-1) ,\qquad \Gamma_i x_j = q^{\delta_{ij}} x_j \Gamma_i, \label{koorn} 
 \end{eqnarray}
 where
 \begin{eqnarray*}
\Phi_{i,\epsilon}^{(a,b,c,d)}(x) &=&\frac{(1-a x_i^{\epsilon})(1-b x_i^{\epsilon})(1-c x_i^{\epsilon})(1-d x_i^{\epsilon})}{(1-x_i^{2\epsilon})(1-q x_i^{2\epsilon})}\prod_{j\neq i} \frac{t x_i^\epsilon-x_j}{x_i^\epsilon-x_j}\frac{t x_i^\epsilon x_j-1}{x_i^\epsilon x_j-1}.\nonumber 
\end{eqnarray*}
\end{defn}

Koornwinder polynomials are the unique monic, symmetric, Laurent polynomial eigenfunctions of the eigenvalue equation
\begin{eqnarray*}\K_1^{(a,b,c,d)}(x;q,t) \, P_\lambda^{(a,b,c,d)}(x)&=&
\sigma t^{N-1} \left( \hat e_1(s)- \hat e_1(\sigma t^\rho) \right) P_\lambda^{(a,b,c,d)}(x)\\
\end{eqnarray*}
where $\lambda$ is any integer partition coding the leading term $x^\lambda$ of $P_\lambda^{(a,b,c,d)}(x)$. We make use of the notations
\begin{equation}\label{sabcd}
s_i=q^{\lambda_i} t^{\rho_i} \sigma, \qquad \sigma=\sqrt{{abcd}/{q}}, \qquad \rho_i=N-i \qquad(i=1,2,...,N).
\end{equation}

As an example, the first two Koornwinder polynomials are  $P_\emptyset^{(a,b,c,d)}(x)=1$, with eigenvalue 0, and
$$P_{1,0,0,...,0}^{(a,b,c,d)}(x)={\hat e}_1(x)+\frac{1-t^N}{1-t}\, 
\frac{  abcd(a^{-1}+b^{-1}+c^{-1}+d^{-1})t^{N-1}-(a+b+c+d)}{1-a b c d \, t^{2N-2}}.$$

\subsubsection{Koornwinder-Macdonald operators}
We define the Koornwinder-Macdonald operators to be the set of mutually commuting difference operators, which commute with the Koornwinder operator, chosen so that their eigenvalues are proportional to the elementary symmetric functions $\hat e_m(s)$.
The first order Koornwinder-Macdonald operator \eqref{MaKop}
is
\begin{equation}\label{MKopone}
{\mathcal D}_1^{(a,b,c,d)}(x;q,t)%
=\K_1^{(a,b,c,d)}(x;q,t)+\frac{1-t^N}{1-t}\left( 1+ \frac{abcd}{q}t^{N-1}\right).
\end{equation}
The additive constant in \eqref{MKopone} is $\sigma \,t^{N-1}\,{\hat e}_1(\sigma t^\rho)$, so that Koornwinder polynomials satisfy the eigenvalue equation
\begin{equation}\label{eigenKo}
{\mathcal D}_1^{(a,b,c,d)}(x;q,t)\, P_\lambda^{(a,b,c,d)}(x)= \sigma \,t^{N-1}\, 
{\hat e}_1(s)\, P_\lambda^{(a,b,c,d)}(x).
\end{equation}

\begin{defn}\label{KorMacdef}
The Koornwinder-Macdonald operators are a commuting family of difference operators $\{\D_m^{(a,b,c,d)}(x;q,t): m\in [1,N]\}$, which are linear
combinations of van Diejen's commuting difference operators \cite{vDselfdual}, uniquely defined by their eigenvalues: 
\begin{equation}
{\mathcal D}_m^{(a,b,c,d)}(x;q,t)\, P_\lambda^{(a,b,c,d)}(x)= d_{\lambda;m}^{(a,b,c,d)}\, P_\lambda^{(a,b,c,d)}(x),\quad d_{\lambda;m}^{(a,b,c,d)}= \sigma^m\,t^{m(N-\frac{m+1}{2})}\, {\hat e}_m(s). \label{eigenvecabcd}
\end{equation}
\end{defn}
These  operators are explicitly constructed in Appendix \ref{appA}, see Definitions \ref{defVandiej},  \ref{malphadef} and Theorem \ref{ABCDmacdopol}.

\subsubsection{Rains operators}\label{rainsubsec}
A useful alternative ``$N$-th order" difference operator was constructed by
Rains \cite{Rains}\footnote{Some of these operators actually appeared in earlier works of van Diejen,
but Rains' construction is more systematic.} in his study of BC-symmetric polynomials. Rains shows that the  operators
\begin{equation}\label{rains}
{\mathcal R}_N^{(u,v)}(x;q,t)= \sum_{\epsilon_1,\epsilon_2,...,\epsilon_N=\pm 1} 
\prod_{i=1}^N \frac{(1-u x_i^{\epsilon_i})(1-v x_i^{\epsilon_i})}{1-x_i^{2\epsilon_i}}
\, \prod_{1\leq i<j\leq n} \frac{1- t x_i^{\epsilon_i}x_j^{\epsilon_j}}{1- x_i^{\epsilon_i}x_j^{\epsilon_j}}\,
\prod_{i=1}^N \Gamma_i^{\epsilon_i/2} 
\end{equation}
are maps between $W$-invariant spaces with different values of the parameters  $(a,b,c,d)$. They act between bases  of Koornwinder polynomials as follows:
\begin{equation}\label{rainskor}
\begin{array}{rl}
{\mathcal R}_N^{(q^{-\half}{a},q^{-\half}b)}(x;q,t)\, P_\lambda^{(a,b,c,d)}(x)&= q^{-|\lambda|/2} \, \prod_{i=1}^N (1-a b q^{\lambda_i-1} t^{N-i}) \, P_\lambda^{(q^{-\half}a,q^{-\half}b,q^{\half}c,q^{\half}d)}(x), \\
{\mathcal R}_N^{(q^{-\half}c,q^{-\half}d)}(x;q,t)\, P_\lambda^{(a,b,c,d)}(x)&= q^{-|\lambda|/2} \, \prod_{i=1}^N (1-c d q^{\lambda_i-1} t^{N-i}) \, P_\lambda^{(q^{\half}a,q^{\half}b,q^{-\half}c,q^{-\half}d)}(x),
\end{array}
\end{equation}
where $|\lambda|=\sum_{i=1}^N\lambda_i$.
The product
\begin{eqnarray}
\widehat{\mathcal D}_N^{(a,b,c,d)}(x;q,t)&=&{\mathcal R}_N^{(a,b)}(x;q,t)\,{\mathcal R}_N^{(q^{-\half}c,q^{-\half}d)}(x;q,t)\label{rainsop}
\end{eqnarray}
commutes with the Koornwinder operators, since Koornwinder polynomials satisfy the eigenvalue equation
$$
\widehat{\mathcal D}_N^{(a,b,c,d)}(x;q,t)\, P_\lambda^{(a,b,c,d)}(x)=\hat{d}_{\lambda;N}^{(a,b,c,d)}\, P_\lambda^{(a,b,c,d)}(x),
$$
with eigenvalues
\begin{equation}\label{eigenrains}
\hat{d}_{\lambda;N}^{(a,b,c,d)}=q^{-|\lambda|}\prod_{i=1}^N (1-a b q^{\lambda_i} t^{N-i})(1-c d q^{\lambda_i-1} t^{N-i}).
\end{equation}
The operator \eqref{rainsop} is not linearly independent of the set of Koornwinder-Macdonald operators: See Section~\ref{rainsec}, Lemma \ref{rainstokor} for the explicit expression. However, its factored form will play a crucial role in the proof of our main theorems below.

\subsubsection{Duality}
\label{secduabcd}
Due to the existence of an anti-involution $*$ of the DAHA,
Koornwinder polynomials obey a remarkable duality property  \cite{vDselfdual,Sahi}. The involution acts on the parameters $(a,b,c,d)$ as 
\begin{equation}
a^*=\Big(\frac{abcd}{q}\Big)^{1/2},\quad
b^*=-\Big(q\frac{ab}{cd}\Big)^{1/2},\quad
c^*=\Big(q\frac{ac}{bd}\Big)^{1/2},\quad
d^*=-\Big(q\frac{ad}{bc}\Big)^{1/2},\label{dua}
\end{equation}
so that $\sigma^*=a$. The duality property for Koornwinder polynomials is 
\begin{equation}\label{duapolsabcd}
\frac{P_\lambda^{(a,b,c,d)}(q^\mu t^{\rho^*})}{P^{(a,b,c,d)}_\lambda(t^{\rho^*})}=
\frac{P_\mu^{(a^*,b^*,c^*,d^*)}(q^\lambda t^{\rho})}{P_\mu^{(a^*,b^*,c^*,d^*)}(t^{\rho})}.
\end{equation}
where $\mu,\lambda$ are both integer partitions.
$\rho,\rho^*$ are defined by
\begin{equation}\label{rhosabcd} 
t^{\rho_i}=t^{\rho^{(a,b,c,d)}_i}=\sigma t^{N-i}, \qquad 
t^{\rho^*_i}=t^{\rho^{(a^*,b^*,c^*,d^*)}_i}=a t^{N-i} .\end{equation}

\subsubsection{Pieri rules} \label{sec:koorpieri}
Using the duality relation \eqref{duapolsabcd}, one obtains the Pieri rules for Koornwinder polynomials with parameters $(a,b,c,d)$ from the eigenvalue equation with parameters $(a^*,b^*,c^*,d^*)$.
It is useful to define the function
\begin{eqnarray} 
\Delta^{(a,b,c,d)}(x)&:=&\prod_{i=1}^N \frac{(\frac{q}{x_i^2};q)_\infty}{(\frac{q}{a x_i};q)_\infty(\frac{q}{b x_i};q)_\infty (\frac{q}{c x_i};q)_\infty (\frac{q}{d x_i};q)_\infty  }
\prod_{1\leq i < j \leq N} \prod_{\epsilon=\pm1} \frac{(\frac{qx_j^\epsilon}{ x_i};q)_\infty}{(\frac{qx_j^\epsilon}{t  x_i};q)_\infty}.
\label{deltabcd}
\end{eqnarray}
The normalization factor in the duality \eqref{duapolsabcd} is given by (See Theorem 5.1 of \cite{vDselfdual}):
\begin{equation}\label{normpolabcd}
P_\lambda^{(a,b,c,d)}(t^{\rho^*})=t^{\rho^*\cdot\lambda}\frac{\Delta^{(a^*,b^*,c^*,d^*)}(t^{\rho})}{\Delta^{(a^*,b^*,c^*,d^*)}(q^\lambda t^{\rho})},\ \ 
P_\mu^{(a^*,b^*,c^*,d^*)}(t^{\rho})=t^{\rho\cdot\mu}\frac{\Delta^{(a,b,c,d)}(t^{\rho^*})}{\Delta^{(a,b,c,d)}(q^\mu t^{\rho^*})},
\end{equation}

Using the parametrization  $x=q^\mu t^{\rho^*}$ in the eigenvalue equation \eqref{eigenvecabcd},
\begin{equation}\label{KHmacdo}
{\mathcal D}_m(q^\mu\,t^{\rho^*};q,t) P_\lambda(q^\mu\,t^{\rho^*})=\theta_m\, {\hat e}_m(q^\lambda t^\rho) \, P_\lambda(q^\mu\,t^{\rho^*}), \qquad m\in[1,N],
\end{equation}
where $\theta_m=\sigma^m \,t^{m(N-\frac{m+1}{2})}$ and $\Gamma_i: \mu_j \mapsto \mu_j+\delta_{i,j}$ in the difference  operator $\D_m=\D_m^{(a,b,c,d)}$. Upon specializing  $\mu$ to integer partitions, we  use
the duality relation \eqref{duapolsabcd}, where we replace the normalization factor by \eqref{normpolabcd}. Equation \eqref{KHmacdo} becomes
\begin{eqnarray*}
\theta_m{\hat e}_m(q^\lambda t^\rho)\, P_\mu^*(q^\lambda\,t^{\rho})
&=& P_\mu^*(t^{\rho})\, {\mathcal D}_m(q^\mu\,t^{\rho^*};q,t)\,P_\mu^*(t^{\rho})^{-1}\, P_\mu^*(q^\lambda\,t^{\rho})\\
&=&\left( t^{\rho\cdot\mu} \Delta(q^\mu\,t^{\rho^*})^{-1}\, {\mathcal D}_m(q^\mu\,t^{\rho^*};q,t)\, \Delta(q^\mu\,t^{\rho^*})\, t^{-\rho\cdot\mu}\right)P_\mu^*(q^\lambda\,t^{\rho}).
\end{eqnarray*}
Acting with the involution $*$ on the parameters and 
interchanging the roles of $\lambda$ and $\mu$ (i.e. $x$ and $s$) above, we arrive at the following Pieri rules.
\begin{thm}\label{kortopieri}
The Pieri rules for the Koornwinder polynomials are
\begin{equation}
\cH_m^{(a,b,c,d)}(s;q,t) \, P_\lambda^{(a,b,c,d)}(x)= {\hat e}_m(x)\, P_\lambda^{(a,b,c,d)}(x),\label{PieriK}\end{equation}
where the  Pieri operators $\cH_m^{(a,b,c,d)}(s;q,t)$ are 
\begin{equation}
\cH_m^{(a,b,c,d)}(s;q,t)=
\frac{1}{\theta_m^*}  t^{\rho^*\cdot\lambda}\,
\Delta^{(a^*,b^*,c^*,d^*)}(s)^{-1}\, {\mathcal D}_m^{(a^*,b^*,c^*,d^*)}(s;q,t)\,\Delta^{(a^*,b^*,c^*,d^*)}(s)\,t^{-\rho^*\cdot\lambda}.\label{PieriOpK}
\end{equation}
In the difference operators $\mathcal H_m$, $s$ is specialized to $q^\lambda\,t^{\rho}$ with $\lambda$ an integer partition, $\theta_m^*=a^{m} t^{m(N-\frac{m+1}{2})}$, and $T_i: \lambda_j\mapsto \lambda_j+\delta_{i,j}$.
\end{thm}
\noindent The explicit first Pieri operator $\cH_1^{(a,b,c,d)}(s;q,t)$ is derived  in Theorem \ref{korpier} of Appendix \ref{appB}.

\subsection{Specialization of the parameters $(a,b,c,d)$}
\label{sec:macdop}
\begin{table}
\begin{center}
\renewcommand{\arraystretch}{2}
\begin{tabular}{ |c|c|c|c|c|c|c|c|c|c|} 
\hline
 $\g$ & $\g^*$& $a$ & $b$ & $c$ & $d$ & $R$  & S & $R^*$ & $\xi_\g$ \\ 
\hline
$D_{N}^{(1)}$ & $D_{N}^{(1)}$ &$1$ & $-1$ & $q^{\frac{1}{2}}$ & $-q^{\frac{1}{2}}$ & $D_N$ & $D_N$ & $D_N$ & 0 \\ 
\hline
$B_{N}^{(1)}$ & $C_{N}^{(1)}$ &$t $ & $-1$ & $q^{\frac{1}{2}}$ & $-q^{\frac{1}{2}}$ & $B_N$ & $B_N$ & $C_N $ &$\frac12$\\ 
\hline
$C_{N}^{(1)}$ &$B_{N}^{(1)}$ &$t^{\frac{1}{2}}$ & $-t^{\frac{1}{2}}$ & $t^{\frac{1}{2}}\,q^{\frac{1}{2}}$ & $-t^{\frac{1}{2}}\,q^{\frac{1}{2}}$ & $C_N$ & $C_N$ &$B_N$& 1 \\
\hline
$A_{2N-1}^{(2)}$ & $A_{2N-1}^{(2)}$ &$t^{\frac{1}{2}}$ & $-t^{\frac{1}{2}}$ & $q^{\frac{1}{2}}$ & $-q^{\frac{1}{2}}$ & $C_N$ & $B_N$ &$C_N$& $\frac12$ \\
\hline
$D_{N+1}^{(2)}$ & $D_{N+1}^{(2)}$ &$t $ & $-1$ & $t\, q^{\frac{1}{2}}$ & $-q^{\frac{1}{2}}$ & $B_N$ & $C_N$ &$B_N$& 1\\
\hline
$A_{2N}^{(2)}$ & $A_{2N}^{(2)}$ &$t $ & $-1$ & $t^{\frac{1}{2}}\,q^{\frac{1}{2}}$ & $-t^{\frac{1}{2}}\,q^{\frac{1}{2}}$ & $BC_N$ & -- & $BC_N$& 1\\
\hline
\end{tabular}
\vskip.5cm
\caption{Specialization of the Koornwinder parameters $a,b,c,d$ corresponding to the affine algebra $\g$. The pair  $(R,S)$ refer to a pair of classical root systems corresponding to Macdonald's notation \cite{macdoroot}, used in Appendix \ref{appendixA}, except for $A_{2n}^{(2)}$. }\label{tableone}    \label{korspec}
\end{center}
\end{table}
\subsubsection{$\g$-Macdonald operators and polynomials}
\label{macdopsec}

The Macdonald operators for type $\g$ are a set of $N$ commuting operators which commute with the specialization of the Koornwinder operators at values of $(a,b,c,d)$ indicated in Table \ref{korspec}. The list of operators is given in Definition \ref{malphadef}. In most cases, these are just the specialized Koornwinder-Macdonald operators, but for a few exceptions, where the operators are chosen so that their eigenvalues be fundamental characters of the finite-dimensional algebra $R^*=R(\g^*)$. This occurs when $R^*=B_N$ or $D_N$.

\begin{remark}\label{Atwotworem}
The case of $A_{2N}^{(2)}$ is special, as $R=BC_N=R^*$ is non-reduced. In this case, we must use the set of simple roots $\al_a=\al_a^*$ for $B_N$,
and the set of fundamental weights $\omega_a=\omega_a^*$ for $C_N$.
\end{remark}

The principle for the choice of $\g$-Macdonald operators is that their spectrum generates
the
Grothendieck ring of the algebra $R^*$ associated to each $\g$. The choice is not unique. 
Define
\begin{equation}\label{NGdef} N_\g:=\left\{ \begin{array}{ll}
N, &{\rm for}\ \g=A_{N-1}^{(1)},B_N^{(1)},A_{2N-1}^{(2)},A_{2N}^{(2)} ,\\
N-1, &{\rm for}\ \g= C_N^{(1)},D_{N+1}^{(2)} ,\\
N-2, & {\rm for}\  \g=D_N^{(1)}. \end{array}\right. 
\end{equation}
\begin{table}
\begin{center}
\renewcommand{\arraystretch}{2}
\begin{tabular}{|c|c|c|} 
\hline
R & $\hat{e}_a^{(R)}(x)$ \\ \hline
$D_N$ & $\hat{e}_1,...,\hat e_{N-2}, \hat e_{N-1}^{(D_N)}, \hat e_{N}^{(D_N)}$  \\ 
\hline
$B_N$ & $\hat e_1, \hat e_2,...,\hat e_{N-1}, \hat e_N^{(B_N)}$ \\ 
\hline
$C_N$, $BC_N$ & $\hat e_1, \hat e_2,...,\hat e_N$ \\ \hline
\end{tabular}
\vskip.5cm
\caption{\small The list of chosen symmetric functions forming a basis for the fundamental characters of the finite Lie algebra $R$. The exceptional symmetric functions in the table are given in \eqref{restofehat}.} 
\label{chartable}
\end{center}
\end{table}

\begin{defn}\label{gmacdodef} For $m\leq N_\g$,
the $\g$-Macdonald operators ${\mathcal D}_m^{(\g)}= {\mathcal D}_m^{(\g)}(x)$ are
$${\mathcal D}_m^{(\g)}:= {\mathcal D}_m^{(a,b,c,d)}, \qquad m\in[1,N_\g], $$
where $ {\mathcal D}_m^{(a,b,c,d)}$ are as in Definition \ref{KorMacdef} with
$(a,b,c,d)$ specialized according Table \ref{korspec}. For $N_\g<m\leq N$, we use the operators constructed by Macdonald,
described in Section \ref{appmacsec}:
\begin{eqnarray}
D_N^{(1)}:\quad &&
{\mathcal D}_a^{(D_N^{(1)})}:={\mathcal E}_{\omega_a}^{(D_N^{(1)})},\qquad a=N-1,N\label{dplus};\\
C_N^{(1)}:\quad && {\mathcal D}_N^{(C_N^{(1)})}:={\mathcal E}_{\omega_N}^{(C_N^{(1)})}, \label{cplus}\\
D_{N+1}^{(2)}:&&
{\mathcal D}_N^{(D^{(2)}_{N+1})}:={\mathcal E}_{\omega_N}^{(D^{(2)}_{N+1})} \label{dtwoplus}.
\end{eqnarray}
where ${\mathcal E}_{\omega_a}^{(D_N^{(1)})},{\mathcal E}_{\omega_N}^{(C_N^{(1)})},{\mathcal E}_{\omega_N}^{(D^{(2)}_{N+1})}$ are as in (\ref{MD2}-\ref{MD3}),\eqref{CNN},\eqref{D2NN} respectively.
\end{defn}

The choice of higher commuting difference operators \eqref{gmacdodef} is such that the eigenvalue equation for the $m$th $\g$-Macdonald operator is (see Theorem \ref{specmacG}):
\begin{equation}\label{macdo}
{\mathcal D}_m^{(\g)}(x;q,t) P_\lambda^{(\g)}(x)=\theta_m^{(\g)}\, {\hat e}_m^{(R^*)}(s) \, P_\lambda^{(\g)}(x), \qquad 
\end{equation}
where $\theta_m^{(\g)}$ are as in (\ref{defthone}-\ref{defthtwo}) and ${\hat e}_m^{(R)}(x)$ are listed in Table \ref{chartable} (as these involve the 
fundamental weights we choose the same eigenvalues for type $BC_N$ and $C_N$, see Remark \ref{Atwotworem}).

In particular, the first $\g$-Macdonald operators are simply the specialization of the Koornwinder-Macdonald operator \eqref{MKopone}. 
\begin{equation}
\D_1^{(\g)}(x;q,t) = t^{\rho_1^\g} \hat e_1(t^{\rho^\g}) + \sum_{i=1}^N \sum_{\epsilon=\pm1} \phi_{i,\epsilon}^{(\g)}(x;q,t) (\Gamma_i^\epsilon -1).
\end{equation}\label{macdops}
The factor $\sigma$ under the specialization is $t^{\xi_\g}$ (see Table \ref{tableone}), so that 
\begin{equation}\label{svarg}
s_i=q^{\lambda_i} t^{N+\xi_\g-i}=q^{\lambda_i} t^{\rho_i^{(\g)}}, \qquad (i=1,2,...,N) ,\end{equation}
where $\rho^{(\g)}=\rho(S)$ is the half-sum of positive roots of $S$ (except when $\g=A_{2N}^{(2)}$).
The functions $\phi_{i,\epsilon}^{(\g)} (x;q,t)$ are
\begin{eqnarray*}\renewcommand{\arraystretch}{2}
\phi_{i,\epsilon}^{(\g)} (x;q,t) = \prod_{j\neq i}\prod_{\epsilon'=\pm1}
\frac{t x_i^\epsilon x_j^{\epsilon'}-1}{x_i^\epsilon x_j^{\epsilon'}-1}\times
\left\{
\begin{array}{l l}\renewcommand{\arraystretch}{2}
1 , &\g= D_N^{(1)};\\ 
\displaystyle{\frac{t x_i^\epsilon-1}{x_i^\epsilon-1}}, & \g=B_N^{(1)};\\
\displaystyle{\frac{t x_i^{2\epsilon}-1}{x_i^{2\epsilon}-1}\frac{t q x_i^{2\epsilon}-1}{q x_i^{2\epsilon}-1}}, & \g=C_N^{(1)};\\
\displaystyle{\frac{t x_i^{2\epsilon}-1}{x_i^{2\epsilon}-1}}, & \g=A_{2N-1}^{(2)};\\
\displaystyle{\frac{t x_i^{\epsilon}-1}{x_i^{\epsilon}-1}\frac{t q^{1/2} x_i^{\epsilon}-1}{q^{1/2} x_i^{\epsilon}-1}}, & \g=D_{N+1}^{(2)};\\
\displaystyle{\frac{t x_i^{\epsilon}-1}{x_i^{\epsilon}-1}\frac{t q x_i^{2\epsilon}-1}{q x_i^{2\epsilon}-1}}, & \g=A_{2N}^{(2)}.\\
\end{array}
\right.
\end{eqnarray*}

The unique monic eigenfunction $P_\lambda^{(\g)}(x)$ of $D_1^{(\g)}(x;q,t)$ with eigenvalue $t^{\rho^\g_1} \hat{e}_1(s)$, where $\lambda$ is $\g$-partition, is the $\g$-type Macdonald (Laurent) polynomial.

\begin{defn}
A $\g$-partition is a set $\lambda=(\lambda_1,...,\lambda_N)$ such that $\sum_i \lambda_i e_i$ is a dominant integral weight of $R=R(\g)$. 
\end{defn}
The $\g$-partitions are simply integer partitions except in the cases $R=B_N, D_N$, where
\begin{eqnarray*}
\lambda\in \left\{\begin{array}{ll}
\left\{(\Z_+)^{N-1}\times \Z \right\}\cup 
\left\{(\Z_+\!\!+\!\!{\scriptstyle\frac{1}{2}})^{N-1}\times (\Z+\!\!{\scriptstyle\frac{1}{2}})\right\},\  
\lambda_1\geq \cdots \geq \lambda_{N-1}\geq |\lambda_N|,& R=D_N;\\ \\
(\Z_+)^N \cup (\Z_+\!\!+\!\!{\scriptstyle\frac{1}{2}})^{N},\ 
\lambda_1\geq \lambda_2\geq \cdots \geq \lambda_N\geq 0,& R=B_N.
\end{array}\right.
\end{eqnarray*}

 \begin{remark}
The $\g$-Macdonald polynomial $P_\lambda^{(\g)}(x)$ with $\lambda$ an integer partition is the specialization of the Koornwinder polynomial, indexed by the same partition, to the parameters of Table \ref{tableone}. 
For non-integer partitions, $\g$-Macdonald polynomials can be obtained from a different specialization of the parameters \cite{vandiej}, see Section \ref{bnproofsec} in type $B$. Alternatively, we can use the specialization of the
universal functions of Section \ref{sec:universal}, see remark \ref{accessrem}.
\end{remark}

\subsubsection{Duality}
\label{sec:duag}
The involution $*$ acting on the parameters $(a,b,c,d)$ implies an involution $\g\mapsto \g^*$ and $R\mapsto R^*$, as listed in Table \ref{tableone}. In particular, $\g=\g^*$ except in the cases $ (C_N^{(1)})^*= B_{N}^{(1)}$ and $(B_N^{(1)})^*= C_{N}^{(1)}$.

The duality relation \eqref{duapolsabcd}, specialized to $\g$-Macdonald polynomials is
\begin{equation}\label{duapols}
\frac{P_\lambda^{(\g)}(q^\mu t^{\rho^*})}{P^{(\g)}_\lambda(t^{\rho^*})}=\frac{P^{(\g^*)}_\mu(q^\lambda t^{\rho})}
{P^{(\g^*)}_{\mu}(t^{\rho})}.
\end{equation}
Note, however, that the range of validity of \eqref{duapols} is wider, as $\lambda$ can be any $\g$-partition, 
while $\mu$ is any $\g^*$-partition, and these are not necessarily integer partitions as in \eqref{duapolsabcd}.
First conjectured by Macdonald, the duality \eqref{duapols} was successively proved for all types in the case of integer partitions: 
for type A, it appears in \cite{macdo}, for other types the main proof is in \cite{Cherednik95}, 
supplemented by \cite{vDselfdual} and \cite{Sahi} upon specialization of \eqref{duapolsabcd}. See  Section \ref{sec:universal} for a more general duality statement.

\begin{table}
  \begin{center}
       \begin{tabular}{ |c|c|} 
\hline
 G & {\rm affine roots in} $\widehat{R}_{++}$
\\
\hline
$D_{N}^{(1)}$ & $n \delta +(e_i\pm e_j)\quad (1\leq i<j \leq N; n\geq 1)$ \\
\hline
$B_{N}^{(1)}$ &$ n\delta+ e_i$\ ($1\leq i\leq N$),\  $n \delta +(e_i\pm e_j)\  (1\leq i<j \leq N; n\geq 1)$ \\
\hline
$C_{N}^{(1)}$ &$n \delta+ 2e_i$\ ($1\leq i\leq N$),\  $n \delta +(e_i\pm e_j)\  (1\leq i<j \leq N; n\geq 1)$ \\
\hline
$A_{2N-1}^{(2)}$ &$2 n \delta+ 2e_i$\ ($1\leq i\leq N$),\  $n \delta +(e_i\pm e_j)\quad (1\leq i<j \leq N; n\geq 1)$ \\
\hline
$D_{N+1}^{(2)}$ &$\frac{n}{2} \delta+ e_i$\ ($1\leq i\leq N$),\  $n \delta +(e_i\pm e_j)\quad (1\leq i<j \leq N; n\geq 1)$ \\
\hline
$A_{2N}^{(2)}$ &$n \delta+ e_i$,\  $(2n-1)\delta+2 e_i$\ ($1\leq i\leq N$), $n \delta +(e_i\pm e_j)\quad (1\leq i<j \leq N; n\geq 1)$ \\
\hline
    \end{tabular}
    \vskip.5cm
     \caption{The subset $\widehat{R}_{++}$ of the affine roots for each affine algebra $\g$.}\label{posroots}
  \end{center} 

\end{table}

For $\alpha=n\delta + \alpha_1 + \cdots + \alpha_N\in\widehat{R}$, define $x^\alpha = q^{-n} x_1^{\alpha_1}\cdots x_N^{\alpha_N}$. The function $\Delta$ in Equation \eqref{deltabcd}  has the following specializations:
\begin{equation}
\Delta^{(\g)}(x) = \prod_{\alpha\in \widehat{R}_{++}}\frac{1-x^{-\alpha}}{1-t^{-1} x^{-\alpha}},
\label{deltaG}
\end{equation}
where $\widehat{R}_{++} = \{\alpha + n \delta\in \widehat{R} : \alpha\in R_+, n>0\}$ is a subset of the affine roots of $\g$, see
Table \ref{posroots}.
In all cases but $\g=A_{2N}^{(2)}$, $\Delta$ is closely related to a simplified version of Macdonald's function $\Delta^+$ \cite{macdoroot}, suitable for taking the dual $q$-Whittaker limit $t\to\infty$. 
The function $\Delta$ enters the duality relation \eqref{duapols} via the evaluation formulas (see (0.6) in \cite{Cherednik95})
\begin{equation}\label{normpol}
P^{(\g)}_\lambda(t^{\rho^*})
=t^{\rho^*\cdot\lambda}\frac{\Delta^{(\g^*)}(t^{\rho})}{\Delta^{(\g^*)}(q^\lambda t^{\rho})},\quad
P^{(\g^*)}_\mu(t^{\rho})=t^{\rho\cdot\mu}\frac{\Delta^{(\g)}(t^{\rho^*})}{\Delta^{(\g)}(q^\mu t^{\rho^*})}.
\end{equation}

\subsubsection{Pieri rules}
\label{sec:macdopieri}
As in the case of generic $(a,b,c,d)$, the Pieri rules follow from the eigenvalue equations \eqref{macdo} and the duality \eqref{duapols}. Start with \eqref{macdo} with $x=q^\mu\,t^{\rho^*}$, and specialize $\mu$ to a $\g^*$-partition:
\begin{equation}\label{Hmacdo}
{\mathcal D}_m^{(\g)}(q^\mu\,t^{\rho^*};q,t) P_\lambda^{(\g)}(q^\mu\,t^{\rho^*})=\theta_m^{(\g)}\, {\hat e}_m^{(R^*)}(q^\lambda t^\rho) \, P_\lambda^{(\g)}(q^\mu\,t^{\rho^*}),
\end{equation}
where $(x,\Gamma)$ is specialized to $(q^\mu t^{\rho^*},e^{\partial_\mu})$.
Following the same steps as in Section \ref{sec:koorpieri}, using the duality relation \eqref{duapols}, one obtains the $\g$-Pieri formulas:
\begin{thm}\label{mactopieri}
The Pieri rules for $\g$-Macdonald polynomials are
\begin{equation}
\cH_m^{(\g)}(s;q,t) \, P_\lambda^{(\g)}(x)= {\hat e}_m^{(R)}(x)\, P_\lambda^{(\g)}(x),
\label{PieriG}
\end{equation}
where the difference operators $\cH_m^{(\g)}(s;q,t)$ are given in terms of $\D^{(\g^*)}_m(s;q,t)$:
\begin{equation}
\cH_m^{(\g)}(s;q,t)=\frac{1}{\theta_m^{(\g^*)}}\, t^{\rho^*\cdot\lambda} \, \Delta^{(\g^*)}(s)^{-1}\, {\mathcal D}_m^{(\g^*)}(s;q,t)\, \Delta^{(\g^*)}(s) \,t^{-\rho^*\cdot\lambda} .
\label{PieriOpG}
\end{equation}
\end{thm}
The explicit Pieri operators $\cH_1^{(\g)}(s;q,t)$ are listed in Section \ref{firpierapp}.

\subsection{The $q$-Whittaker limit}
\label{sec:qwhittak}
The Macdonald polynomials have a symmetry $(t,q)\mapsto (t^{-1},q^{-1})$. For certain root systems, the $q$-Whittaker polynomials are the $t\to 0$ limit of the Macdonald polynomials, and therefore the $q^{-1}$-Whittaker polynomials are the $t\to\infty$ limit of these polynomials. In this paper, we define the various functions and operators so that they have well-defined limits as $t\to\infty$. We refer to this as the {\it $q$-Whittaker limit} by slight abuse of terminology. We also call $q$-Whittaker polynomials what technically should be called $q^{-1}$-Whittaker polynomials.

\subsubsection{Whittaker difference operators}\label{sec:qDiffOps}
The eigenvalue equations and Pieri rules of the previous section have well-defined $q$-Whittaker limits. Define
\begin{equation}\label{whitpol}
\Pi_\lambda^{(\g)}(x)=\lim_{t\to\infty} P_\lambda^{(\g)}(x)
\end{equation}
and
\begin{equation}\label{macdoD}
D_a^{(\g)}(x;q)=\lim_{t\to \infty} (\theta_a^{(\g)})^{-2} \,\D_a^{(\g)}(x;q,t),\qquad (a=1,2,...,N).
\end{equation}
The first $q$-Whittaker difference operators are obtained from Equation \eqref{macdops}:
\begin{equation}
D_1^{(\g)}(x;q)=1+\sum_{i=1\atop \epsilon=\pm1}^N \phi_{i,\epsilon}^{(\g)}(x) (\Gamma_i^{\epsilon}-1), \label{uniG}
\end{equation}
with
\begin{eqnarray}
\phi_{i,\epsilon}^{(\g)}(x)&=&\prod_{j\neq i}\frac{x_i^\epsilon}{x_i^\epsilon-x_j}\, \frac{x_i^\epsilon x_j}{x_i^\epsilon x_j-1}
\times\left\{ \begin{array}{ll}
1 & (D_N^{(1)}); \\
\frac{x_i^\epsilon}{x_i^\epsilon-1}, & (B_N^{(1)}), \\
\frac{x_i^{2\epsilon}}{x_i^{2\epsilon}-1}\, \frac{q x_i^{2\epsilon}}{q x_i^{2\epsilon}-1}, & (C_N^{(1)}),\\
\frac{x_i^{2\epsilon}}{x_i^{2\epsilon}-1}, & (A_{2N-1}^{(2)} ),\\
\frac{x_i^\epsilon}{x_i^\epsilon-1}\, \frac{q^{\frac{1}{2}} x_i^\epsilon}{q^{\frac{1}{2}} x_i^\epsilon-1}, & (D_{N+1}^{(2)}),\\
\frac{x_i^\epsilon}{x_i^\epsilon-1}\,\frac{q x_i^{2\epsilon}}{q x_i^{2\epsilon}-1}, & (A_{2N}^{(2)}).
\end{array}\right. \label{phig}
\end{eqnarray}

The limiting eigenvalues of $D_m^{(\g)}(x;q)$ are not symmetric functions because of the dependence of $s$ on $t$. Instead, they are given by the dominant term in $t$ in the functions $\hat{e}^{(R^*)}_a(s)$, $\Lambda^{\omega_a^*}$.
The eigenvalue equations are (see Section \ref{appqwhit}, Theorem \ref{qwithakthm})
\begin{equation}\label{eigenqwhit}
D_a^{(\g)}(x;q)\,  \Pi_\lambda^{(\g)}(x)=\Lambda^{\omega_a^*}\, \Pi_\lambda^{(\g)}(x).
\end{equation}

\subsubsection{Pieri rules and Toda Hamiltonians}\label{TodaHamil}

The Pieri operators \eqref{PieriOpG} have well-defined limits as $t\to\infty$. Let
\begin{equation}\label{pierilimq}
H_m^{(\g)}(\lL;q):=\lim_{t\to \infty} \cH_m^{(\g)}(s=\lL t^{\rho^{(\g)}};q,t) .
\end{equation}
Using the Pieri operators from Section \ref{firpierapp} (see also Remark \ref{remlimops}), and taking the $t\to\infty$ limit, we obtain the following list of first Pieri operators.
\begin{thm}\label{hamilist}
For all $\g$, the first Pieri operators take the form
\begin{equation}
H_1^{(\g)}(\Lambda;q) = \sum_{i=1}^N(1- \Lambda^{-\alpha_{i-1}^*})T_i + \sum_{i=1}^{N_{R*}} (1-\Lambda^{-\alpha_i^*})T_i^{-1} + M^{(\g)}(\Lambda;q), \label{Hamil}
\end{equation}
where the roots $\al_i^*$ are the simple roots of $R^*$
and $\al_0^*=0$ by convention. 
Here, $N_{R^*}=N-1$ if $R^*=B_N$, $N-2$ if $R^*=D_N$, and is equal to $N$ otherwise.
The boundary terms  $M^{(\g)}(\Lambda;q)$ are as follows:
\begin{equation}
\renewcommand{\arraystretch}{1.5}
M^{(\g)} (\Lambda;q)= \left\{ \begin{array}{ll}
(1-\Lambda^{-\alpha_{N}^*})(T_{N}^{-1} + (1-\Lambda^{-\alpha_{N-1}^*})T_{N-1}^{-1}),& \g=D_N^{(1)},\\ 
(1-\Lambda^{-\alpha_N^*})(1-q\Lambda^{-\alpha_N^*})T_N^{-1} + \Lambda^{-\alpha_N^*}(q^{-1}\Lambda^{-\alpha_{N-1}^*} - (1+q^{-1})),&\g=B_N^{(1)},\\ 
0,& \g=C_N^{(1)}, A_{2N-1}^{(2)},\\ 
(1-\Lambda^{-\alpha_{N}^*})(1-q^{\frac12}\Lambda^{-\alpha_{N}^*})T_N^{-1},& \g=D_{N+1}^{(2)},\\ 
-\Lambda^{-\alpha_N^*}, & \g=A_{2N}^{(2)}.
\end{array}\right.\label{Hamil2}
\end{equation}
\end{thm}
Explicitly, the first Pieri operators are:
\begin{eqnarray*}
&&H_1^{(D_N^{(1)})}(\lL)=T_1+\sum_{i=2}^N \left(1-\frac{\lL_i}{\lL_{i-1}}\right)T_i+\sum_{i=1}^{N-2}
 \left(1-\frac{\lL_{i+1}}{\lL_{i}}\right)T_i^{-1}\nonumber \\
&&\qquad\qquad\qquad +\left(1-\frac{\lL_N}{\lL_{N-1}}\right)\left(1-\frac{1}{\lL_{N-1}\lL_N}\right)T_{N-1}^{-1}+
 \left(1-\frac{1}{\lL_{N-1}\lL_N}\right)T_{N}^{-1},\nonumber\\
&&H_1^{(B_N^{(1)})}(\lL)=T_1+\sum_{i=2}^N \left(1-\frac{\lL_i}{\lL_{i-1}}\right)T_i+\sum_{i=1}^{N-1}
 \left(1-\frac{\lL_{i+1}}{\lL_{i}}\right)T_i^{-1}\nonumber \\
&&\qquad\qquad\qquad +\left(1-\frac{1}{\lL_N^2}\right)\left(1-\frac{q}{\lL_N^2}\right)T_N^{-1}
 +\frac{q^{-1}}{\lL_{N-1}\lL_N}-\frac{1+q^{-1}}{\lL_N^2},\nonumber\\
&&H_1^{(C_N^{(1)})}(\lL)=T_1+\sum_{i=2}^N \left(1-\frac{\lL_i}{\lL_{i-1}}\right)T_i+\sum_{i=1}^{N-1}
 \left(1-\frac{\lL_{i+1}}{\lL_{i}}\right)T_i^{-1}
+\left(1-\frac{1}{\lL_{N}}\right)T_N^{-1}, \nonumber \\
&&H_1^{(A_{2N-1}^{(2)})}(\lL)=T_1+\sum_{i=2}^N \left(1-\frac{\lL_i}{\lL_{i-1}}\right)T_i+\sum_{i=1}^{N-1}
 \left(1-\frac{\lL_{i+1}}{\lL_{i}}\right)T_i^{-1}
+\left(1-\frac{1}{\lL_{N}^2}\right)T_N^{-1}, \nonumber\\
&&H_1^{(D_{N+1}^{(2)})}(\lL)= T_1+\sum_{i=2}^N \left(1-\frac{\lL_i}{\lL_{i-1}}\right)T_i+\sum_{i=1}^{N-1}
\left(1-\frac{\lL_{i+1}}{\lL_{i}}\right)T_i^{-1}\nonumber\\
&&\qquad\qquad\qquad +\left(1-\frac{1}{\lL_N}\right)\left(1-\frac{q^{\frac{1}{2}}}{\lL_N}\right)T_N^{-1}
 -\frac{1+q^{-\frac{1}{2}}}{\lL_N},\nonumber\\
 &&H_1^{(A_{2N}^{(2)})}(\lL)=T_1+\sum_{i=2}^N \left(1-\frac{\lL_i}{\lL_{i-1}}\right)T_i+\sum_{i=1}^{N-1}
 \left(1-\frac{\lL_{i+1}}{\lL_{i}}\right)T_i^{-1}+\left(1-\frac{1}{\lL_{N}}\right)T_N^{-1}-\frac{1}{\lL_{N}}.
\end{eqnarray*}

The first $q$-Whittaker Pieri rule is
\begin{equation}\label{pieriwhit}H_1^{(\g)}(\lL)\, \Pi_\lambda^{(\g)}(x)= {\hat e}_1(x)\, \Pi_\lambda^{(\g)}(x).
\end{equation}
Alternatively this equation can be interpreted as the eigenvector equation for $q$-Whittaker functions
in which the roles of variables $\lL$ and $x$ are interchanged. 

\begin{remark}\label{todarem}
The Pieri operators $H_1^{(\g)}(\Lambda;q)$ match the relativistic $U_q(R)$-Toda Hamiltonians acting on functions of $\Lambda$ 
\cite{vDE,gotsy} upon the following correspondence: 
$$D_N^{(1)}\to D_N,\ \  D_{N+1}^{(2)}\to B_N, \ \ A_{2N-1}^{(2)}\to C_N.$$
The equivalence uses the Etingof automorphism \cite{Etingof} $T_i\mapsto T_{2\varpi_i}$, $\lL_i \mapsto \lL_i T_{-\varpi_i}$ and $v=q^\half$ in the notations of \cite{gotsy}.
To our knowledge, the cases $\g=B_N^{(1)},C_N^{(1)},A_{2N}^{(2)}$ do not appear in the literature in relation to
standard constructions of $q$-Whittaker functions for quantum groups. By a slight abuse of terminology, we still call the corresponding limits of Macdonald polynomials $q$-Whittaker functions, and the limiting Pieri operators Toda Hamiltonians, and keep our labeling 
with (twisted) affine algebras to avoid confusion. The same correspondence in the $q=1$ limit occurs in relation to factorization dynamics \cite{Williams}, where the $B_N,C_N$
cases  match classical Q-system evolutions for the twisted algebras $\g=D_{N+1}^{(2)},A_{2N-1}^{(2)}$ respectively.
\end{remark}

\subsubsection{Generalized $q$-Whittaker difference operators}\label{sec:genqW}

In this section, we consider the $\tau_+\in SL(2,\Z)$-action on the $q$-Whittaker difference operators of Section \ref{sec:qDiffOps}.
Let $\gamma(x)$ be as in Equation \eqref{gamma}. We will define ``translated" operators $D^{(\g)}_{a;n}$ for all $a=1,2,...,N$ and $n\in \Z$,
by suitable use of iterated conjugations with $\gamma^{-1}$. There is a subtlety arising from a distinction according to whether $a$ is a long or short label.
These do not necessarily correspond to long and short roots of $R$, but are determined instead by the Q-system evolutions described in the next section. 
\begin{defn}\label{longshortlabels}
All labels $a\in[1,N]$ are long except for the following cases:
$a=N$ for $B_N^{(1)}$, $a\in[1,N-1]$ for $C_N^{(1)}$, and $a\in[1,N]$  for $A_{2N}^{(2)}$.
\end{defn}

In general, we think of the integer $n$ as a discrete time, and as ${\rm Ad}_{\gamma^{-1}}$
as a ``time translation" operator. 

\noindent{\bf Definition of $D_{1;n}^{(\g)}(x)$ and properties.} 
We start with the  definition of the operators $D_{1;n}^{(\g)}(x)$, which depends on whether 
$a=1$ is a long or short label.

\begin{defn}\label{defoneD1}
When the label $a=1$ is long ($\g=D_N^{(1)},B_N^{(1)},A_{2N-1}^{(2)},D_{N+1}^{(2)}$), 
define, for any $n\in\Z$,
\begin{equation}\label{longdef}
D_{1;n}^{(\g)}(x):=
q^{-\frac{n}{2}} \, \gamma^{-n}\, D_1^{(\g)}(x;q)\, \gamma^n =q^{-\frac{n}{2}}+\sum_{i,\epsilon} \phi_{i,\epsilon}^{(\g)}(x) \, (x_i^{n\epsilon}\Gamma_i^{\epsilon}-q^{-\frac{n}{2}}),
\end{equation}
with $\phi_{i,\epsilon}^{(\g)}(x)$ as in \eqref{phig}.
\end{defn}
\begin{defn}\label{deftwoD1} When the label $a=1$ is short ($\g=C_N^{(1)},A_{2N}^{(2)}$) define
\begin{equation}\label{minusone}
D_{1;-1}^{(\g)}(x)=\sum_{i,\epsilon} \phi_{i,\epsilon}^{(\g)}(x)\, x_i^{-\epsilon}(\Gamma_i^{\epsilon}-1) ,
\end{equation}
with $\phi_{i,\epsilon}^{(\g)}(x)$ as in \eqref{phig}.
Then for all $n\in \Z$ and for $i=0,-1$, define
\begin{equation}\label{evolodd}
D_{1;2n+i}^{(\g)}(x) =q^{-n}\, \gamma^{-2n}\, D_{1,i}^{(\g)}(x;q)\, \gamma^{2n},%
\end{equation}
where $D^{(\g)}_{1,0}(x;q)=D^{(\g)}_1(x;q)$.
\end{defn}
Conjugating \eqref{uniG} and \eqref{minusone} by $\gamma^{-2n}$, we have, when $\g=C_N^{(1)},A_{2N}^{(2)}$,
\begin{eqnarray}
D_{1;2n}^{(\g)}(x)&=&q^{-n}+ \sum_{i,\epsilon} \phi_{i,\epsilon}^{(\g)}(x)\, (x_i^{2n\epsilon}\Gamma_i^{\epsilon}-q^{-n}),\label{defsevodGone}\\
D_{1;2n-1}^{(\g)}(x)&=&\sum_{i,\epsilon} \phi_{i,\epsilon}^{(\g)}(x)\, x_i^{-\epsilon}(x_i^{2n\epsilon}\Gamma_i^{\epsilon}-q^{-n})  .\label{defsevodG}
\end{eqnarray}

The main difference between the cases of Definitions \ref{defoneD1} and \ref{deftwoD1} is that in the latter case the translation splits 
into distinct even and odd time evolution equations (\ref{defsevodGone}-\ref{defsevodG}). 

\begin{lemma}\label{comuonelem}
We have the commutation relation
\begin{equation}
[D_{1;n}^{(\g)}(x),{\hat e}_1(x)]=(q-1) \, D_{1;n+1}^{(\g)}(x)+(q^{-1}-1)\, D_{1;n-1}^{(\g)}(x)\label{Gid}.
\end{equation}
\end{lemma}
\begin{proof}
Using
$$[\Gamma_i^\epsilon,\hat e_1(x)]  = (q^{\epsilon}-1)x_i^\epsilon\Gamma_i^\epsilon + (q^{-\epsilon}-1)x_i^{-\epsilon}\Gamma_i^\epsilon  = \sum_{\epsilon'} (q^{\epsilon'}-1)x^{\epsilon \epsilon'}\Gamma^\epsilon.$$ 
we have
\begin{eqnarray*}
[D_{1,n}^{(\g)}, e_1^{(x)}]&=& \sum_{i=1}^N\sum_{\epsilon=\pm1} \phi_{i,\epsilon} x^{n \epsilon} [\Gamma_i^\epsilon,\hat e_1(x)] 
= \sum_{i,\epsilon,\epsilon'} \phi_{i,\epsilon} x^{(n+\epsilon')\epsilon}(q^{\epsilon'}-1)\Gamma^\epsilon.
\end{eqnarray*}
\end{proof}

\noindent{\bf Definition of higher generalized $q$-Whittaker operators. }

The higher, generalized $q$-Whittaker operators are denoted by $D_{a;n}^{(\g)}(x)$, with $a\in[1,N]$ and $n\in \Z$. 
They are defined as follows.
\begin{defn}\label{moredefn} 
For all $a\in[1,N]$ and for all $n$,
$$
D_{a,t_a n}^{(\g)}(x) = q^{-nt_a\omega^*_a\cdot \omega_a/2} \gamma^{-t_1 n} D_{a}^{(\g)}(x) \gamma^{t_1 n},
$$
where $D_a^{(\g)}(x)$ are defined in \eqref{macdoD} and $\omega_a$ $(\omega_a^*) $ are fundamental weights of $R$ ($R^*$), $t_a=2$ for short labels and $t_a=1$ for long labels  in Definition \ref{longshortlabels}.

\end{defn}

\begin{defn}\label{cadefs}
For the short labels $a\geq 2$ of $C_N^{(1)}$ and $A_{2N}^{(2)}$, define
\begin{equation}\label{oddshortn}
D_{a;-1}^{(\g)}:=\frac{(-1)^a}{q-1}[D_{1;-a}^{(\g)},D_{a-1;0}^{(\g)}]_{q^a},\qquad (n\in \Z),
\end{equation}
where $[A,B]_p= [A,B]_p=AB-p BA$. Then
$$D_{a;2n-1}^{(\g)}=q^{-n a}\, \gamma^{-2n}\, D_{a;-1}^{(\g)}\, \gamma^{2n} = \frac{(-1)^a}{q-1}[D_{1;2n-a}^{(\g)},D_{a-1;2n}^{(\g)}]_{q^a}.
$$
\end{defn}
Definition \ref{cadefs} anticipates on the desired connection to the quantum Q-systems of the next section and uses a recursive reformulation of the quantum Q-system solution obtained in \cite{DFK16}. 

Finally, to define the generalized difference operators with odd $n$ corresponding to the label $a=N$,  $\g=B_N^{(1)}$, we use
the Rains operators ${\mathcal R}_N^{(u,v)}$ and  $\widehat{\mathcal D}_N^{(a,b,c,d)}$of \eqref{rains} and \eqref{rainsop},   with the specialization $(a,b,c,d)=(t,-1,q^{1/2},-q^{1/2})$:
\begin{eqnarray*}
\widehat{\mathcal D}_N^{(B_N^{(1)})}&=&\widehat{\mathcal D}_N^{(t,-1,q^{1/2},-q^{1/2})}={\mathcal R}_N^{(1)}\, {\mathcal R}_N^{(0)} 
\end{eqnarray*}
with
$$
{\mathcal R}_N^{(1)}={\mathcal R}_N^{(t,-1)}, \qquad  {\mathcal R}_N^{(0)}={\mathcal R}_N^{(1,-1)} .
$$
In the $q$-Whittaker limit $t\to\infty$, these become 
\begin{eqnarray}
\widehat{D}_N^{(B_N^{(1)})}&=&\lim_{t\to \infty}t^{-N^2}\,   \widehat{\mathcal D}_N^{(B_N^{(1)})}= R_N^{(1)}\, R_N^{(0)} ,\label{rainsB}
\end{eqnarray}
with
\begin{eqnarray}\qquad R_N^{(0)}&=&\lim_{t\to\infty} t^{-N(N-1)/2} {\mathcal R}_N^{(1,-1)}=\sum_{\epsilon_1,...,\epsilon_N=\pm 1} 
\prod_{1\leq i<j\leq N}\frac{ x_i^{\epsilon_i}x_j^{\epsilon_j}}{x_i^{\epsilon_i}x_j^{\epsilon_j}-1}\, \prod_{i=1}^N \Gamma_i^{\epsilon_i/2}\label{rzero}\\
R_N^{(1)}&=&\lim_{t\to\infty} t^{-N(N+1)/2} {\mathcal R}_N^{(t,-1)}=\sum_{\epsilon_1,...,\epsilon_N=\pm 1} \prod_{i=1}^N \frac{x_i^{\epsilon_i}}{x_i^{\epsilon_i}-1}
\prod_{1\leq i<j\leq N}\frac{ x_i^{\epsilon_i}x_j^{\epsilon_j}}{x_i^{\epsilon_i}x_j^{\epsilon_j}-1}\, \prod_{i=1}^N \Gamma_i^{\epsilon_i/2}.\label{rone}
\end{eqnarray}
The factorized form of $\widehat{D}_N^{(B_N^{(1)})}$ can be used to define the discrete time evolution
as follows:
\begin{defn} Let $n\in \Z$. Define
\begin{eqnarray}
R_{N;n}^{(i)}&=& q^{-\frac{Nn}{8}} \gamma^{-n}\, R_N^{(i)} \, \gamma^{n},\qquad (i=0,1) , \nonumber\\
D^{(B_N^{(1)})}_{N;2n-1}&=& q^{-\frac{Nn}{4}} \, R_{N;n-1}^{(1)}\, R_{N;n}^{(0)}.\label{defDBNminusone}
\end{eqnarray}
\end{defn}
Note in particular that $D^{(B_N^{(1)})}_{N;2n-1}=q^{-\frac{n N}{2}}\, \gamma^{-n}\,D_{N;-1}^{(B_N^{(1)})}\gamma^{n}$, similarly to the even $n$ case
where $D^{(B_N^{(1)})}_{N;2n}=q^{-\frac{n N}{2}}\, \gamma^{-n}\,D_{N;0}^{(B_N^{(1)})}\gamma^{n}$.

\section{Quantum Q-systems and Macdonald operator conjecture}
\label{sec:proof}

For any affine algebra $\g$, there is a corresponding Q-system \cite{HKOTY,HKOTT}, which is a recursion relation satisfied by the characters of the KR-modules \cite{ChariMoura} of the Yangian $Y(\g)$. Their deformation into recursion relations for non-commuting variables, called quantum Q-systems, were first defined \cite{krKR} using the identification of the Q-system recursion relations as mutations in a cluster algebra. In that case, one may use the canonical quantization of the cluster algebra \cite{BZ}, and in \cite{krKR,Simon,SimonTwisted} it is shown that this quantization is related to the graded fusion characters for $\g$. In special cases, these are known to be related to $q$-Whittaker functions.

One of the main conjectures presented in \cite{DFKconj} is that for untwisted $\g$ in Table \ref{tableone}, 
some $q$-difference operators
$D_{a;n}^{(\g)}(x)$ satisfy
the quantum Q-system relations of type $\g$, and
that the operators 
$D_{a;1}^{(\g)}$ and $D_{a;-1}^{(\g)}$ act as raising and lowering operators when acting 
on the eigenfunctions $\Pi_\lambda^{(\g)}(\bx)$ of $D_{a;0}^{(\g)}$. The purpose of this section is to prove these statements for all $\g$ of Table \ref{tableone}.

\subsection{Statement of the main theorems}\label{sec:41}

\subsubsection{Quantum Q-systems}\label{QsysSec}
For each $\g$, we consider an algebra generated by invertible, non-commuting elements $\{Q_{a;n}^{\pm1}: a\in[1,N], n\in\Z\}$. The quantum Q-system of type $\g$ is a set of relations among these generators, which depend on the root data of $\g$. 

\begin{defn} \label{deflambda}
For each $\g$, let $\Lambda^{(\g)}$ be the $N\times N$ matrix defined by
\begin{equation}\label{deflambdag}
\Lambda^{(\g)}_{a,b}=\omega_a^*\cdot\omega_b , \qquad a,b\in[1,N],
\end{equation}
where $\omega_a$ and $\omega_a^*$ are the fundamental weights of type $R$ and $R^*$ respectively,
and $\cdot$ is the standard inner product.
\end{defn}
In the case of untwisted $\g$,  $\Lambda^{(\g)}$ is the inverse of the Cartan matrix $C$ of $R$,
since $\omega_a^*=\omega_a^\vee$.
For the case $\g=D_{N+1}^{(2)}$, $\Lambda^{(\g)}$ is the inverse of the {\it symmetrized} Cartan matrix of type $B_N$, and for $\g=A_{M}^{(2)}$, it is the inverse of the symmetrized Cartan matrix of type $C_{\lfloor \frac{M+1}{2}\rfloor}$:
\begin{eqnarray*}
\Lambda^{(D_{N+1}^{(2)})}_{a,b}&=& (D_{a,a})^{-1} \,\Lambda^{(B_N^{(1)})}_{a,b}=\left\{\begin{array}{ll}
\min(a,b) ,& (a,b<N) \\
\half \min(a,b), & (a=N,\, {\rm or}\, b=N)\\
\frac14 N, & (a=b=N),
\end{array}\right.\\
\Lambda^{(A_{2N-1}^{(2)})}_{a,b}&=&\Lambda^{(A_{2N}^{(2)})}_{a,b}=(D'_{a,a})^{-1} \,\Lambda^{(C_N^{(1)})}_{a,b}=\min(a,b).
\end{eqnarray*}
Here, $D={\rm diag}(1,1,...,1,2)$ and $D'=D^{-1}$ are such that $C=C^{(B_N)}D$ and $C=C^{(C_N)}D'$ are symmetric. These are chosen so that $\Lambda^{(\g)}_{1,1}=1$.
Defining $t_a=2$ for $\alpha_a$ a short simple root of $R$, $t_a=1$ for $\alpha_a$ long, $t_1/t_a$ is $D_{a,a}$ for type $B$
and $D'_{a,a}$ for type $C$.

The quantum Q-system relations are of two types. The first are $q$-commutation relations:
\begin{equation}\label{qsys1}
\Q_{a;t_a k+i}\,\Q_{b;t_b k+j}= q^{\Lambda^{(\g)}_{a,b}\, j-\Lambda^{(\g)}_{b,a}\,i}\,\Q_{b;t_b k+j}\, \Q_{a;t_a k+i}, \qquad (i,j=0,1,k\in \Z).
\end{equation}
The second type of relations are  evolution equations in the discrete time variable $n$, and
have the form $q^{\Lambda_{aa}} \Q_{a;n+1} \Q_{a;n-1} = \Q_{a;n}^2 - \mathcal T_{a;n}$ for some monomials $\mathcal T_{a;n}$. 
Let $\bar{N}_\g$ be the maximal integer such that $(\lL^{(\g)})_{a,a}=a$ for $a\leq\bar{N}_\g+1$:
$$\bar N_{D_N^{(1)}}=N-3,\quad \bar N_\g=N-2,\ (\g=B_N^{(1)},C_N^{(1)},D_{N+1}^{(2)}),\quad \bar N_{\g}=
 N-1,\ (\g=A_{2N}^{(2)}, A_{2N-1}^{(2)}).$$ 
The evolution equations are
\begin{equation}\label{qsys2}
q^{a}\, \Q_{a;n+1}\,\Q_{a;n-1}= \Q_{a;n}^2 - \Q_{a+1;n}\,\Q_{a-1;n}, \qquad a\in[1,\bar{N}_\g], \hbox{ all $\g$}.
\end{equation}
\begin{equation}\label{qsys3}\renewcommand{\arraystretch}{1.5}
\begin{array}{rcl}
D_N^{(1)}:\ \ \ q^{N-2}\, \Q_{N-2;n+1}\, \Q_{N-2;n-1}&=& \Q_{N-2;n}^2 -q^{-\frac{(N-2)n}{4}}\, \Q_{N-3;n}\,\Q_{N-1;n}\,\Q_{N;n} ,\\
q^{\frac{N}{4}}\,\Q_{N-1;n+1}\, \Q_{N-1;n-1}&=& \Q_{N-1;n}^2-q^{\frac{(N-4)n}{4}}\,\Q_{N-2;n} ,\\
q^{\frac{N}{4}}\,\Q_{N;n+1}\, \Q_{N;n-1}&=& \Q_{N;n}^2-q^{\frac{(N-4)n}{4}}\,\Q_{N-2;n} ,\\
B_N^{(1)}:\ \ \ q^{N-1}\,\Q_{N-1;n+1}\, \Q_{N-1;n-1}&=& \Q_{N-1;n}^2-\Q_{N-2;n}\Q_{N;2n}, \\
q^{\frac{N}{2}} \,\Q_{N;2n+1}\, \Q_{N;2n-1}&=& \Q_{N;2n}^2-q^{-n}\,\Q_{N-1;n}^2 ,\\
q^{\frac{N}{2}} \,\Q_{N;2n+2}\, \Q_{N;2n}&=& \Q_{N;2n+1}^2-q^{\frac{N}{2}-n-1}\,\Q_{N-1;n+1}\Q_{N-1;n}, \\
C_N^{(1)}:\ q^{N-1}\,\Q_{N-1;2n+1}\, \Q_{N-1;2n-1}&=& \Q_{N-1;2n}^2-q^{-\frac{Nn}{2}}\,\Q_{N-2;2n}\Q_{N;n}^2 ,\\
q^{N-1}\,\Q_{N-1;2n+2}\, \Q_{N-1;2n}&=& \Q_{N-1;2n+1}^2-q^{-\frac{Nn}{2}}\,\Q_{N-2;2n+1}\Q_{N;n+1}\Q_{N;n} ,\\
q^{\frac{N}{2}}\,\Q_{N;n+1}\, \Q_{N;n-1}&=& \Q_{N;n}^2-q^{\frac{(N-2)n}{2}}\, \Q_{N-1;2n}, \\
D_{N+1}^{(2)}:\ q^{N-1}\,\Q_{N-1;n+1}\, \Q_{N-1;n-1}&=& \Q_{N-1;n}^2-q^{-\frac{Nn}{4}}\, \Q_{N-2;n}\Q_{N;n}^2 ,\\
q^{\frac{N}{4}}\,\Q_{N;n+1}\, \Q_{N;n-1}&=& \Q_{N;n}^2-q^{\frac{(N-2)n}{4}}\,\Q_{N-1;n}, \\
A_{2N-1}^{(2)}:\ \ \ \ \ \ \ \ q^{N}\,\Q_{N;n+1}\, \Q_{N;n-1}&=& \Q_{N;n}^2-q^{-n}\, \Q_{N-1;n}^2,\\
A_{2N}^{(2)}:\ \ \ \ \ \ \ \ \ q^N\,\Q_{N;2n+1}\, \Q_{N;2n-1}&=& \Q_{N;2n}^2-q^{-n}  \Q_{N-1;2n} \Q_{N;2n},\\
q^N\,\Q_{N,2n+2}\, \Q_{N,2n}&=& \Q_{N;2n+1}^2-q^{-n}\, \Q_{N-1;2n+1} \Q_{N;2n+1}.
\end{array}\end{equation}

\begin{remark}\label{qsyscalgrem} If $\g\neq A_{2N}^{(2)}$, the evolution equations above 
are equivalent to the quantization of the Q-system
cluster algebras. These correspond to exchange matrices \cite{clusK,clusDFK,Williams}:
$B=\begin{pmatrix} C^{t}-C& -C^{t} \\ C & 0 \end{pmatrix}$ for untwisted $\g$ and
$B=\begin{pmatrix} 0 & -C \\ C & 0 \end{pmatrix}$ for twisted $\g$, where $C$ is the Cartan matrix of $R$. These correspond to the initial cluster data $(Q_{a;i}: a\in[1,N], i=0,1)$.
We choose a skew-symmetric $q$-commutation matrix to be the inverse of the skew-symmetrized matrix $B$,
of the form $\begin{pmatrix} 0 & \Lambda \\ -\Lambda^{t} & \Lambda^{t}-\Lambda\end{pmatrix}$ with
 $\lL=\lL^{(\g)}$ as in Definition \ref{deflambda}. 
The quantized cluster variables obey the commutation relations \eqref{qsys1}, and 
the relevant quantum mutations \cite{clusDFK} are
\begin{equation}\label{qmutation} q^{\Lambda^{(\g)}_{a,a}}\, \Q_{a;n+1}\Q_{a;n-1}= (\Q_{a;n})^2- q^{\half D_{a,a}^{-1}} : \mathcal \mathcal T_{a;n} : \qquad (a=1,2,...,N), \end{equation}
with $D_{a,a}=1$ except for $a=N$ in types $A_{2N-1}^{(2)}, D_{N+1}^{(2)}$, where it is equal to $\half, 2$, respectively. The monomial $T_{a;k}$ is the product of $\Q_{b;k'}$ appearing as the second term in the right hand side of each Q-system relation, not including any factors of $q$,
and the normal ordering $: \cdot :$ is defined as
$:\prod X_i^{\mu_i}{:} = q^{-\half\sum_{i<j} a_{i,j}\mu_i\mu_j}  X_{1}^{\mu_1} \cdots X_k^{\mu_k}$ if
 $X_iX_j=q^{a_{i,j}}X_jX_i$.
The  quantum Q-system relations (\ref{qsys2}-\ref{qsys3}) are equivalent to the quantum mutations \eqref{qmutation} after a renormalization of the cluster variables
(see \cite{DFKconj} Lemma 4.4).
\end{remark}

In this section, we prove the following Theorem, which is one of the main results of this paper:
\begin{thm}\label{Qsysconj} For each $\g$ in Table \ref{tableone}, 
the limit $t\to\infty$ of the generalized $\g$-Macdonald operators $D_{a;n}^{(\g)}(x,q)$ of Section \ref{sec:genqW} satisfy the corresponding quantum Q-system relations (\ref{qsys1}-\ref{qsys3}).
\end{thm}
Therefore the algebra generated by the elements $\Q_{a;k}$ subject to the quantum Q-system relations has a functional representation given by the difference operators of Section \ref{sec:genqW}.

\subsubsection{Raising/lowering operator conditions}

We refer to the polynomials $\Pi^{(\g)}_\lambda(x)$ \eqref{whitpol}, which are the common eigenfuctions of
$D_{a;0}^{(\g)}(x;q)=D_{a}^{(\g)}(x;q)$, as $q$-Whittaker polynomials.
For non-twisted $\g$, we conjectured in \cite{DFKconj} that $D_{a;1}^{(\g)}$, $D_{a;-1}^{(\g)}$ are raising and lowering operators acting on $\Pi_\lambda^{(\g)}(x)$. The following are the statement of this result for all $\g$ in Table \ref{tableone}, and will be proven in this Section:
\begin{thm}\label{raiseconj}
For any $\g$ in Table \ref{tableone}, 
\begin{eqnarray}
D_{a;0}^{(\g)}(x)\, \Pi_\lambda^{(\g)}(x)&=&\lL^{\omega_a^*}\,\Pi_\lambda^{(\g)}(x)\label{eigenD},\\
D_{a;1}^{(\g)}(x)\, \Pi_\lambda^{(\g)}(x)&=& \lL^{\omega_a^*}\, \Pi_{\lambda+\omega_a}^{(\g)}(x),
\label{raiseD}
\end{eqnarray}
where $\omega_a$ and $\omega_a^*$ are fundamental weights of $R$ or $R^*$, respectively.
\end{thm}

Combining Theorem \ref{raiseconj} and Theorem \ref{Qsysconj} with the relevant quantum Q-system relations, it follows that $D_{a;-1}^{(\g)}(x)$ is a lowering operator:
\begin{cor}\label{lowerconj}
For all $\g$, we have the following lowering operator conditions
\begin{equation}\label{lowerD}
D_{a;-1}^{(\g)}(x)\, \Pi_\lambda^{(\g)}(x)=\lL^{\omega_a^*}\,(1-\lL^{-\al_a^*})\, \Pi_{\lambda-\omega_a}^{(\g)}(x),\end{equation}
where $\al_a^*$ are the simple roots of $R^*$.
\end{cor}
\noindent Note that the prefactor guarantees that the result is $0$ whenever $\lambda-\omega_a$ is not a dominant integral $R$-weight, i.e. not a $\g$-partition.

The proof follows the steps of Section \ref{Atype}, using the Fourier transformed operators $\bar D_{a;n}^{(\g)}(\lL)$.

\subsection{Opposite quantum Q-systems and integrability}\label{sec:42}
For each $\g$, we consider the solutions $\bar\cQ_{a;n}$ which satisfy the quantum Q-system with the opposite multiplication, referred to as the quantum \barQ-system. 
We construct a representation $\barD_{a;n}(\Lambda)$ of the solutions $\bar \Q_{a;n}$ of the quantum \barQ-system, subject to appropriate initial data, in terms of $q$-difference operators written in terms of the quantum torus ${\mathbb T}_\lL=\langle\lL_i,T_i\rangle_{i=1}^N$. Elements in this algebra act on functions of $\Lambda$ as $q$-difference operators in $\Lambda$.
The
initial data $\barD_{a;0}(\lL)$ are deduced from the eigenvalue equation \eqref{eigenD}. 
These are supplemented with a choice
of $\barD_{a;1}(\lL)$
ensuring that the quantum commutation relations opposite to those of Eq. \eqref{qsys1} are satisfied:
\begin{equation}\label{cacommence} \barD_{a;0}\, \barD_{b;1}=q^{- \lL^{(\g)}_{a,b} } \, \barD_{b;1}\, \barD_{a;0},
\end{equation}
with $\lL^{(\g)}$ as in  \eqref{deflambda}.

\begin{defn}\label{firstdefops}
For all $\g$,  define
$$\barD_{a;0}:= \barD_{a;0}^{(\g)}(\lL)=\Lambda^{\omega_a^*}, \qquad \barD_{a;1}:=\barD_{a;1}^{(\g)}(\lL) =\Lambda^{\omega_a^*}\, T^{\omega_a},$$
where $\omega_a,\omega_a^*$ are fundamental weights of $R,R^*$.
\end{defn}
These obey the relations \eqref{cacommence}, since 
$\lL^{\omega_a^*}T^{\omega_b}=q^{-\omega_a^*\cdot\omega_b}T^{\omega_b}\lL^{\omega_a^*}$.
Since all Q-system evolutions are two-step recursion relations, the following are uniquely defined:
\begin{defn}\label{Def:Dbar}
Define $\barD_{a;n}^{(\g)}(\lL) = \bar\Q_{a;n}$ for all $n\neq 0,1$ to be the solutions of the $\g$-type quantum \barQ-system relations subject to the initial data in Definition \ref{firstdefops}.
\end{defn}

\begin{remark}
Due to the Laurent property of quantum cluster algebras \cite{BZ}, the solutions $\barD_{a;n}$ are Laurent in the initial data $\{\barD_{a;i}: i=0,1,a\in[1,N]\}$. Since these are monomials in $\{T_a,\Lambda_a,a\in[1,N]\}$, {\em all} quantum cluster variables are Laurent  in the quantum torus generators, therefore they
are $q$-difference operators.
Although this argument doesn't apply to $\g=A_{2N}^{(2)}$, we will show that all solutions of the \barQ-system are Laurent in this case also.
\end{remark}

\subsubsection{Time translation operator $g$}
\begin{thm}\label{longshort}
Let $\barD_{a,n}^{(\g)}$ be as in Definition \ref{Def:Dbar}.
For each $\g$ there exists an element $g=g^{(\g)}$ in a completion of the quantum torus ${\mathbb T}_\lL$, such that for all $n\in \Z$:
\begin{eqnarray}
\barD_{a,n}&=&q^{-\frac{n}{2}\lL^{(\g)}_{a,a}}\, g^n\, \barD_{a,0} \, g^{-n} , \qquad \hbox{$a$ long},\label{evolutlong}\\
\barD_{a,2n+i}&=&q^{-n\lL^{(\g)}_{a,a}}\, g^n\, \barD_{a,i} \, g^{-n},\qquad i=0,1\label{evolutshort} , \ \hbox{$a$ short}.
\end{eqnarray}
These elements are
\begin{equation}
\renewcommand{\arraystretch}{2.5}
g^{(\g)} =\left\{
\begin{array}{ll}
g_T\, g_\lL \,  \left(\frac{1}{\lL_{N-1}\lL_N};q\right)^{-1}_\infty, & \g=D_N^{(1)},\\
\left(g_T^{1/2}\,  \left(\frac{1}{\lL_N^2};q\right)^{-1}_\infty \right)^2\, g_\lL,&  \g=B_N^{(1)},  \\
(g_T \,g_\lL)^2  \left(\frac{1}{\lL_N};q\right)^{-1}_\infty, & \g=C_N^{(1)}, \\
g_T\, g_\lL\,  \left(\frac{1}{\lL_N^2};q^2\right)^{-1}_\infty, &  \g=A_{2N-1}^{(2)} ,\\
g_T\, g_\lL\, \prod_{n=0}^\infty \left(\frac{1}{\lL_N};q^{\frac12}\right)_\infty^{-1}, &  \g=D_{N+1}^{(2)},\\
 g_T\, g_\lL\, \left(q^{\frac{1}{2}}\frac{1}{\lL_N};q\right)^{-1}_\infty\,
g_T\, g_\lL\, \left(\frac{1}{\lL_N};q\right)_\infty^{-1}, & \g=A_{2N}^{(2)} ,
\end{array}\right.
\label{variousg}
\end{equation}
where we use the shorthand $g_T, g_\Lambda$ of Equations \eqref{gtvalue} and \eqref{glambdavalue}.
\end{thm}
Note the $\Lambda$-dependence of $g$ is only via $\Lambda^{-\alpha_a^*}$ where $\alpha_a^*$ are the simple roots of $R^*$.
\begin{proof} The proof is by induction on $n$. The inductive step relies on the fact
that Equations \eqref{evolutlong} and \eqref{evolutshort}, which can be written as
\begin{equation}\label{evolutboth}
\barD_{a;m+t_a} = q^{-\Lambda_{a,a}^{(\g)} t_a/2} g\barD_{a;m} g^{-1},\quad  t_a=\left\{ \begin{matrix}1 & {\rm if} \ a\ {\rm long}\\
2 & {\rm if} \ a\ {\rm short}
\end{matrix}\right. ,
\end{equation}
are compatible with the \barQ-system evolution, from which $\barD_{a;n}$ are defined. 
Define for all labels $a$: %
$\bar{\mathcal T}_{a;n}:=\bar\Q_{a;n}^2 -q^{\Lambda_{a,a}} \bar\Q_{a;n-1} \bar\Q_{a;n+1}.
$
If \eqref{evolutboth} hold for all  $m\leq n$,  then they hold for $m=n+1$ iff $g \bar{\mathcal T}_{a;n}g^{-1} = q^{\Lambda_{a,a}^{(\g)}t_a }\bar{\mathcal T}_{a;n+t_a}$. This is easily checked case-by-case.
To complete the proof, we must check (\ref{evolutlong}-\ref{evolutshort}) for some initial values of $n$.
Using the form of $\bar{\mathcal T}_{a;n}$, we see that for long roots, it is enough  to show that \eqref{evolutlong} holds for $n=-1$ and $n=0$, so we need the explicit expressions for 
$\barD_{a,n}$ with $n=-1,0,1$ to start the induction, and for short roots, we also need the expression for $n=2$:
\begin{eqnarray}
\barD_{a,-1} &=& \barD_{a,0} (1-\Lambda^{-\alpha_a^*}) T^{-\omega_a} ,\label{genericm1a} \qquad a\in[1,N],\\
\barD_{a,2} &=& \barD_{a,1} (1-\Lambda^{-\alpha_a^*}T^{-\alpha_a}) T^{\omega_a},\label{generic2a} \qquad \hbox{$a$ short.}
 \end{eqnarray}
 
Let
$h_{\al}(u)=(\lL^{-\al};q^u)_\infty^{-1}$, $h_\al=h_\al(1)$ and $h_{\al}^+=(q^\half \lL^{-\al};q)_\infty^{-1}$.
Then $g_\lL=\prod_{a=1}^{N-1} h_{\al_a^*}$.
It is useful to rewrite the $g$ operators of \eqref{variousg} as
\begin{eqnarray*}
g&=& g_T\,g_\lL\,h_{\al_N^*}(u),\quad \g=D_N^{(1)},A_{2N-1}^{(2)},D_{N+1}^{(2)} , \ u=1,2,\half\ {\rm resp.},
\nonumber \\
g&=& g_1 g_2\ \hbox{ where }\ g_1=\left\{ \begin{array}{l} 
g_T^\half \,h_{\al_N^*},\\ %
g_T\,g_\lL,\\%
g_T \,g_\lL\, h_{\al_N^*}^+,\\%
\end{array}\right.\  \hbox{ and } \
g_2=\left\{ \begin{array}{ll} 
g_1\,g_\lL,& \g=B_N^{(1)},\\
g_1\, h_{\al_N^*}, &\g=C_N^{(1)},\\
g_T g_\lL h_{\al_N^*},& \g=A_{2N}^{(2)},
\end{array}\right.
\end{eqnarray*}
We use the exchange relations, valid for any $x,y\in\mathbb Q$,
\begin{eqnarray}
(g_T)^x \,(\lL_a)^y &=&q^{\frac{x y^2}{2}} \,(\lL_a)^y\, (T_a)^{xy}\, (g_T)^x , \label{gtexch}\\
h_{\al_b^*}\, T^{\omega_a}&=& T^{\omega_a}\,(1-\lL^{-\al_a^*})\, h_{\al_b^*},\qquad
h_{\al_b^*}^{-1}\, T^{-\omega_a}=(1-\lL^{-\al_a^*}) T^{-\omega_a}\,h_{\al_b^*}^{-1} .\label{hexch}
\end{eqnarray}

{\em Long labels.} In this case, we need only verify that \eqref{evolutlong} hold for $n=\pm1$,
Note that $\barD_{a;0}=\Lambda^{\omega_a^*}$ commutes with $g_\Lambda$.
In the self-dual cases $\g=D_N^{(1)}, A_{2N-1}^{(2)},D_{N+1}^{(2)}$, all labels are long
and Equation \eqref{evolutlong} with $n=1$, $g \barD_{a;0} g^{-1}=q^{\Lambda_{a,a}/2}\barD_{a;1}$, follows from
$$g_T \,\barD_{a;0}=  %
q^{\frac{\Lambda_{a,a}}{2}} \lL^{\omega_a^*} T^{\omega_a}  g_T  =q^{\frac{\Lambda_{a,a}}{2}} \barD_{a;1}\, g_T.$$
For $n=-1$ it is equivalent to
\begin{eqnarray*}
g^{-1}\, \barD_{a;0}&=& h_{\al_N^*}^{-1}g_\lL^{-1}g_T^{-1} \,\lL^{\omega_a^*} =
q^{-\frac{\lL_{a,a}}{2}} \, \lL^{\omega_a^*} \,g_\lL^{-1}T^{-\omega_a}\, h_{\al_N^*}^{-1}\,g_\lL^{-1}\\
&=& q^{-\frac{\lL_{a,a}}{2}}\, \barD_{a;0}\,(1-\lL^{-\al_a^*})T^{-\omega_a}\,g^{-1}=q^{-\frac{\lL_{a,a}}{2}} \, \barD_{a;-1}\,g^{-1} .
\end{eqnarray*}
For the long labels $a<N$ of $B_N^{(1)}$, $\omega_a^*=\omega_a$, $\Lambda_{a,a}=a$ and \eqref{evolutlong} with $n=\pm1$ follow from
\begin{equation*}g_T^{1/2}\,h_{\al_N^*} \,g_T^{1/2}\,\barD_{a;0}= q^{\frac{a}{4}} g_T^{1/2}\,\lL^{\omega_a^*}\,T^{\omega_a/2}\,h_{\al_N^*} \,g_T^{1/2}\\
= q^{\frac a2} \,\lL^{\omega_a^*} T^{\omega_a}g_T^{1/2}\,h_{\al_N^*} \,g_T^{1/2}= q^{\frac a2} \barD_{a;1}g_T^{1/2}\,h_{\al_N^*} \,g_T^{1/2}
\end{equation*}
and
\begin{eqnarray*}
g^{-1}\, \barD_{a;0}&=&g_\lL^{-1}(h_{\al_N^*}^{-1}g_T^{-1/2})^2\,\lL^{\omega_a^*} 
=q^{-\frac{a}{2}}\,\lL^{\omega_a^*}\,
g_\lL^{-1}  T^{-\omega_a^*} \,h_{\al_N^*}^{-1}g_T^{-1/2}h_{\al_N^*}^{-1}g_T^{-1/2}\\
&=&q^{-\frac{a}{2}}\, 
\barD_{a,0}\,(1-\lL^{-\al_a^*})T^{-\omega_a^*} \,g^{-1}=q^{-\frac{\Lambda_{a,a}}{2}}\, \barD_{a;-1}\,g^{-1}.
\end{eqnarray*}
For the long label $a=N$ of $\g=C_N^{(1)}$, $2\omega_N^*=\omega_N$, \eqref{evolutlong} with $n=\pm1$ follows from
\begin{eqnarray*}g_T g_\lL g_T \,\barD_{N;0}&=&q^{\frac{N}{2}} g_T \lL^{\omega_N^*} g_\lL  T^{\omega_N^*}\, g_T 
=q^N \lL^{\omega_N^*} T^{2\omega_N^*}g_T g_\lL g_T=q^N\barD_{N;1}\, g_T g_\lL g_T 
\end{eqnarray*}
and
\begin{eqnarray*}
&&g^{-1}\, \barD_{N;0}=h_{\al_N^*}^{-1}g_\lL^{-1}g_T^{-1}g_\lL^{-1}g_T^{-1}\lL^{\omega_N^*} = 
q^{-\frac{N}{4}}\, h_{\al_N^*}^{-1}g_\lL^{-1}g_T^{-1}\lL^{\omega_N^*}T^{-\omega_N^*}g_\lL^{-1}g_T^{-1}\\
&&\quad=q^{-\frac{N}{2}}\,\lL^{\omega_N^*}\,h_{\al_N^*}^{-1}T^{-2\omega_N^*}\,(g_\lL^{-1}g_T^{-1})^2=q^{-\frac{N}{2}}\,\barD_{N;0}\,(1-\lL^{-\al_N^*}) T^{-2\omega_N^*}\,g^{-1}=q^{-\frac{N}{2}}\,\barD_{N;-1}\, g^{-1}.
\end{eqnarray*}
Therefore \eqref{evolutlong} holds when $n=\pm1$ and $a$ long.

{\em Short labels.} Next we verify \eqref{evolutshort} for short labels $a$. We claim the following half-evolution equations hold 
in types $B_N^{(1)}, C_N^{(1)}, A_{2N}^{(2)}$:
\begin{eqnarray} q^{-\half \lL_{a,a}} g_2\, \barD_{a;0}\,g_2^{-1}&=&\barD_{a;1},\label{g2one}\\
q^{-\half \lL_{a,a}} g_1\, \barD_{a;1}\,g_1^{-1}&=&\barD_{a;2}.\label{g2two}
\end{eqnarray}
To see this, Equation \eqref{g2one} for $B_N^{(1)}$ with $a=N$, where $\omega_N^*=2\omega_N$, follows from
$$g_2\, \barD_{N;0}\,g_2^{-1}=g_T^{\half}\lL^{\omega_N^*}g_T^{-\half}=q^{\frac N4} \lL^{\omega_N^*}T^{\omega_N}
=q^{\frac{\Lambda_{N,N}}{2}}\, \barD_{N;1}.$$
For the short labels of $C_N^{(1)},A_{2N}^{(2)}$, where $\omega_a=\omega_a^*$, it follows from
$$g_2\, \barD_{a;0}\,g_2^{-1}=g_T\lL^{\omega_a^*}g_T^{-1}=q^{\frac{a}{2}}\, \lL^{\omega_a^*} \,T^{\omega_a^*}=
q^{\frac{\Lambda_{a,a}}{2}} \,\barD_{a;1}.$$
Equation \eqref{g2two} $\g=B_N^{(1)}$ with $a=N$, where $\al_N^*=2\al_N$, follows from
\begin{eqnarray*}
g_1\, \barD_{N;1}&=& g_T^\half\, \lL^{\omega_N^*} h_{\al_N^*} T^{\omega_N}=g_T^\half 
\lL^{\omega_N^*} T^{\omega_N}(1-\lL^{-\al_N^*})\, h_{\al_N^*}\\
&=& q^{\frac{N}{4}}\, \lL^{\omega_N^*} T^{2\omega_N}(1-q \lL^{-\al_N^*}T^{-\al_N^*/2} )\,g_T^\half\,h_{\al_N^*}
=q^{\frac{\Lambda_{N,N}}{2}}\, \barD_{N;2} \,g_1.
\end{eqnarray*}
When $a\leq N-1$ for $C_N^{(1)},A_{2N}^{(2)}$, for $a\leq N-1$, $\al_a^*=\al_a$ and
\begin{eqnarray*}
g_1\, \barD_{a;1}&=& g_T \lL^{\omega_a^*}\, g_\lL T^{\omega_a} =g_T \,
\lL^{\omega_a^*}\,T^{\omega_a} \,(1-\lL^{-\al_a^*})\, g_\lL\\
&=& q^{\frac{a}{2}} \, \barD_{a;1}\, (1-\lL^{-\al_a^*}T^{-\al_a^*})\,T^{\omega_a}\,g_T\,g_\lL=q^{\frac{a}{2}} \, \barD_{a;2}\,g_1 .
\end{eqnarray*}
For $A_{2N}^{(2)}$, $a=N$, 
\begin{eqnarray*}
g_1\, \barD_{N;1}&=& g_T \, \lL^{\omega_N^*}\,g_\lL\, h_{\al_N^*}^+ T^{\omega_N} =
g_T \, \lL^{\omega_N^*} (1-q^{-\half} \lL^{-\al_N^*}) T^{\omega_N}\,g_\lL\, h_{\al_N^*}^+\\
&=&q^{\frac{N}{2}}\, \barD_{N;1}\, (1- \lL^{-\al_N^*}T^{-\al_N^*})\,g_T\,g_\lL\, h_{\al_N^*}^+=
q^{\frac{N}{2}}\, \barD_{N;2}\,g_1.
\end{eqnarray*}
Equations \eqref{g2one},\eqref{g2two} imply \eqref{evolutshort} follows for $i=1,n=1$.
We  also have the half-evolution equations for short labels:
\begin{eqnarray} q^{\half \lL_{a,a}} g_1^{-1}\, \barD_{a;1}\,g_1&=&\barD_{a;0},\label{gone}\\
q^{\half \lL_{a,a}} g_2^{-1}\, \barD_{a;0}\,g_2&=&\barD_{a;-1}.\label{gtwo}
\end{eqnarray}
Equation \eqref{gone} is equivalent to
$\barD_{a;1}=q^{-\half \lL_{a,a}}g_1 \,\barD_{a;0}\,g_1^{-1}$, which is unchanged if we replace $g_1$ by $g_2$,
as only the $T$-dependent part $g_T$ acts on $\barD_{a;0}$, and the equation is therefore equivalent to
\eqref{g2one}.
To show  \eqref{gtwo} for $B_N^{(1)}$,
\begin{eqnarray*}
g_2^{-1}\, \barD_{N;0}&=& g_\lL^{-1}h_{\al_N^*}^{-1}\,g_T^{-\half} \lL^{\omega_N^*} 
=q^{\frac{N}{2}} \,\lL^{\omega_N^*}\, g_\lL^{-1}h_{\al_N^*}^{-1}\, T^{-\omega_N} \,g_T^{-\half}\\
&=& q^{\frac{N}{2}} \,\lL^{\omega_N^*}\,(1-\lL^{-\al_N^*}) \,T^{-\omega_N}g_\lL^{-1}\,h_{\al_N^*}^{-1}\,g_T^{-\half} 
= q^{\frac{N}{2}} \,\barD_{N;-1}\, g_2^{-1} ,
\end{eqnarray*}
and for $a$ short in $C_N^{(1)},A_{2N}^{(2)}$,
\begin{equation*}
g_2^{-1}\, \barD_{a;0}= q^{\frac{a}{2}} \, \lL^{\omega_a^*}\,h_{\al_N^*}^{-1}g_\lL^{-1}  T^{-\omega_a}\, g_T^{-1} 
= q^{\frac{a}{2}} \, \lL^{\omega_a^*}(1-\lL^{-\al_a^*}) T^{-\omega_a}\,h_{\al_N^*}^{-1}g_\lL^{-1} \, g_T^{-1}
=q^{\frac{a}{2}}\, \barD_{a;-1}\,g_2^{-1} .
\end{equation*}
These half-evolutions imply Equation \eqref{evolutshort} with $i=1,n=-1$.

The Theorem follows by induction.
\end{proof}

\subsubsection{Integrability and conserved quantities}

We claim that the time-translation operators $g$ commute with the Pieri operators for each $\g$, which therefore have an interpretation as the conserved quantities of the quantum \barQ-system.
\begin{thm} \label{lemone}
For all $\g$, the operator $g^{(\g)}(\lL)$ commutes with the first Pieri operator $H_1^{(\g)}(\lL)$.
\end{thm}
The proof is by explicit calculation for each $\g$. It is given in  Appendix \ref{appC}. 

In what follows, we need to generalize to type $\g$ the statement of Section \ref{sec:unique} about the uniqueness of the solutions to the first Pieri equation. Starting with
$D_1^{(\g)}(x) \Pi_\lambda^{(\g)}(x)=\lL_1\, \Pi_\lambda^{(\g)}(x)$, writing $\Pi_\lambda^{(\g)}(x)=x^\lambda p_\lambda^{(\g)}(x)$, and
conjugating \eqref{uniG} with $x^{-\lambda}$ results in the equation
$$ \left(\frac{1}{\lL_1}+\sum_{i=1\atop \epsilon=\pm 1}^N \phi_{i,\epsilon}^{(\g)}(x) 
\left(\frac{\lL_i^\epsilon}{\lL_1} \Gamma_i^\epsilon -\frac{1}{\lL_1}\right)-1\right)\, p_\lambda^{(\g)}(x)=0 .$$
The difference operator is polynomial in $\{\lL^{-\al_i^*}\}$, $\al_i^*$ the simple roots of $R^*$, so we may analytically continue the solution 
$p_\lambda^{(\g)}(x)$ to $\hat p_\lambda^{(\g)}(x)$, $\lambda\in \C^N$, 
and then consider $\Lambda$ to be a formal parameter. The coefficients $\phi_{i,\epsilon}^{(\g)}(x)$ have series expansions 
in the variables $x^{-\al_i}$, $\al_i$ the simple roots of $R$, hence we may expand $\hat p_\lambda^{(\g)}(x)=\sum_{\beta\in Q_+(R)} c_\beta^{(\g)}(\lL)x^{-\beta}$, as well as $c_\beta^{(\g)}(\lL)= \sum_{\delta\in Q_+(R^*)} c_{\beta,\delta}^{(\g)} \lL^{-\delta}$. When dealing with formal variables, we may exchange the summations and write a new series expansion,
$$ \hat p_\lambda^{(\g)}(x)= \sum_{\delta\in Q_+(R^*)} \hat c_\delta^{(\g)}(x) \lL^{-\delta}, \ \ \hat c_\delta^{(\g)}(x)=\sum_{\beta\in Q_+(R)} c_{\beta,\delta}^{(\g)} x^{-\beta} .$$
The first Pieri equation \eqref{pieriwhit} is also easily extended to formal $\lambda$, as the 
dependence on $\lL^{-\alpha_a^*}$ is polynomial, while $T_i: \Lambda_j\mapsto q^{\delta_{ij}}\Lambda_j$.

\begin{lemma}\label{uniqueglem}
Let $\Theta_\lambda(x)$ ($\lambda$ arbitrary) be a solution of the first Pieri equation \eqref{pieriwhit}, such that
$\Theta_\lambda(x)=x^\lambda\sum_{\beta\in Q_+(R^*)} \tau_\beta(x)\lL^{-\beta}$.
Then when $\lambda$ is evaluated as a $\g$-partition,
$$\Pi_\lambda^{(\g)}=\frac{\hat c_0^{(\g)}(x)}{\tau_0(x)}\, \Theta_\lambda(x).$$
\end{lemma}
\begin{proof}
The Pieri equation is a triangular linear system for the coefficients $\tau_\beta(x)$, determined uniquely for $\beta\neq0$ up to the factor $\tau_0(x)$. The same holds for the coefficients $\hat c_\beta^{(\g)}(x)$ in the expansion of $\hat p_\lambda^{(\g)}(x)$, up to $\hat c_0^{(\g)}(x)$. We deduce that $x^\lambda \hat p_\lambda^{(\g)}(x)=
\frac{\hat c_0^{(\g)}(x)}{\tau_0(x)}\, \Theta_\lambda(x)$. The Lemma follows by specialization to 
$\g$-partitions.
\end{proof}
\begin{thm}\label{lemtwo}
The action of $g(\Lambda)$ on $\Pi_\lambda^{(\g)}(x)$ is equivalent to acting by a Gaussian in $x$:
\begin{equation}\label{ftga}
g^{(\g)}(\lL)\, \Pi_\lambda^{(\g)}(x) =\gamma^{(\g)}(x)\, \Pi_\lambda^{(\g)}(x),\ \ \gamma^{(\g)}(x)=\gamma(x)^{t_1},
\end{equation}
where 
 $t_1=2$ if $1$ is a short label (i.e. $\g=C_N^{(1)},A_{2N}^{(2)}$), and is equal to 1 otherwise.
\end{thm}
\noindent In other words, $g^{(\g)}(\Lambda)$ is the Fourier transform of $\gamma(x)^{t_1}$.
\begin{proof} 
Like $\Pi_\lambda^{(\g)}(x)$, the quantity $g^{(\g)}(\lL)\Pi_\lambda^{(\g)}(x)$ can be continued to 
arbitrary $\lambda$, and we write it as $g^{(\g)}(\lL) x^{\lambda} \hat p_\lambda^{(\g)}(x)$.
Multiplying the first Pieri equation by $g^{(\g)}$ on the left, we get:
$$g^{(\g)}\,H_1^{(\g)}(\lL)\, x^{\lambda} \hat p_\lambda^{(\g)}(x)=H_1^{(\g)}(\lL)\,\Big(g^{(\g)}\,x^{\lambda} \hat p_\lambda^{(\g)}(x)\Big)
={\hat e}_1(x)\,\Big(g^{(\g)}\,x^{\lambda} \hat p_\lambda^{(\g)}(x)\Big),$$
hence both $g^{(\g)}\,x^{\lambda} \hat p_\lambda^{(\g)}(x)$ and $x^{\lambda} \hat p_\lambda^{(\g)}(x)$ obey the same first Pieri rule. Apply Lemma \ref{uniqueglem} to the function 
$\Theta_\lambda(x)=g^{(\g)}\,x^{\lambda} \hat p_\lambda^{(\g)}(x)
=g^{(\g)}\,x^{\lambda}(\hat c_0^{(\g)}(x)+O(\lL^{-\al_i^*}))$. 
Expanding $g^{(\g)}=1+O(\lL^{-\alpha_a^*})$, we see that the leading order term
in $g^{(\g)}\,x^{\lambda} \hat p_\lambda^{(\g)}(x)$ has only contributions from the action of the $g_T$ parts of $g^{(\g)}$ on the leading term $x^\lambda$. The total contribution of the $g_T$ terms is $g_T^{t_1}$, as directly read off \eqref{variousg}. 
Finally noting that $(g_T)^{t_1} \,x^\lambda = \gamma(x)^{t_1} \, x^\lambda T^{t_1\xi}(g_T)^{t_1}$, where $x=q^\xi$ (using \eqref{gtexch}), we find that the leading term is $\tau_0(x)=\gamma(x)^{t_1}\hat c_0^{(\g)}(x)$,
and the Theorem follows from the relation $\tau_0(x)/\hat c_0^{(\g)}(x)=\gamma(x)^{t_1}$.
\end{proof}

\begin{remark}\label{newajim} As noted in Remark \ref{ajim}, in the cases $\g=D_{N+1}^{(2)},A_{2N-1}^{(2)},D_{N}^{(1)}$,
the relation \eqref{ftga} boils down to the recursion relation of \cite{FJMMfermio}
(Theorem 3.1) for the coefficients $J_\beta^\mu=J_\beta(x=q^\mu)$ for the root systems $R=B_N, C_N, D_N $ respectively (see Table \ref{korspec}). These are the coefficients in the series expansion
$\tilde \Pi^{(\g)}_\lambda(x)=x^\lambda g^{(\g)}_\Lambda\,\Pi^{(\g)}_\lambda(x)=x^\lambda\sum_{\beta \in Q_+}J_\beta(x)\Lambda^{-\beta} $, up to a rescaling $q\to q^2$ in the cases $\g=D_N^{(1)},D_{N+1}^{(2)}$. 
Here we use the notation
$$g^{(\g)}_\Lambda = g_\Lambda \times \left\{ \begin{matrix} &1/\left(\frac{1}{\Lambda_{N}};q^{1/2}\right)_\infty & {\rm for}\ \g=D_{N+1}^{(2)},\,R=B_N\\
&1/\left(\frac{1}{\Lambda_{N}^2};q^2\right)_\infty & {\rm for}\ \g=A_{2N-1}^{(2)},\,R=C_N\\
&1/\left(\frac{1}{\Lambda_{N-1}\Lambda_N};q\right)_\infty & {\rm for}\ \g=D_{N}^{(1)},\,R=D_N
\end{matrix}\right. .$$
\end{remark}

\begin{cor}\label{integrability}
The Pieri operators $H_a^{(\g)}(\lL)$, $a=1,2,...,N$ of Eq. \eqref{pierilimq} are algebraically independent conserved quantities of the $\g$-quantum \barQ-systems.
\end{cor}
\begin{proof}
To see that the time translation operator $g^{(\g)}(\lL)$  commutes with all higher Pieri operators $H_a^{(\g)}(\lL)$, $a\in[1,N]$, 
act with $(g^{(\g)}(\lL))^{-1}\gamma(x)^{t_1}$ on
the Pieri equations, and use Theorem \ref{lemtwo} 
$$g^{(\g)}(\lL)^{-1}\gamma(x)^{t_1}\,H_a^{(\g)}(\lL)\, \Pi_\lambda^{(\g)}(x)=g^{(\g)}(\lL)^{-1}H_a^{(\g)}(\lL)\,g^{(\g)}(\lL)\,\Pi_\lambda^{(\g)}(x)
={\hat e}_a^{(\g)}(x)\, \Pi_\lambda^{(\g)}(x),$$
which implies that $g^{(\g)}(\lL)^{-1}H_a^{(\g)}(\lL)\,g^{(\g)}(\lL)=H_a^{(\g)}(\lL)$ by the definition of the Pieri operators. The statement follows by noting that any dependence between $H_a^{(\g)}(\lL)$ would imply a dependence between $\hat e_a^{(\g)}(x)$, which are independent. 
\end{proof}

\subsection{Proof of  Theorems \ref{Qsysconj} and \ref{raiseconj}}\label{sec:43}
The proof of both Theorems relies on the following.
\begin{thm} \label{foutranthm}
For all $\g=D_N^{(1)},B_N^{(1)},C_N^{(1)},D_{N+1}^{(2)},A_{2N-1}^{(2)},A_{2N}^{(2)}$ we have the relation
\begin{equation}\label{identcrucial}
D_{a;n}^{(\g)}(x)\,\Pi_\lambda^{(\g)}(x)=\barD_{a;n}^{(\g)}(\lL)\,\Pi_\lambda^{(\g)}(x),\qquad a\in[1,N], n\in\Z
\end{equation}
valid for any $\g$-partition $\lambda$.
\end{thm}
Theorem \ref{Qsysconj} follows from Theorem \ref{foutranthm}, as any relation satisfied by the difference operators $\{\barD_{a;n}\}$ implies the opposite relation for the difference operators $\{ D_{a;n}\}$. Theorem \ref{raiseconj} is the
particular case of \eqref{identcrucial} with $n=0,1$. We provide a proof of  \eqref{identcrucial} for long labels, short labels except for $B_N^{(1)}$,
and for the short label of $B_N^{(1)}$ separately.

\subsubsection{Long labels and even time short labels}

In this case, the proof is similar to that in type $A_{N-1}^{(1)}$. Acting on both sides of the eigenvalue equation \eqref{eigenD}, which we write as \eqref{identcrucial} with $n=0$,
with $\gamma(x)^{-t_1n}\,(g^{(\g)}(\lL))^n$ and using Theorem \ref{lemtwo},
\begin{eqnarray*}
&& \gamma^{-t_1n}\,D_{a;0}^{(\g)}(x) \, (g^{(\g)})^n\,\Pi_\lambda^{(\g)}(x)=\gamma^{-t_1n}\,D_{a,0}^{(\g)}(x) \,\gamma^{{t_1n}} \, 
\Pi_\lambda^{(\g)}(x)
=q^{\frac{\Lambda_{a,a}t_a n}{2}}\,D_{a;nt_a}^{(\g)}(x)\,\Pi_\lambda^{(\g)}(x)\\
&&=(g^{(\g)})^n\,\barD_{a;0}^{(\g)}(\lL)\,\gamma^{-t_1n}\, \Pi_\lambda^{(\g)}(x)
=(g^{(\g)})^n\,\barD_{a;0}^{(\g)}(\lL)\,(g^{(\g)})^{-n}\,\Pi_\lambda^{(\g)}(x)= q^{\frac{\Lambda_{a,a}t_a n}{2}}\,\barD_{a;nt_a}^{(\g)}\,\Pi_\lambda^{(\g)}(x),
\end{eqnarray*}
where we have used Definition \ref{moredefn}. Therefore, Eq. \eqref{identcrucial} holds for $n$ a multiple of $t_a$.

\subsubsection{Short labels, odd $n$ for $\g=C_N^{(1)}, A_{2N}^{(2)}$}
In both cases,  $a=1$ is a short label, $t_1=2$.
\begin{lemma}\label{firstlemCAe}
For $\g=C_N^{(1)},A_{2N}^{(2)}$ and for any $\g$-partition $\lambda$,
$$\Big(D_{1,1}^{(\g)}(x)-\barD_{1,1}^{(\g)}(\lL)\Big)\,\Pi_\lambda^{(\g)}(x)=q^{-1} \Big(D_{1,-1}^{(\g)}(x)-\barD_{1,-1}^{(\g)}(\lL)\Big)\,\Pi_\lambda^{(\g)}(x).$$
\end{lemma}
\begin{proof}
Applying  Equation \ref{Gid} with $n=0$, acting on
$\Pi_\lambda^{(\g)}(x)$,
\begin{eqnarray}
&&[D_{1,0}^{(\g)}(x),{\hat e}_1(x)]\,\Pi_\lambda^{(\g)}(x)=
\left(D_{1,0}^{(\g)}(x)H_1^{(\g)}(\lL)-{\hat e}_1(x)\barD_{1,0}^{(\g)}(\lL)\right)\,\Pi_\lambda^{(\g)}(x)\nonumber\\
&&\quad =\left[H_1^{(\g)}(\lL),\barD_{1,0}^{(\g)}(\lL)\right]\,\Pi_\lambda^{(\g)}(x)=(q-1)\left( D_{1,1}^{(\g)}(x) -q^{-1}\, D_{1,-1}^{(\g)}(x)\right)\, \Pi_\lambda^{(\g)}(x),
\label{tobecomp} 
\end{eqnarray}
where we used the  Pieri rule \eqref{pieriwhit}, and the eigenvalue equation \eqref{identcrucial} with $n=0, a=1$.
The only terms in $H_1^{(\g)}(\lL)$ of Equations (\ref{Hamil}-\ref{Hamil2}) which fail to commute with $\barD_{1,0}^{(\g)}(\lL)=\lL_1$ 
are $T_1+\left(1-\frac{\lL_2}{\lL_1}\right)T_1^{-1}$ if $N\geq 2$. Therefore,
\begin{eqnarray*}[H_1^{(\g)}(\lL),\barD_{1,0}^{(\g)}(\lL)]&=&[T_1+\left(1-\frac{\lL_2}{\lL_1}\right)T_1^{-1},\lL_1]= (q-1) \,\lL_1 \,T_1+(q^{-1}-1) (\lL_1-\lL_2)T_1^{-1}\\
&=&(q-1)\barD_{1,1}^{(\g)}(\lL)+(q^{-1}-1)\barD_{1,-1}^{(\g)}(\lL),
\end{eqnarray*}
where used $\barD_{1,-1}^{(\g)}(\lL)=(\lL_1-\lL_2)T_1^{-1}$. The lemma follows.
\end{proof}

Defining $\Sigma_n^{(\g)} =(D_{1,n}^{(\g)}(x)-\barD_{1,n}^{(\g)}(\lL))\,\Pi_\lambda^{(\g)}(x)$, Lemma \ref{firstlemCAe} says that
$\Sigma_1^{(\g)}=q^{-1}\,\Sigma_{-1}^{(\g)}$.
\begin{lemma}\label{seclemCAe}
For $\g=C_N^{(1)},A_{2N}^{(2)}$, we have:
$$H_1^{(\g)}(\lL)\, \Sigma_i^{(\g)}= {\hat e}_1(x) \, \Sigma_i^{(\g)}, \qquad (i=\pm 1). $$
\end{lemma}
\begin{proof}
As $\Sigma_{-1}^{(\g)}=q\, \Sigma_1^{(\g)}$, we may restrict ourselves to $i=1$.
\begin{eqnarray}
H_1^{(\g)}(\lL)\, \Sigma_1^{(\g)}&=&\Big( D_{1,1}^{(\g)}(x)\,H_1^{(\g)}(\lL)\, -H_1^{(\g)}(\lL)\,\barD_{1,1}^{(\g)}(\lL)\Big)\,\Pi_\lambda^{(\g)}(x)\nonumber \\
&=&\Big\{ {\hat e}_1(x)\,D_{1,1}^{(\g)}(x)+(q-1) \big(D_{1,2}^{(\g)}(x)-q^{-1}D_{1,0}^{(\g)}(x)\big)  \nonumber \\
&&\qquad \quad -\barD_{1,1}^{(\g)}(\lL)\,H_1^{(\g)}(\lL) -(q-1)\big(\barD_{1,2}^{(\g)}(\lL)-q^{-1}\barD_{1,0}^{(\g)}(\lL)\big) \Big\}\,\Pi_\lambda^{(\g)}(x)\nonumber \\
&=& \Big\{ {\hat e}_1(x)\,D_{1,1}^{(\g)}(x) -\barD_{1,1}^{(\g)}(\lL)\,H_1^{(\g)}(\lL)\Big\}\,\Pi_\lambda^{(\g)}(x)={\hat e}_1(x) \,\Sigma_1^{(\g)}, \label{tobecont}
\end{eqnarray}
where we used Lemma \ref{comuonelem},  as well as the $a=1$, $n=0,2$ cases of 
\eqref{identcrucial} proven above.
\end{proof}
Lemma \ref{firstlemCAe}, together with a uniqueness argument (Lemma \ref{uniqueglem}), implies that
there exists a function $\al^{(\g)}(x)$ 
such that $\Sigma_1^{(\g)}=\al^{(\g)}(x) \,\Pi_\lambda^{(\g)}(x)$, $\Sigma_{-1}^{(\g)}=q\,\al^{(\g)}(x) \,\Pi_\lambda^{(\g)}(x)$:
$$\{ D_{1,1}^{(\g)}(x)-{\barD}_{1,1}^{(\g)}(\lL)-\al^{(\g)}(x)\}\,\Pi_\lambda^{(\g)}(x)=\{ D_{1,-1}^{(\g)}(x)-\barD_{1,-1}^{(\g)}(\lL)-q\,\al^{(\g)}(x)\}\,\Pi_\lambda^{(\g)}(x)=0.$$
Using this equation with $\lambda=0$, 
 $D_{1,-1}^{(\g)}(x)\, \Pi_0^{(\g)}(x)=D_{1,-1}^{(\g)}(x)\, 1=0$, using  \eqref{minusone}.
Similarly,
$\barD_{1,-1}^{(\g)}(\lL)\,\Pi_\emptyset^{(\g)}(x)=(\lL_1-\lL_2)T_1^{-1}\, \Pi_\lambda^{(\g)}(x)\vert_{\lambda=0}=0$, since the prefactor vanishes. 
Therefore, $\al^{(\g)}(x)=0$ and $\Sigma_1^{(\g)}=\Sigma_{-1}^{(\g)}=0$. 
Multiplying $\Sigma_{-1}^{(\g)}$ by $(g^{(\g)})^n \gamma^{-2 n}$, we conclude that for $\g=C_N^{(1)},A_{2N}^{(2)}$, $a=1$, $n\in \Z$ and any integer partition $\lambda$, Equation \eqref{identcrucial} holds.

This result can be extended to all short labels for these algebras, as follows.
\begin{lemma}\label{shortoddthm}
Equation \eqref{identcrucial} holds for all short labels $a$ and $n$ odd in the case of  $\g=C_N^{(1)},A_{2N}^{(2)}$.
\end{lemma}
\begin{proof}
Let $N_\g$ be as in Equation \eqref{NGdef}. Then
if $a\leq N_\g-1$, the relations of the quantum \barQ-system are identical to those of type $A_{N_\g-1}^{(1)}$, and hence we can apply
Theorem 2.8 of \cite{DFK16},  which states that any solution of the relations for $a=1,2,...,N_\g-1$
$$ q^a\,\cQ_{a;n+1}\,\cQ_{a;n-1}=(\cQ_{a;n})^2-\cQ_{a+1;n}\cQ_{a-1;n},\qquad \cQ_{a;n}\cQ_{b;n+1}=q^{{\min}(a,b)}\,\cQ_{b;n+1}\,\cQ_{a;n},$$
satisfies 
$$(-1)^a (q-1) \cQ_{a;n}=[\cQ_{1;n-a+1},\cQ_{a-1;n+1}]_{q^a},\qquad (a=1,2,...,N_\g). $$
Therefore the $\g$-quantum \barQ-system solution $\barD_{a;n}^{(\g)}$ obeys the opposite relations
$$(-1)^a (q-1) \,\barD_{a;n}^{(\g)}=[\barD_{a-1;n+1}^{(\g)},\barD_{1;n-a+1}^{(\g)}]_{q^a}.$$
Using this and Definition \ref{cadefs},
we deduce that
\begin{eqnarray*}&&\{D_{a;2n-1}^{(\g)}(x)-\barD_{a;2n-1}^{(\g)}(\lL)\}\, \Pi_\lambda^{(\g)}(x)\\
&&\qquad\qquad = \frac{(-1)^a}{q-1} \{[D_{1;2n-a}^{(\g)}(x),D_{a-1;2n}^{(\g)}(x)]_{q^a}-[\barD_{a-1;2n}^{(\g)}(\lL),\barD_{1;2n-a}^{(\g)}(\lL)]_{q^a}\}\,\Pi_\lambda^{(\g)}(x)=0,
\end{eqnarray*}
by use of Eq. \eqref{identcrucial} for $D_{a-1;2n}^{(\g)}$ and $D_{1;2n-a}^{(\g)}$ with $a$ even 
and for
$D_{1;2n-a}^{(\g)}$ when $a$ is odd. 
\end{proof}

\subsubsection{The case of $B_N^{(1)}$ for odd $n$}
\label{bnproofsec}
It remains to prove \eqref{identcrucial} 
for $\g=B_N^{(1)}$, $a=N$, $n$ odd. The methods above are inapplicable, and instead we use expression \eqref{defDBNminusone} for $D_{N;2n-1}^{(B_N^{(1)})}$ in terms of Rains operators. 

Recall that  $B_N^{(1)}$-partitions are both integer and half-integer partitions, due to the spin representation with highest weight $\omega_N$ and character  $s_{\omega_N}(x)=\prod_{i=1}^N\frac{1+x_i}{\sqrt{x_i}}$.
Denote
\begin{equation}\label{special}
P_\lambda= P_\lambda^{(B_N^{(1)})}=P_\lambda^{(t,-1,q^\half,-q^\half)}, \qquad
\tilde P_\lambda= P_\lambda^{(t,-q,q^\half,-q^\half)} .
\end{equation}
There is a factorization of Macdonald polynomials \cite{vandiej},
\begin{equation}\label{factopol}
P_{\lambda+\omega_N}(x)=s_{\omega_N}(x)\, \tilde P_{\lambda}(x) , \qquad \hbox{$\lambda$ integer partition}.
\end{equation}
The parameter specialization \eqref{special} for $\tilde P_\lambda$ is obtained by conjugating 
${\mathcal D}_1^{(B_N^{(1)})}(x,q,t)$ by $s_{\omega_N}^{-1}$ and identifying the resulting parameters $(a,b,c,d)$.
We  must therefore consider the action of the difference operators on both functions 
$\Pi_\lambda=\lim_{t\to\infty}P_\lambda$ and $\tilde \Pi_\lambda=\lim_{t\to\infty} \tilde P_\lambda$ with $\lambda$ integer partitions.

Moreover, the Rains operators map the eigenfunctions of the $B_{N}^{(1)}$-type difference 
operators to those of
$B_N^{(1)\,'}$-type, corresponding to different parameters. We denote the corresponding $q$-Whittaker limit of the eigenfunctions functions as $\Pi_\lambda'$. To these parameters there corresponds a different quantization of the Q-system, a
time-translation operator $g'(\lL)$, and Hamiltonians.

Specializing to the $B_N^{(1)}$ parameters $(a,b,c,d)=(t,-1,q^\half,-q^\half)$ and taking the $t\to\infty$ limit as in \eqref{rainsB}, 
\eqref{rainskor} become
\begin{eqnarray}
R_{N}^{(0)}(x)\, \Pi_\lambda&=&(\lL_1\lL_2\cdots \lL_N)^\half \left(1+\frac{1}{\lL_N}\right) \, \Pi_\lambda',\nonumber \\
R_{N}^{(1)}(x)\, \Pi_\lambda'&=&(\lL_1\lL_2\cdots \lL_N)^\half\, \Pi_\lambda, \label{backnforthB}
\end{eqnarray}
for integer partitions $\lambda$. 
We identify the $B_N^{(1)\ '}$ specialization to be $(a,b,c,d)=(t q^\half,-q^\half,1,-1)$,
$$\Pi_\lambda'= \Pi_\lambda^{(B_N^{(1)\ '})}=\lim_{t\to\infty} P_\lambda^{(t q^\half,-q^\half,1,-1)}.
$$

\begin{lemma}
Equations \eqref{backnforthB} also hold for half-integer partitions.
\end{lemma}
\begin{proof}
Using the specialization $(a,b,c,d)=(t,-q,q^\half,-q^\half)$ as in \eqref{special},
in the limit $t\to\infty$ \eqref{rainskor} becomes
$$
\tilde R^{(0)}_N\, \tilde \Pi_\lambda=q^{\frac{N}{4}}\, (\lL_1\lL_2\cdots \lL_N)^\half \left( 1+\frac{q^{-\half}}{\lL_N}\right)\, \tilde \Pi_\lambda',\ \
\tilde R^{(1)}_N\, \tilde \Pi_\lambda' =q^{\frac{N}{4}}\, (\lL_1\lL_2\cdots \lL_N)^\half \, \tilde \Pi_\lambda,
$$
where 
\begin{eqnarray*}
\tilde R^{(0)}_N&=&\lim_{t\to\infty} t^{-N(N-1)/2} \,q^{-\frac{N}{4}} \mathcal R^{(1,-q^\half)}_N= s_{\omega_N}^{-1}\, R_N^{(0)} \,s_{\omega_N},\\
\tilde R^{(1)}_N&=&\lim_{t\to\infty} t^{-N(N+1)/2}\, q^{-\frac{N}{4}} \mathcal R^{(t,-q^\half)}_N=s_{\omega_N}^{-1}\, R_N^{(1)} \,s_{\omega_N},
\end{eqnarray*}
by use of the limit of \eqref{factopol}. Here,
$R_N^{(0)}$ and $R_N^{(1)}$ are as in \eqref{rainsB}. Restoring the factors $s_{\omega_N}$, 
\begin{eqnarray*}
R_{N}^{(0)}(x)\,s_{\omega_N}\,\tilde \Pi_\lambda&=&(\tilde \lL_1\tilde \lL_2\cdots \tilde\lL_N')^\half \left(1+\frac{1}{\tilde\lL_N}\right) \,s_{\omega_N}\, 
\tilde\Pi_\lambda',\nonumber \\
R_{N}^{(1)}(x)\,s_{\omega_N}\,\tilde \Pi_\lambda'&=&(\tilde\lL_1\tilde\lL_2\cdots \tilde\lL_N)^\half\,s_{\omega_N}\, 
 \tilde\Pi_\lambda, 
\end{eqnarray*}
where $\tilde\lL=q^{\omega_N}\lL=q^{\omega_N+\lambda}$. Combining this with the limit of \eqref{factopol} implies that \eqref{backnforthB}  is satisfied for half-integer partitions $\tilde\lambda=\omega_N+\lambda$ as well as integer partitions.
\end{proof}

Combining equations \eqref{backnforthB} leads to the eigenvalue equation for $\hat D^{(B_N^{(1)})}(x)$ of \eqref{rainsB}
$$\hat D_N^{(B_N^{(1)})}(x)\, \Pi_\lambda(x)=R_{N}^{(1)}(x)\,R_{N}^{(0)}(x)\, \Pi_\lambda=\lL_1\lL_2\cdots \lL_N \left(1+\frac{1}{\lL_N}\right) \, \Pi_\lambda ,$$
which is consistent with the relation $\hat D_N^{(B_N^{(1)})}=R_{N}^{(1)}\,R_{N}^{(0)}=D_{N,0}+D_{N-1,0}$ and the eigenvalues $\barD_{a,0}$ in  Definition \eqref{firstdefops}.

The $q$-Whittaker limit of the difference operators for the ${B_N^{(1)}}'$ parameters are obtained as in \eqref{DlimW}, and are
denoted by $D_{a}'(x)= D_{a}^{(B_N^{(1)\, '})}(x)$, $a=1,2,...,N$. The corresponding limit 
of the Rains operator \eqref{rainsop} is
$\hat{D}_N'(x)= R_N^{(0)}(x)\, R_N^{(1)}(x) =D_N'(x)+D_{N-1}'(x)$. 
This identification follows from the $t\to \infty$ limit of Lemma \ref{rainstokor}, and from the eigenvalues,
which are identical to those of the $B_N^{(1)}$ case, since $\xi_\g=\xi_{\g'}$. 
As in the $B_N^{(1)}$ case, 
We define the generalized Macdonald operators
\begin{eqnarray}
D_{a;n}'(x)&=&q^{-na/2}\, \gamma^{-n}\, D_{a}'(x)\, \gamma^{n}, \qquad (a=1,2,...,N-1),\nonumber \\
D_{N;2n}'(x)&=&q^{-nN/2}\, \gamma^{-n}\, D_{N}'(x)\, \gamma^{n},\nonumber \\
D_{N;2n+1}'(x)&=&q^{-nN/4}\, R_{N;n+1}^{(0)}(x)\, R_{N;n}^{(1)}(x)=q^{-nN/2}\, \gamma^{-n}\, D_{N;1}'(x)\, \gamma^{n}.
\label{defBNprime}
\end{eqnarray}
Note the reversal of the order in the last product and the different time index compared to \eqref{defDBNminusone}. 
The operators $D_{a;n}'$ will be shown to satisfy the set of recursion relations below.
\begin{defn}\label{bprimeqsys}
The type ${B_N^{(1)}}'$-quantum Q-system relations\footnote{This ``quantum Q-system" is new, and
we call it a Q-system by analogy with the other $\g$ cases. 
}
are the same as the type $B_N^{(1)}$ relations in \eqref{qsys1}--\eqref{qsys3} except for the two equations with $a=N$:
\begin{eqnarray*}
q^{N/2} \Q_{N;2n+2}\,\Q_{N;2n}&=&\Q_{N;2n+1}^2-q^{\frac{N-1}{2}-n} \,\Q_{N-1;n+1}\,\Q_{N-1;n},\\
q^{N/2} \Q_{N;2n+1}\,\Q_{N;2n-1}&=&(\Q_{N;2n}-q^{-\frac{n}{2}} \,\Q_{N-1;n})(\Q_{N;2n}+q^{\frac{1-n}{2}} \,\Q_{N-1;n}).
\end{eqnarray*}
\end{defn}
\noindent In particular, the q-commutation relations are as in type $B_{N}^{(1)}$ \eqref{qsys1}.

The eigenvalue equation for  $D_a'$ corresponding to any $\lambda$ is
$D_{a;0}'(x)\, \Pi_\lambda'= \lL^{\omega_a^*} \, \Pi_\lambda'$,
with $\omega_a^*$  a fundamental weight of  type $C_N$. Therefore, 
the candidate Fourier transforms $\barD_{a,n}'(\Lambda)$ are defined  so that they satisfy the opposite quantum Q-system to that of Definition \ref{bprimeqsys}, subject to the same initial data as in type $B_N^{(1)}$,
$\barD'_{a;0}=\lL^{\omega_a^*}$ and
$\barD_{a;1}'=\lL^{\omega_a^*}T^{\omega_a}$,
$\barD_{N;1}'=\lL^{\omega_N^*}T^{\omega_N}$.
Together with the recursions of Def. \ref{bprimeqsys}, this determines $\barD_{a;n}'$ for all $a,n$.
In particular,
\begin{eqnarray*}
\barD_{a;-1}'&=&\barD_{a;0}'\left( 1-\frac{\lL_{a+1}}{\lL_a}\right) T^{-\omega_a} \qquad (a=1,2,...,N-1),\\
\barD_{N;-1}'&=&\barD_{N;0}'\left(1-\frac{1}{\lL_N}\right)\left(1+\frac{q^\half}{\lL_N}\right) T^{-\omega_N}, \ \ 
\barD_{N;2}'
=\barD_{N;1}' \left(1-q^\half\frac{1}{\lL_N^2}\frac{1}{T_{N} }\right)T^{\omega_N}.
\end{eqnarray*}
As before, these are sufficient to determine the form of the time-translation operator $g'(\Lambda)$.
\begin{lemma}
The operator
\begin{eqnarray}
g'&=&g_T^\half \, g_{\lL_N}' \, g_T^\half \, g_{\lL_N}'' \, g_\lL ,\label{gprime}\\
g_{\lL_N}' &=&
\frac{1}{(q^\half\lL_N^{-2};q)_\infty}
,\quad g_{\lL_N}'' =
\frac{1}{(\lL_N^{-1};q^\half)_\infty \, (-q^\half \lL_N^{-1};q^\half)_\infty},
\nonumber
\end{eqnarray}
where $g_T$ is as in \eqref{gtvalue} and  $g_\lL$ as in \eqref{glambdavalue},
is the time translation operator for the $B_N^{(1) \, '}$ opposite quantum Q-system. That is, for all $n\in \Z$,
\begin{eqnarray*} \barD_{a;n}'&=& q^{-an/2} \, (g')^n\, \barD_{a}'\, (g')^{-n} \qquad (a=1,2,...,N-1) ,\\
\barD_{N;2n+i}'&=& q^{-Nn/2} (g')^n\, \barD_{N;i}'\, (g')^{-n} ,\quad i=0,1.
\end{eqnarray*}
\end{lemma}

The Pieri rules for $B_N^{(1)\,'}$ are obtained by duality. Using Theorem \ref{korpier} with
$(a,b,c,d)=(1,-1,t q^\half,-q^\half)$, $(a^*,b^*,c^*,d^*)=(t^\half,-t^{-\half},t^\half q^\half,-t^{-\half} q^\half)$, the $q$-Whittaker limit of the first Pieri operator is
\begin{eqnarray*}&&H_1^{(B_N^{(1)\,'})}(\lL) =T_1+\sum_{a=2}^N \left(1-\frac{\lL_{a}}{\lL_{a-1}}\right) T_a +\sum_{a=1}^{N-1} \left(1-\frac{\lL_{a+1}}{\lL_{a}}\right) T_a^{-1}\\
&&+\left(1-\frac{1}{\lL_{N}}\right)\left(1+\frac{q}{\lL_{N}}\right)\left(1-\frac{q}{\lL_{N}^2}\right) T_N^{-1}+ \frac{q^\half-q^{-\half}}{\lL_N} -\frac{q^\half+q^{-\half}}{\lL_N^2}+\frac{q^{-\half}}{\lL_{N-1}\lL_N}.
\end{eqnarray*}
By direct calculation, this operator commutes with $g'(\Lambda)$ and as a consequence
\begin{equation}\label{cmhamprime}
g'(\lL)\, \Pi_\lambda'(x)=\gamma(x)\, \Pi_\lambda'(x)\ .
\end{equation}

We want to show that $D_{N;1}(x)\Pi_\lambda= \barD_{N;1}(\Lambda)\Pi_\lambda$ (case $a=N,n=1$ of \eqref{identcrucial}). Using \ref{lemtwo}, Definition \eqref{defDBNminusone} and the relations \eqref{backnforthB}, we find
\begin{eqnarray}
D_{N;1}(x)\, \Pi_\lambda&=& q^{-\frac{N}{4}}\,R_{N;0}^{(1)} \,R_{N;1}^{(0)} \, \Pi_\lambda
\nonumber\\
&=&q^{-\frac{3N}{8}}\,g\, (\lL_1\cdots \lL_N)^\half \left(1+\frac{1}{\lL_N}\right)\, (g')^{-1} \,(\lL_1\cdots \lL_N)^\half \, \Pi_\lambda\label{presque}.
\end{eqnarray}
Recall the notation
$h_{\al_N^*}=\prod_{n=0}^{\infty} \left( 1-\frac{q^n}{\lL_N^2}\right)^{-1}$, then we have the relations
$\left(1+\frac{1}{\lL_N}\right)(g_{\lL_N}'')^{-1}=h_{\al_N^*}^{-1}$
and $T_N^{\frac{1}{4}} (g_{\lL_N}')^{-1}=h_{\al_N^*}^{-1}T_N^{\frac{1}{4}}$. These allow to compute
\begin{eqnarray}&&\!\!\!\!\!\!\!\!\!\!\!\!\!\!\!\!\!\!\!\!g\, (\lL_1\cdots \lL_N)^\half \left(1+\frac{1}{\lL_N}\right) \,(g')^{-1}\,(\lL_1\cdots \lL_N)^\half\nonumber \\
&=& g_T^\half \, h_{\al_N^*}\,g_T^\half\, h_{\al_N^*}\,g_\lL(\lL_1\cdots \lL_N)^\half \!\!\left(1+\frac{1}{\lL_N}\right)
g_\lL^{-1}\,(g_{\lL_N}'')^{-1} g_T^{-\half}\,(g_{\lL_N}')^{-1}g_T^{-\half}\,(\lL_1\cdots \lL_N)^\half\nonumber \\
&=& g_T^\half \, h_{\al_N^*}\,g_T^\half\, (\lL_1\cdots \lL_N)^\half 
g_T^{-\half}\,(g_{\lL_N}')^{-1}g_T^{-\half}\,(\lL_1\cdots \lL_N)^\half\nonumber \\
&=& q^{\frac{N}{16}} g_T^\half \, h_{\al_N^*}\, (\lL_1\cdots \lL_N)^\half(T_1\cdots T_N)^{\frac{1}{4}}\,(g_{\lL_N}')^{-1}\,g_T^{-\half}\,(\lL_1\cdots \lL_N)^\half\nonumber \\
&=& q^{\frac{N}{16}} g_T^\half \,  (\lL_1\cdots \lL_N)^\half(T_1\cdots T_N)^{\frac{1}{4}}\,g_T^{-\half}\,(\lL_1\cdots \lL_N)^\half\nonumber \\
&=&q^{\frac{N}{8}} (\lL_1\cdots \lL_N)^\half\, (T_1\cdots T_N)^{\half}\,(\lL_1\cdots \lL_N)^\half\nonumber \\
\qquad &=&q^{\frac{3N}{8}}\,\lL_1\cdots \lL_N\, (T_1\cdots T_N)^\half=q^{\frac{3N}{8}}\, \barD_{N;1}. \label{cestca}
\end{eqnarray}
Combining this with \eqref{presque} results in the relation
$D_{N;1}(x)\Pi_\lambda= \barD_{N;1}(\Lambda)\Pi_\lambda$. Multiplying by $\gamma^{-n}g^n$ and using \eqref{cmhamprime}:
\begin{equation}\label{cestlafin}
D_{N;2n+1}^{(B_N^{(1)})}(x)\, \Pi_\lambda^{(B_N^{(1)})}(x)=\barD_{N;2n+1}^{(B_N^{(1)})}(\lL)\, \Pi_\lambda^{(B_N^{(1)})}(x),\qquad (n\in \Z).
\end{equation}
This completes the proof of \eqref{identcrucial}, and Theorems \ref{Qsysconj} and \ref{raiseconj} follow.

\begin{cor}
The $D'$-operators satisfy the quantum $B_N^{(1)\,'}$-quantum Q-system of Definition \ref{bprimeqsys}. Moreover we have the raising
operator conditions for any $B_N^{(1)}$-partition $\lambda$:
$$ D_{a;1}'(x)\,\Pi_\lambda'= \lL^{\omega_a^*}\, \Pi_{\lambda+\omega_a}', \qquad (a\in [1,N]),
$$
$\omega_a,\omega_a^*$ the fundamental weights of $B_N,C_N$.
\end{cor}
\begin{proof}
Starting from the eigenvalue equations
$D'_{a}(x)\, \Pi_\lambda'=\barD_{a}'\, \Pi_\lambda'$ for $a=1,2,...,N$ and 
multiplying with $\gamma^{-n}(g')^n$, and still denoting $\bar X$ the prime Fourier transform of $X$,
$$ D'_{a;t_a n}(x) \Pi_\lambda'=q^{-na/2} \gamma^{-n}\,D'_{a}(x)\,\gamma^n \,\Pi_\lambda'=
q^{-na/2} (g')^n\,\barD'_{a}\,(g')^{-n} \,\Pi_\lambda'=\barD'_{a;t_a n}\, \Pi _\lambda'.$$
For the short label $a=N$, using \eqref{defBNprime},  \eqref{backnforthB}, \eqref{cestca} and $\barD_{N;1}=\barD_{N;1}'=\lL^{\omega_N^*}T^{\omega_N}$:
\begin{eqnarray*}
D'_{N;1}(x)\,\Pi _\lambda' &=&R_{N;1}^{(0)}\, R_{N}^{(1)}\, \Pi _\lambda' =q^{-\frac{N}{8}}\,(\lL_1\cdots \lL_N)^{\half}g\, (\lL_1\cdots \lL_N)^\half \left(1+\frac{1}{\lL_N}\right) \,(g')^{-1}\, \Pi_\lambda
\\
&=&q^{-\frac{N}{8}}\, (\lL_1\cdots \lL_N)^{\half} \,q^{\frac{3N}{8}}\,\bar D_{N;1}\,(\lL_1\cdots \lL_N)^{-\half}\,\Pi_\lambda
=\barD_{N;1}'\,\Pi_\lambda'.
\end{eqnarray*}
The relation is easily extended by multiplication with 
$\gamma^{-n}(g')^n$, resulting in $D'_{N;2n+1}(x)\,\Pi _\lambda' =\barD_{N;2n+1}'\,\Pi_\lambda'$ for all $n\in \Z$. 
The Corollary follows.

\end{proof}

\section{Universal Solutions}
\label{sec:universal}

The relations between Macdonald eigenvalue equations and Pieri rules are embodied by the duality property of Macdonald polynomials \eqref{duapols}.
We now discuss a reformulation of this duality in terms of universal solutions along the same lines as the case of type $A$,
treated in Section \ref{universAsec}. This allows to re-prove the main results of this paper in terms of universal solutions.

\subsection{Universal solutions}\label{sec:51}

\subsubsection{Universal Koornwinder-Macdonald and $\g$-Macdonald eigenvalue solutions}
As in the case of type A, it is extremely fruitful to think directly of $x$ and $s$
as dual variables, whose roles may be interchanged. 
To this end, we use changes of variables:
$$
x_i=q^{\mu_i}\, t^{{\rho_i}^*}=q^{\mu_i}\, a t^{N-i}, \quad s_i=q^{\lambda_i}\, t^{\rho_i}=q^{\lambda_i}\, \sigma t^{N-i}, \qquad \lambda,\mu\in \C^N,
$$
where we use the notations of Section \ref{secduabcd}, see Equation \eqref{rhosabcd}.
Note that if we write $P_\lambda^{(a,b,c,d)}(x)=x^\lambda p_\lambda^{(a,b,c,d)}(x)$ and substitute this into the 
Koornwinder-Macdonald eigenvalue equation \eqref{eigenKo}, we obtain
$$\left(\frac{(1-t^N)(1+\sigma t^{N-1})}{1-t}-\sigma t^{N-1} \hat e_1(s)+\sum_{i=1\atop\epsilon=\pm 1}^N \Phi_{i,\epsilon}^{(a,b,c,d)}(x)(\lL_i^\epsilon \Gamma_i^\epsilon-1) \right) \, p_\lambda^{(a,b,c,d)}(x;s)=0,$$
where we have conjugated \eqref{MKopone} with $x^{-\lambda}$. Noting that $\Phi_{i,\epsilon}^{(a,b,c,d)}(x)$ can be expanded in series 
of the variables $x^{-\al_i}$, $\al_i$ the $B_N$ simple roots, this suggests what we call a universal solution $P^{(a,b,c,d)}(s;x)$ of the 
Koornwinder-Macdonald eigenvalue equation \eqref{eigenKo}
in the form 
\begin{equation}\label{sumK}
P^{(a,b,c,d)}(x;s)=q^{\lambda\cdot\mu}\,\sum_{\beta\in Q_+} c_\beta^{(a,b,c,d)}(s)\, x^{-\beta},\ \ 
c_0^{(a,b,c,d)}(s)=1,\end{equation}
where $Q_+$ denotes the positive cone of the root lattice of $B_N$.
The  normalizing prefactor for $P^{(a,b,c,d)}(x;s)$ is such that 
$q^{\lambda\cdot\mu}=x^\lambda \, t^{-\rho^*\cdot\lambda} =s^\mu \, t^{-\rho\cdot\mu}$,
and is invariant under the interchange of $\lambda\leftrightarrow \mu$, therefore under $x\leftrightarrow s$.
Equation \eqref{eigenKo} for $P^{(a,b,c,d)}(x;s)$ and generic $s$,
\begin{equation}\label{genericeigen}
{\mathcal D}_1^{(a,b,c,d)}\,P^{(a,b,c,d)}(x;s)=\sigma t^{N-1}\,\hat e_1(s)\,P^{(a,b,c,d)}(x;s),
\end{equation}
is equivalent to a linear triangular, generically nonsingular, system for the
coefficients $c_\beta^{(a,b,c,d)}(s)$, which uniquely fixes them for all $\beta\in Q_+$. The solution $P^{(a,b,c,d)}(x;s)$ is therefore unique,
and we refer to it as {\it the universal Koornwinder-Macdonald solution}.
The normalization ${c}_0^{(a,b,c,d)}(s)$ uniquely fixes the solution, which is otherwise determined up to the overall normalization given by this coefficient.

Specializing the Koornwinder parameters $(a,b,c,d)$ according to Table \ref{korspec}, 
we obtain the universal solutions
$P^{(\g)}(x;s)$ of the first $\g$-Macdonald eigenvalue equations \eqref{macdo}:
\begin{equation}\label{genericg}
{\mathcal D}_1^{(\g)}\, P^{(\g)}(x;s)=t^{\rho_1}\,\hat e_1(s)\, P^{(\g)}(x;s).
\end{equation}
These universal solutions have the expansion
\begin{equation}\label{sumG} P^{(\g)}(x;s)=q^{\mu\cdot \lambda}\, \sum_{\beta\in Q_+(R)} c_\beta^{(\g)}(s) \, x^{-\beta},\ \ 
c_0^{(\g)}(s)=1,
\end{equation}
with the coefficients $c_\beta^{(\g)}(s)$ uniquely determined by \eqref{genericg}, up to the normalization $c_0^{(\g)}(s)$.

\begin{remark}
The apparent discrepancy between $Q_+$
in \eqref{sumK} and $Q_+(R)$ in \eqref{sumG} is just an artifact of the specializations. 
For instance using the $C_N^{(1)}$ specialization $(a,b,c,d)=t^{1/2}(1,-1,q^{1/2},-q^{1/2})$
leads to an operator $\D_1^{(C_N)}(x;q,t)$ which has an expansion as a series of the variables $\{x_{i+1}/x_i\}_{1\leq i\leq N-1}$ and $x_N^{-2}$, as opposed to
the generic case $\D_1^{(a,b,c,d)}(x;q,t)$ which has an expansion as a series of the variables $\{x_{i+1}/x_i\}_{1\leq i\leq N-1}$ and $x_N^{-1}$. This simply means that the series solution of the Macdonald eigenvalue equation at the specific $C_N^{(1)}$ specialization is an {\it even} function of $x_N^{-1}$, i.e. the coefficients of odd powers of $x_N^{-1}$ vanish.
\end{remark}

\begin{remark}\label{accessrem}
Specializing $s=q^\lambda t^\rho$ to $\lambda$ an integer partition, and using the uniqueness of the universal solution, we recover
the Koornwinder polynomial
\begin{equation}\label{unitokorc} 
P^{(a,b,c,d)}(x;s=q^\lambda t^\rho)=t^{-\rho^* \cdot\lambda}\, P_\lambda^{(a,b,c,d)}(x).
\end{equation}
The specialization therefore truncates the series \eqref{sumK} to finitely many terms.
Specializing the Koornwinder parameters to the values in Table \ref{tableone},
\begin{equation}\label{unitomac} 
P^{(\g)}(x;s=q^\lambda t^{\rho})= t^{-\rho^* \cdot\lambda}\, P_\lambda^{(\g)}(x),
\end{equation}
for $\lambda$ an integer partition, $\rho=\rho^{(\g)}$ and $\rho^*=\rho^{(\g^*)}$. Moreover, by uniqueness,
\eqref{unitomac} holds also for $\lambda$ any $\g$-partition, for example the non-integer partitions in the cases $\g=D_N^{(1)},B_N^{(1)},D_{N+1}^{(2)}$.
This presentation is therefore more economical, as a {\it single} universal function $P^{(a,b,c,d)}(x;s)$
contains the information on {\it all} the $\g$-Macdonald polynomials as well.
\end{remark}

Another important specialization of the universal Koornwinder function $P^{(a,b,c,d)}(x;s)$ is the Koornwinder Baker-Akhiezer quasi-polynomials introduced by Chalykh \cite{Chalykh}. The latter correspond to specializing the paramaters $a,b,c,d,t$ to arbitrary negative integer powers of $q$. The effect is again a truncation of the series to finitely many terms.

\subsubsection{Universal Pieri solutions}

We also define the universal solution  $Q^{(a,b,c,d)}(s;x)$ of the first Pieri equation \eqref{PieriK}:
\begin{equation}\label{expieri}
Q^{(a,b,c,d)}(s;x)=q^{\lambda\cdot\mu}\, \sum_{\beta\in Q_+} {\bar c}_\beta^{(a,b,c,d)}(x)\, s^{-\beta},\ \ {\bar c}_0^{(a,b,c,d)}(x)=1
\end{equation}
subject to:
$${\hat e}_1(x)\, Q^{(a,b,c,d)}(s;x)= \hat{\mathcal H}_1^{(a,b,c,d)}(s;q,t)\, Q^{(a,b,c,d)}(s;x),$$
where
\begin{eqnarray}
\hat{\mathcal  H}_1^{(a,b,c,d)}(s;q,t)&:=&t^{-\rho^*\cdot \lambda} \,{\mathcal H}_1^{(a,b,c,d)}(s;q,t)\, t^{\rho^*\cdot \lambda}\nonumber\\
&=&\frac{1}{a t^{N-1}}\,
\Delta^{(a^*,b^*,c^*,d^*)}(s)^{-1}\, {\mathcal D}_1^{(a^*,b^*,c^*,d^*)}(s;q,t)\,\Delta^{(a^*,b^*,c^*,d^*)}(s).\label{pieriHamabcd}
\end{eqnarray}
The series expansion \eqref{expieri} exists because $\hat{\mathcal H}_1^{(a,b,c,d)}(s;q,t)$
has a series expansion in the variables $s^{-\al_i}$. As above, the normalization 
${\bar c}_0^{(a,b,c,d)}(x)=1$ uniquely fixes the solution.

Specializing the parameters $(a,b,c,d)$ as in Table \ref{korspec}, we have the universal solutions  to 
the first $\g$-Pieri equations \eqref{PieriG}:
\begin{equation}\label{Upieri}
{\hat e}_1(x)\, Q^{(\g)}(s;x)=\hat\cH_1^{(\g)}(s;q,t)\, Q^{(\g)}(s;x),
\end{equation}
where we have used the fact that ${\hat e}_1^{(R)}(x)={\hat e}_1(x)$ for all $R$, and
$\hat\cH_1^{(\g)}(s=\Lambda t^\rho;q,t):=t^{-\rho^*\cdot \lambda}\,\cH_1^{(\g)}(\Lambda;q,t)\,t^{\rho^*\cdot \lambda}$,
with $\cH_1^{(\g)}(\Lambda;q,t)$ as in \eqref{PieriOpG}. The solution $Q^{(\g)}(s;x)$  is a series of the form
$$Q^{(\g)}(s;x) = q^{\mu\cdot\lambda} \sum_{\beta\in Q_+^*} {\bar c}_\beta^{(\g)}(x)\, s^{-\beta},\ \ {\bar c}_0^{(\g)}(x)=1,
$$
where the sum extends over the positive root cone $Q_+^*$ of $R^*$. 
Due to triangularity, the coefficients of the series are uniquely determined by \eqref{Upieri}.

\subsection{Duality}\label{sec:52}

\subsubsection{Duality of universal functions}

The duality of Macdonald polynomials can be extended to the universal functions as follows.

\begin{thm}\label{UKconj}
The functions $Q^{(a,b,c,d)}(s;x)$ and $P^{(a,b,c,d)}(x;s)$ 
are related via
\begin{equation} \label{DUoneabcd}
Q^{(a,b,c,d)}(s;x)=\frac{P^{(a,b,c,d)}(x;s)}{\Delta^{(a,b,c,d)}(x)},
\end{equation}
with $\Delta^{(a,b,c,d)}$ as in \eqref{deltabcd}.
\end{thm}
\begin{proof}
The proof is similar to that of Theorem \ref{macpiedelta}.
The universal solution $P^{(a,b,c,d)}(x;s)$ also obeys 
the Pieri equation \eqref{pieriHamabcd}, as a consequence of the existence of a solution to the bispectral problem \cite{Cherednik95,vDselfdual,Sahi,Chalykh},
i.e. of both Koornwinder eigenvalue and Pieri equations, and of the uniqueness of the solution up to the overall normalization determined by the leading coefficient. Expanding the coefficients 
$c_\beta^{(a,b,c,d)}(s)=\sum_{\delta\in Q_+} c_{\beta,\delta}^{(a,b,c,d)} s^{-\delta}$
allows to write an expansion
$$P^{(a,b,c,d)}(x;s)=q^{\lambda\cdot \mu} \sum_{\delta\in Q_+} \hat c_\delta^{(a,b,c,d)}(x)\, s^{-\delta}, \ \ 
\hat c_\delta^{(a,b,c,d)}(x)=\sum_{\beta\in Q_+} c_{\beta,\delta}^{(a,b,c,d)}x^{-\beta} .$$
Ehe Pieri equation uniquely fixes the coefficients in this expansion, up to the overall factor $\hat c_0^{(a,b,c,d)}(x)$,  hence $P^{(a,b,c,d)}(x;s)=\hat c_0^{(a,b,c,d)}(x)\, Q^{(a,b,c,d)}(s;x)$ as a series in $s^{-\al_i}$. 

To compute $\hat c_0^{(a,b,c,d)}(x)$, we note that it can be extracted as the 
successive limits $s_1\to\infty$, $s_2\to \infty$, ..., $s_N\to \infty$ of $P^{(a,b,c,d)}(x;s)$. Let us examine
the $m$ eigenvalue equation \eqref{eigenvecabcd}, for $P^{(a,b,c,d)}(x;s)=q^{\lambda.\mu}(\hat c_0^{(a,b,c,d)}(x)+O(\lL^{-\al_i}))$. Writing 
$${\mathcal D}_m^{(a,b,c,d)}=d_m^{(a,b,c,d)}(x)+
\sum_{I\subset [1,N]\atop |I|\leq m} \sum_{\epsilon_k=\pm 1\atop k\in I}
d_{I,\epsilon,m}(x) \prod_{k\in I} \Gamma_k^{\epsilon_k},
$$ 
we have
$$\left(d_m^{(a,b,c,d)}(x)-\theta_m\,\hat e_m(x)+\sum_{I\subset [1,N]\atop |I|\leq m} \sum_{\epsilon_k=\pm 1\atop k\in I}
d_{I,\epsilon,m}(x) \prod_{k\in I} \lL_k^{\epsilon_k}\,\Gamma_k^{\epsilon_k}\right) (\hat c_0^{(a,b,c,d)}(x)+O(\lL^{-\al_i}))=0,$$
where $\theta_m=\sigma^m t^{m(N-\frac{m+1}{2})}$. Dividing by $\lL_1\lL_2\cdots \lL_m$, we find
that all the terms have a factor of the form $\lL^{-\beta}$ for some $\beta\in Q_+$, and only those with $\beta=0$
survive in the limit $|\lL_1|>>|\lL_2|>>\cdots >>|\lL_m|>>1$. This leaves us with 
\begin{equation}\label{yeah} \left( \Gamma_1\Gamma_2\cdots \Gamma_m -\frac{\theta_m^2}{d_{[1,m],(+)^m,m}(x)}\right)\hat c_0^{(a,b,c,d)}(x)=0.\end{equation}
Let us compute $d_{[1,m],(+)^m,m}(x)$. From Definition \ref{malphadef}, and the explicit expression of the 
van Diejen operators (\ref{vandiealpha}-\ref{defV}), we easily see that the coefficient of  $\Gamma_1\Gamma_2\cdots \Gamma_m$ in ${\mathcal D}_m^{(a,b,c,d)}$ comes from the highest order
van Diejen operator ${\mathcal V}_m^{(a,b,c,d)}$, and more precisely from the term $s=1$, $J_1=J=\{1,2,...,m\}$
in \eqref{vandiealpha}:
\begin{eqnarray*}
d_{[1,m],(+)^m,m}(x)&=&\prod_{i=1}^m \frac{(1-a x_i)(1-b x_i)(1-c x_i)(1-d x_i)}{(1-x_i^2)(1-q x_i^2)}
\prod_{1\leq i<j\leq m} \frac{1-t x_ix_j}{1-x_ix_j} \frac{1-q t x_ix_j}{1-q x_ix_j}\\
&&\qquad \times \prod_{1\leq i\leq m<j\leq N} \frac{1-t x_ix_j}{1-x_ix_j}  \frac{t x_i-x_j}{x_i-x_j}.
\end{eqnarray*}
Using
\begin{eqnarray*}
\frac{\theta_m^2}{d_{[1,m],(+)^m,m}(x)}&=&
\prod_{i=1}^m \frac{(1-\frac{1}{x_i^2})(1-\frac{1}{q x_i^2})}{(1-\frac{1}{a x_i})(1-\frac{1}{b x_i})(1-\frac{1}{c x_i})(1-\frac{1}{d x_i})}
\prod_{1\leq i<j\leq m} \frac{1-\frac{1}{x_ix_j}}{1-\frac{1}{t x_ix_j}} \frac{1-\frac{1}{q x_ix_j}}{1-\frac{1}{q t x_ix_j}}\\
&&\qquad \times \prod_{1\leq i\leq m<j\leq N} \frac{1-\frac{1}{x_ix_j}}{1-\frac{1}{t x_ix_j}}  \frac{1-\frac{x_j}{x_i}}{1-\frac{x_j}{t x_i}}
\end{eqnarray*}
we conclude that $\hat c_0^{(a,b,c,d)}(x)$ and $\Delta^{(a,b,c,d)}(x)$ of \eqref{deltabcd} both obey 
\eqref{yeah} for $m=1,2,...,N$, so their ratio must be a constant as it is invariant under 
the action of each $\Gamma_i$. This constant is $1$, by noting that 
$q^{-\lambda\cdot\mu}P^{(a,b,c,d)}(x;s)\to 1$
and $\Delta^{(a,b,c,d)}\to 1$ in the limit when all $x^{-\al_i}\to 0$.
The Theorem follows.
\end{proof}

Using the specializations of the parameters $(a,b,c,d)$ as in Table \ref{korspec},
\begin{cor}\label{UGconj}
The functions $Q^{(\g)}(s;x)$ and $P^{(\g)}(x;s)$ 
are related via
\begin{equation} \label{DUoneG}
Q^{(\g)}(s;x)=\frac{P^{(\g)}(x;s)}{\Delta^{(\g)}(x)},
\end{equation}
with $\Delta^{(\g)}$ as in \eqref{deltaG}.
\end{cor}

It is now a simple exercise to relate the universal $(a,b,c,d)$ Koornwinder-Pieri solution to the universal 
solution of the $(a^*,b^*,c^*,d^*)$ Koornwinder-Macdonald eigenvalue equation. As a result we have the following duality relation between universal  Pieri solutions.

\begin{thm}\label{DUAthm}
We have the following duality formulas:
\begin{equation} Q^{(a^*,b^*,c^*,d^*)}(x;s)=Q^{(a,b,c,d)}(s;x), \label{dUabcdthm1}\end{equation}
and their $\g$ specializations:
\begin{equation}Q^{(\g^*)}(x;s)=Q^{(\g)}(s;x). \label{dUathm1}\end{equation}
\end{thm}
\begin{proof} For conciseness, we omit the superscripts $(a,b,c,d)$ and use the superscript $*$ to stand for $(a^*,b^*,c^*,d^*)$. Starting from the equation $\hat {\mathcal H}_1(s)\,Q(s;x)=\hat e_1(x)\, Q(s;x)$, using \eqref{pieriHamabcd}, we have
$${\mathcal D}_1^*(s) \, \Delta^*(s)Q(s;x)=\sigma t^{N-1}\, \hat e_1(s)\, \Delta^*(s)Q(s;x). $$
Interchanging the variables $x\leftrightarrow s$, we find that $\Delta^*(x)Q(x;s)$ is a solution to the
$(a^*,b^*,c^*,d^*)$ eigenvalue equation. Moreover, using the normalization of 
$Q(x;s)$ with $s$ and $x$ interchanged, we have for small
$\{x^{-\al_i}\}$: $\Delta^*(x)Q(x;s)=q^{\lambda\cdot\mu}\big(1+O(\{x^{-\al_i}\})\big)$. We conclude that 
$\Delta^*(x)Q(x;s)=P^*(x;s)$ by uniqueness of the solution. The Theorem follows from
$P^*(x;s)=\Delta^*(x)Q^*(s;x)$ by Theorem \ref{UKconj} applied to $(a^*,b^*,c^*,d^*)$.
\end{proof}

This can alternatively be rephrased as duality between Koornwinder-Macdonald eigenvalue universal solutions:
\begin{equation} \Delta^{(a^*,b^*,c^*,d^*)}(s) P^{(a,b,c,d)}(x;s)=\Delta^{(a,b,c,d)}(x)\, P^{(a^*,b^*,c^*,d^*)}(s;x), \label{dUabcdthm2}\end{equation}
and their specializations:
\begin{equation}\Delta^{(\g^*)}(s)\, P^{(\g)}(x;s)=\Delta^{(\g)}(x)\,P^{(\g^*)}(s;x). \label{dUathm2}\end{equation}

Some of the above relations appear in different guises in the literature: explicitly in the type A case \cite{ShiraishiNoumi} (see also Section \ref{universAsec}), implicitly
for the other types \cite{Cherednik95,cheredMehta} where universal functions are obtained as $x,s$-symmetric
reproducing kernels. We now detail the explicit link between the universal function duality relation \eqref{dUabcdthm2}
and the Koornwinder polynomial duality \eqref{duapolsabcd}.

\begin{thm}
The universal function duality relation \eqref{dUabcdthm2} reduces to the Koornwinder polynomial duality relation \eqref{duapolsabcd} upon specializing the variables $x=q^\mu t^{\rho^*}$ and $s=q^\lambda t^\rho$ for $\lambda,\mu$ integer partitions.
\end{thm}
\begin{proof}
Starting from the universal function $P^{(a,b,c,d)}(x;s)$, we use the specialization
$s=q^\lambda t^\rho$, leading to the 
Koornwinder polynomial
$P^{(a,b,c,d)}(x;q^\lambda t^\rho)=t^{-\rho^*\cdot\lambda} \, P^{(a,b,c,d)}_\lambda(x)$. Similarly, using the specialization
$x=q^\mu t^{\rho^*}$ on the function $P^{(a^*,b^*,c^*,d^*)}(s;x)$ leads to 
$P^{(a^*,b^*,c^*,d^*)}(s,q^\mu t^{\rho^*})=t^{-\rho\cdot\mu} \, P^{(a^*,b^*,c^*,d^*)}_\mu(s)$.
The double specialization $s=q^\lambda t^\rho,x=q^\mu t^{\rho^*}$ of \eqref{dUabcdthm2} results in
\begin{equation}\label{duaP}
t^{-\rho^*\cdot \lambda}\, \Delta^{(a^*,b^*,c^*,d^*)}(q^\lambda t^\rho) \, P^{(a,b,c,d)}_\lambda(q^\mu t^{\rho^*})=t^{-\rho\cdot \mu}\, \Delta^{(a,b,c,d)}(q^\mu t^{\rho^*})\, P_\mu^{(a^*,b^*,c^*,d^*)}(q^\lambda t^{\rho}).
\end{equation}
The relation \eqref{duapolsabcd} follows from the following identity:
\begin{equation}\label{idenDelta}
\frac{\Delta^{(a^*,b^*,c^*,d^*)}(q^\lambda t^\rho)}{\Delta^{(a,b,c,d)}(q^\mu t^{\rho^*})} =t^{\rho^*\cdot \lambda-\rho\cdot \mu}
\,\frac{P_\mu^{(a^*,b^*,c^*,d^*)}(t^{\rho})}{P^{(a,b,c,d)}_\lambda(t^{\rho^*})},
\end{equation}
itself a consequence of \eqref{normpolabcd}, and of the identity $\Delta^{(a^*,b^*,c^*,d^*)}(t^\rho)=\Delta^{(a,b,c,d)}(t^{\rho^*})$.
The Theorem follows.
\end{proof}

Similarly, under the suitable specialization, the $\g$-Macdonald duality relation \eqref{dUathm2} 
reduces to the Macdonald polynomial duality \eqref{duapols} for $x=q^\mu t^{\rho^*}$ and $s=q^\lambda t^\rho$,
where $\lambda$ is any $\g$-partition, and $\mu$ any $\g^*$-partition.

\subsubsection{Duality of universal solutions in the $q$-Whittaker limit}

Universal solutions of the $\g$-Macdonald eigenvector and $\g$-Pieri equations simplify drastically in the $t\to\infty$ limit.
They read:
\begin{equation}\label{draone} \Pi^{(\g)}(x;\lL):=\lim_{t\to\infty} t^{\rho^*\cdot\lambda} \,P^{(\g)}(x;\lL t^\rho)=x^\lambda \sum_{\beta\in Q_+(R)} c_\beta^{(\g)}(\lL) \,x^{-\beta}, 
\end{equation}
and
\begin{equation}\label{dratwo}{\bf K}^{(\g)}(\lL;x):=\lim_{t\to\infty} t^{\rho^*\cdot\lambda} \, Q^{(\g)}(\lL t^\rho;x)=
x^\lambda \sum_{\beta\in Q_+(R^*)}
{\bar c}_\beta^{(\g)}(x)\, \lL^{-\beta},
\end{equation}
with $c_0^{(\g)}(\lL)={\bar c}_0^{(\g)}(x)=1$.
Note that the limit $t\to\infty$ has broken the previous symmetry $x\leftrightarrow s$,
as the $s$ variable itself contained a $t$-dependent factor. However, Corollary \ref{UGconj} of the previous section turns into the following.

\begin{thm}\label{classmodwhit}
The universal solutions $\Pi^{(\g)}(x;\lL)$ and ${\bf K}^{(\g)}(\lL;x)$ are related via:
\begin{equation}\label{pikappa}
\Pi^{(\g)}(x;\lL)=\bar\Delta^{(\g)}(x)\, {\bf K}^{(\g)}(\lL;x),\ \
\bar\Delta^{(\g)}(x):=\lim_{t\to\infty} \Delta^{(\g)}(x) =\prod_{\alpha \in \widehat{\mathcal R}_{++}(\g)} 
(1-x^{-\alpha}).
\end{equation}
\end{thm}

The eigenvalue and Pieri equations are:
\begin{eqnarray}
D_a^{(\g)}(x)\, \Pi^{(\g)}(x;\lL)&=&\lL^{\omega_a^*}\, \Pi^{(\g)}(x;\lL), \label{eigenwhitU}\\
H_m^{(\g)}(\lL)\, {\bf K}^{(\g)}(\lL;x)&=&{\hat e}_m^{(R)}(x)\, {\bf K}^{(\g)}(\lL;x). \label{pieriwhitU}
\end{eqnarray}
Both equations for $a,m=1$ turn into simple triangular recursion relations for the 
coefficients $c_\beta^{(\g)}(\lL)$ and ${\bar c}_\beta^{(\g)}(x)$ respectively.

\subsection{Fourier transform and proof of the Macdonald- Q-system conjecture}\label{sec:53}

Like in the case of type A, we may reformulate the Fourier transform \eqref{FT} in the $q$-Whittaker limit in terms of universal solutions via:
$$f(x) \, \Pi^{(\g)}(x;\lL)={\bar f}(\lL)\, \Pi^{(\g)}(x;\lL)\quad \Leftrightarrow \quad f(x) \, {\bf K}^{(\g)}(\lL;x)={\bar f}(\lL)\, {\bf K}^{(\g)}(\lL;x).$$
The main result of Section \ref{sec:proof}  is the relation \eqref{identcrucial} which expresses the fact that $\bar D_{a;n}^{(\g)}(\lL)$ is the Fourier transform of $D_{a;n}^{(\g)}(x)$.
In terms of the universal Pieri solution, we expect:
$$D_{a;n}^{(\g)}(x)\, {\bf K}^{(\g)}(\lL;x)=\bar D_{a;n}^{(\g)}(\lL)\, {\bf K}^{(\g)}(\lL;x).$$
As a consequence any relation satisfied by the $\bar D$'s is satisfied by the $D$'s in the opposite direction, thus proving  the Macdonald Q-system conjecture.
The proof of these identities is identical to that in Section 4, and relies on the Fourier duality between $g^{(\g)}(\lL)$ and $\gamma^{(\g)}(x)$, whose adjoint action respectively generates the discrete time translation in the $\lL$ and $x$ pictures:
$$\gamma^{(\g)}(x)\, {\bf K}^{(\g)}(\lL;x)=g^{(\g)}(\lL)\, {\bf K}^{(\g)}(\lL;x).$$
The proofs rely on a uniqueness argument which can be rephrased as follows. The universal Pieri solution ${\bf K}^{(\g)}(\lL;x)$ \eqref{dratwo} is fixed by the Pieri equation \eqref{pieriwhitU} for $m=1$, up to an overall multiplicative function independent of $\lL$, and fixed by the leading term normalization ${\bar c}_0^{(\g)}(x)=1$.
Any other series solution of this Pieri equation is therefore proportional to ${\bf K}^{(\g)}(\lL;x)$, by a factor independent of $\lL$. 

As we saw in Section \ref{bnproofsec}, the case of odd times $a=N$ for $\g=B_N^{(1)}$ required the use of Rains operators, and the mapping to
a companion theory $B_N^{(1)\, '}$ with its own $q$-Whittaker polynomials and quantum Q-system. Let us rephrase the action of Rains operators at finite $t$ \eqref{rainskor} in the $B_N^{(1)}$ specialization $(a,b,c,d)=(t,-1,q^\half,-q^\half)$ of Table \ref{korspec} in terms of universal Macdonald functions. Let us denote for short $P(x;s)\equiv P^{(B_N^{(1)})}(x;s)$ and $P'(x;s)\equiv P^{(B_N^{(1)\,'})}(x;s)$ the respective universal Macdonald solutions. As both $B_N^{(1)},B_N^{(1)\, '}$ share the same $\xi_\g=\xi_{\g'}=\half$, i.e. $\sigma=\sigma'=t^\half$ and $R(\g)=R(\g')$ while $t^{\rho^*_i}=t^{N-i+1}$ and $t^{(\rho')^*_i}=q^\half t^{N-i+1}$,
we may rewrite \eqref{rainskor} as:
$$ \mathcal R_N^{(1,-1)}(x) \, P(x;s)= F(s)\, P'(x;s), \qquad  \mathcal R_N^{(t,-1)}(x)\, P'(x;s)= F'(s)\, P(x;s), $$
with 
$$F(s)=\prod_{i=1}^N (1+t^{-\half} s_i),\quad F'(s)= t^{N(N+1)/2}\prod_{i=1}^N (1+t^{-\half} s_i^{-1}), $$
where we have used the same notations as in the proof of Lemma \ref{rainstokor}, with 
$u=ab/\sigma=-t^\half$, and taken into account the product formula $q^{|\lambda |}=t^{-N^2/2}\,\prod_{i=1}^N s_i$.
Taking the limit $t\to\infty$ results in:
$$R_N^{(0)} \, \Pi(x;\lL)= (\lL_1\lL_2\cdots \lL_N)^\half \,\left(1+\frac{1}{\lL_N}\right) \, \Pi'(x;\lL),
\ \ R_N^{(1)} \, \Pi'(x;\lL) =(\lL_1\lL_2\cdots \lL_N)^\half\, \Pi(x;\lL).$$
The steps of the proof can then be repeated identically, in particular establishing that 
$\gamma(x)\,\Pi'(x;\lL)=g'(\lL)\, \Pi'(x;\lL)$, and then $D_{a;n}'\, \Pi'(x;\lL)=\bar D_{a;n}'\, \Pi'(x;\lL)$ for all $a\in [1,N]$ and $n\in \Z$,
as well as finally $D_{a;n}\, \Pi(x;\lL)=\bar D_{a;n}\, \Pi(x;\lL)$, from which the main Theorems follow.

\section{Discussion}\label{sec:conclusion}
\subsection{Companion quantum Q systems}\label{sec:threenew}

Part of the proof of our main theorem, concerning the short label in type $B_N^{(1)}$ (Section \ref{bnproofsec}), revealed that acting with the Rains operators \eqref{rainskor} and their $\tau_+$-translates on the $q$-Whittaker functions gives rise to a companion system $B_N^{(1)\,'}$ \ref{bprimeqsys} to the $B_N^{(1)}$ quantum Q-system. This new system has the same classical ($q\to 1$ limit) as the $B_N^{(1)}$ quantum Q-system but corresponds to a different specialization of the Koornwinder parameters $(a,b,c,d)$, and has different Pieri operators and time translation operator \eqref{gprime}.

One can obtain two further companion systems using the Rains operators, illustrated in Figure \ref{fig:nine}, 
from the theories at specialization parameters corresponding to
$\g=A_{2N-1}^{(2)},A_{2N}^{(2)}$. These also have the same classical limit as the respective quantum Q-systems, and may be considered as alternative quantizations. In type $A_{2N}^{(2)}$, which is not a cluster algebra mutation in the first place, the companion system is, in some sense, simpler and more natural from the quantization point of view, and appears to be the one  related to the graded tensor product character formulas of KR-modules \cite{SimonTwisted}. 

\begin{figure}\label{fig:nine}
\begin{tikzcd}
&B_{N}^{(1)}  \!\!\! \ar[d,bend left,pos=5/9,"\mathcal R^{(0)}_N"] &\!\!\!\! \!\!\!\!(t,-1,q^\half,-q^\half) & & & A_{2N}^{(2)}   \ar[d,bend left,pos=5/9,"\mathcal R^{(2)}_N"]  &\!\!\!\! \!\!\!\!\!\!\!\! \!\!\!\! \!\!\!\!(t,-1,t^\half q^\half,-t^\half q^\half)\\
&B_{N}^{(1)\, '}  \!\!\! \ar[u,bend left,pos=4/9,"\mathcal R^{(1)}_N"] &\!\!\!\! \!\!\!\!(1,-1,t q^\half,-q^\half)& & & A_{2N}^{(2)\, '} \ar[u,bend left,pos=4/9,"\mathcal R^{(1)}_N"] &\!\!\!\! \!\!\!\!\!\!\!\! \!\!\!\! \!\!\!\!(q^\half t,-q^\half,t^\half ,-t^\half )\\
&A_{2N-1}^{(2)} \ar[d,bend left,pos=5/9,"\mathcal R^{(0)}_N"] &\!\!\!\! \!\!\!\!(t^\half,-t^\half,q^\half,-q^\half) \quad & & & D_N^{(1)}\ \  \ar[loop left,"\mathcal R^{(0)}_N"] &\!\!\!\! \!\!\!\!\!\!\!\!\!\!\!\! (1,-1,q^\half,-q^\half) \\
&A_{2N-1}^{(2)\, '} \ar[u,bend left,pos=4/9,"\mathcal R^{(2)}_N"] &\!\!\!\! \!\!\!\!(1,-1,t^\half q^\half,-t^\half q^\half) \quad & & & D_{N+1}^{(2)} \,\, \ar[loop left,"\mathcal R^{(1)}_N"] &\!\!\!\! \!\!\!\!\!\!\!\!\!\!\!\! (t,-1,t q^\half,-q^\half)\\
& & & & & C_{N}^{(1)} \ \  \ar[loop left,"\mathcal R^{(2)}_N"] &\!\!\!\! \!\!\!\!\!\!\!\!\!\!\!\!(t^\half,-t^\half,t^\half q^\half,-t^\half q^\half)  
\end{tikzcd}
\caption{The nine families of Koornwinder-Macdonald operators/polynomials. We have indicated the Rains operators that intertwine the various theories, and the specializations of the Koornwinder parameters $(a,b,c,d)$.}\label{ninesystems}
\end{figure}

Using the definition of the Rains operators \eqref{rains}, together with the properties \eqref{rainskor} of Section \ref{rainsubsec}, 
from with the specializations of Table \ref{korspec}, one obtains from the $\g$-Macdonald polynomials 
$P_\lambda^{(\g)}(x)$ companion polynomials $P_\lambda^{(\g')}(x)$, using 
three different specializations of the  Rains operators, 
${\mathcal R}_N^{(0)}={\mathcal R}_N^{(1,-1)}$,  ${\mathcal R}_N^{(1)}={\mathcal R}_N^{(t,-1)}$ and ${\mathcal R}_N^{(2)}={\mathcal R}_N^{(t^\half,-t^\half)}$. 
In the cases $\g=D_N^{(1)},D_{N+1}^{(2)},C_N^{(1)}$, the Rains operators
leave the $\g$-Macdonald polynomials invariant up to some scalar, so that $\g'=\g$. In those cases, the Rains operator is itself a Koornwinder-Macdonald operator, as shown in Lemma \ref{lesautres}. 
Only the specializations $\g=B_N^{(1)},A_{2N-1}^{(2)}$ and $A_{2N}^{(2)}$ are mapped to new companion theories $\g'\neq \g$.

In the $q$-Whittaker limit, the Rains operators tend to $R_N^{(0)},R_N^{(1)}$
of  (\ref{rzero}-\ref{rone}) and to
$$R^{(2)}_N= \lim_{t\to\infty} t^{-N(N+1)/2} {\mathcal R}_N^{(t^\half,-t^\half)}=
\sum_{\epsilon_1,\ldots,\epsilon_N=\pm 1} \prod_{i=1}^N \frac{x_i^{2\epsilon_i}}{x_i^{2\epsilon_i}-1} \prod_{1\leq i<j\leq N}
\frac{x_i^{\epsilon_i}x_j^{\epsilon_j}}{x_i^{\epsilon_i}x_j^{\epsilon_j}-1} \prod_{i=1}^N \Gamma_i^{{\epsilon_i}/2}.$$
These map the $q$-Whittaker polynomials $\Pi_\lambda=\Pi_\lambda^{(\g)}$ to a multiple of the
monic companions $\Pi_\lambda'=\Pi_\lambda^{(\g')}=\lim_{t\to\infty} P_\lambda^{(\g')}$.
More precisely (c.f. \eqref{backnforthB} for $B_N^{(1)}$), \eqref{rainskor} implies the following mapping of polynomials:
$$
\begin{array}{rll}
A_{2N-1}^{(2)}:&  R^{(0)}_N\, \Pi_\lambda=(\lL_1\lL_2\cdots \lL_N)^\half\left(1+\frac{1}{\lL_N}\right) \Pi_\lambda',&
R^{(2)}_N\, \Pi_\lambda'= (\lL_1\lL_2\cdots \lL_N)^\half \Pi_\lambda,\\
A_{2N}^{(2)}:&  R^{(2)}_N\, \Pi_\lambda
= (\lL_1\lL_2\cdots \lL_N)^\half \Pi_\lambda',\quad &
R^{(1)}_N\, \Pi_\lambda'= (\lL_1\lL_2\cdots \lL_N)^\half \Pi_\lambda.
\end{array}
$$

For the new companion cases $A_{2N-1}^{(2)\,'},A_{2N}^{(2)\,'}$, we define
companion Koornwinder-Macdonald operators by specializing Definition \ref{malphadef} for $m=1,2,...,N$ at the appropriate values of the parameters as in Figure \ref{fig:nine}. Their  $q$-Whittaker limits are denoted by $D_{m}'=D_{m}^{(\g')}$, and their $\tau_+$-translates are defined for all $a\in[1,N]$
and $n\in \Z$ as follows:
\begin{eqnarray*}
&&A_{2N-1}^{(2)\,'}:\ D'_{a;2n+i}=q^{-na}\,\gamma^{-n}\, D'_{a;i} \, \gamma^n, \quad (i=-1,0)\\
&&A_{2N}^{(2)\,'}:\ \ \ \ \ \ D'_{a;n}=q^{-na/2}\,\gamma^{-n}\, D'_{a} \, \gamma^n,
\end{eqnarray*}
where $D'_{a;-1}$ is defined as in \eqref{minusone} with the suitable function $\phi^{(\g')}_{i,\epsilon}(x)$. Note that all labels of $A_{2N-1}^{(2)\,'}$
are short and all labels of $A_{2N}^{(2)\,'}$ are long, which is the reverse of the case for the un-primed theories.
Using the methods of this paper, it can be shown that these new operators obey new quantum Q-systems, with the same $q$-commutation relations and recursion relations as their non-primed companions, except for the label $N$, where
$$
\renewcommand{\arraystretch}{1.5}
\begin{array}{rl}
A_{2N-1}^{(2)\,'}:&  q^N\,  Q_{N,2n+2}\,Q_{N,2n}=Q_{N,2n+1}^2-q^{-2n} \,Q_{N-1,2n+1}^2,\\
& q^N\,  Q_{N,2n+1}\,Q_{N,2n-1}=(Q_{N,2n}-q^{-n} Q_{N-1,2n})(Q_{N,2n}+q^{1-n} Q_{N-1,2n}),\\
A_{2N}^{(2)\,'}:& q^N\,  Q_{N,n+1}\,Q_{N,n-1}=Q_{N,n}^2-q^{-\frac{n}{2}} Q_{N,n}Q_{N-1,n} .
\end{array}
$$
for all $n\in \Z$. To derive these results, we first define the
candidate Fourier transforms $\bar D'_{a;n}$ via the 
opposite quantum Q-systems and same initial 
data as the non-primed companions (that is, the same values of $\omega_a,\omega_a^*$), and compute the time translation operators:
\begin{eqnarray*}
&&A_{2N-1}^{(2)\,'}:\ \ \ g'= g_T g_\lL \frac{1}{(\lL_N^{-2};q^2)_\infty}
g_T g_\lL \frac{1}{(\lL_N^{-1};q)_\infty(-q\lL_N^{-1};q)_\infty},\\
&&A_{2N}^{(2)\,'}:\ \ \ \ \ g'= g_T g_\lL \, \frac{1}{(\lL_N^{-1};q)_\infty},
\end{eqnarray*}
which commute, respectively, with the two Pieri operators or Toda Hamiltonians
\begin{eqnarray*}
&&A_{2N-1}^{(2)\,'}:\ \ \ H_1'(\lL)=T_1+\sum_{i=2}^N {\scriptstyle \left(1-\frac{\lL_i}{\lL_{i-1}}\right)}T_i+\sum_{i=1}^{N-1}{\scriptstyle \left(1-\frac{\lL_{i+1}}{\lL_{i}}\right)} T_i^{-1} 
+{\scriptstyle \left(1-\frac{1}{\lL_{N}}\right)\left(1+\frac{q}{\lL_{N}}\right)} T_N^{-1}, \\
&&A_{2N}^{(2)\,'}:\ \ \ \ \ H_1'(\lL)=T_1+\sum_{i=2}^N {\scriptstyle \left(1-\frac{\lL_i}{\lL_{i-1}}\right)}T_i+\sum_{i=1}^{N-1}
{\scriptstyle \left(1-\frac{\lL_{i+1}}{\lL_{i}}\right)}T_i^{-1}+{\scriptstyle \left(1-\frac{1}{\lL_{N}}\right)}T_N^{-1} -{\scriptstyle \frac{q^{-\half}}{\lL_N}}
\end{eqnarray*}
obtained from the $q$-Whittaker limit of the specialized operators of Theorem \eqref{korpier}.  After proving that
$g'\Pi_\lambda=\gamma^{t_1}\Pi_\lambda'$ ($t_1=2$ for $A_{2N-1}^{(2)\,'}$, $t_1=1$ for $A_{2N}^{(2)\,'}$), 
one concludes that $D_{a;n}\Pi_\lambda'=\bar D_{a;n}'\Pi_\lambda'$ and the quantum Q-system relations for $D_{a;n}$ follow.

Finally, matching the specialization of the formula \eqref{deltabcd} for $\Delta^{(\g')}$ with the form  of the product over affine roots \eqref{deltaG}, we make the following identification of affine roots corresponding to $\g'$:
\begin{eqnarray*}
&&B_N^{(1)\,'}:\ \ \ \widehat{R}_{++} = \scriptstyle{\left\{
(n-\half)\delta+e_i:  \ 1\leq i\leq N, \ n\geq 1\right\} \sqcup \left\{ n\delta+(e_i\pm e_j):\  1\leq i<j\leq N,\  n\geq 1 \right\}}\\
&&A_{2N-1}^{(2)\,'}:\ \widehat{R}_{++} = \scriptstyle{\left\{ (2n-1)\delta+2 e_i: \ 1\leq i\leq N,\  n\geq 1\right\} \sqcup\left\{ n\delta+(e_i\pm e_j): \ 1\leq i<j\leq N,\ n\geq 1\right\}}\\
&&A_{2N}^{(2)\,'}:\ \ \ \widehat{R}_{++} = \scriptstyle{\left\{(n-\half)\delta+e_i,\, 2n\delta+2 e_i,: \ 1\leq i\leq N,\  n\geq 1\right\}\sqcup\left\{ n\delta+(e_i\pm e_j): \ 1\leq i<j\leq N,\ n\geq 1\right\}}.
\end{eqnarray*}

\subsection{Universal solutions and Path models}\label{sec:paths}

In Section \ref{sec:universal}, we  introduced universal solutions $P(x;s)$ and $Q(s;x)$ for the
various eigenvalue equations and Pieri rules, by solving an infinite triangular linear system of equations for the expansion coefficients in each case. Remarkably,
the solutions can all be put in the form of path models, 
such as that for the $q$-Whittaker vectors in the $A_{N-1}^{(1)}$ case 
\cite{DKT}, see also \cite{gotsy}. As shown in Example \ref{pathAex}, this results in an interpretation of the solution of the triangular system as a partition function for weighted paths on the relevant positive root cone $Q_+$ or $Q_+^*$.

Let us start with the case of $Q(s;x)=q^{\lambda\cdot\mu}\sum_{\beta\in Q_+^*} {\bar c}_\beta(x) \, s^{-\beta}$, $\bar c_0(x)=1$, the series solution of 
the Pieri equation $\cH_1(s) \, Q(s;x)=\hat e_1(x)\, Q(s;x)$ (either the Koornwinder or $\g$-Macdonald Pieri rule). 
The important fact is that $\cH_1(s)$ is a $q$-difference operator in the variable $s$
of the form $\cH_1(s)=\sum_{i=1}^N\sum_{\epsilon=0,\pm 1} h_{i,\epsilon}(s) T_i^\epsilon$, where the coefficients $h_{i,\epsilon}$ are  rational functions
of the variables $\{s^{-\al_i^*}\}$, with $\al_i^*$ the simple roots of $R^*$.
Factoring out the common denominator in the coefficients, $h_{i,\epsilon}(s)= p_{i,\epsilon}(s)/q(s)$ with $p_{i,\epsilon}(s)$ and $q(s)$ 
some {\it polynomials} in $\{s^{-\al_i^*}\}$, the Pieri equation is equivalent to
$\left(\hat e_1(x)\, q(s)-\sum_{i=1}^N\sum_{\epsilon=0,\pm 1} p_{i,\epsilon}(s) x_i^\epsilon T_i^\epsilon \right){q^{-\lambda\cdot\mu}Q(s;x)}=0.$
The difference operator can be written as
$\sum_{i=1}^N\sum_{\epsilon=0,\pm 1} \sum_{\al\in U^*} h_{i,\epsilon,\al}(x) s^{-\al} T_i^\epsilon$ 
for some $x$-dependent coefficients $h_{i,\epsilon,\al}(x)$, where $\al$ is summed over a {\it finite} subset $U^*$ of $Q_+^*$.
Noting that $T_i^\epsilon\, s^{-\beta}=q^{-\epsilon e_i\cdot\beta}s^{-\beta}$ and that $s^{-\al}\, s^{-\beta}=s^{-(\al+\beta)}$,
and collecting the coefficient of $s^{-\beta}$ in the Pieri equation, we obtain
$\sum_{i=1}^N\sum_{\epsilon=0,\pm 1} \sum_{\al\in U^*} h_{i,\epsilon,\al}(x)\, q^{-\epsilon e_i\cdot(\beta-\al)}\,{\bar c}_{\beta-\al}(x)=0,$
or equivalently
$$\left(\sum_{i,\epsilon} h_{i,\epsilon,0}(x) q^{-\epsilon e_i\cdot\beta} \right){\bar c}_{\beta}(x)=-
\sum_{i,\epsilon}\sum_{\al\in U^*\setminus \{0\}} h_{i,\epsilon,\al}(x) q^{-\epsilon e_i\cdot(\beta-\al)} {\bar c}_{\beta-\al}(x).$$
The factor on the left hand side is non-vanishing for generic $x$ and is therefore invertible.
The path model interpretation goes as follows. Given the initial data $\bar c_0(x)=1$, the coefficient ${\bar c}_{\beta}(x)$ 
is the partition function of paths from $0$ to $\beta$ in $Q^*_+$, consisting of steps
in the finite subset $U^*\setminus \{0\}$. Each path has a weight equal to the product of vertex and edge weights along the path, 
defined respectively as
$$ w_\varphi= \left(\sum_{i,\epsilon} h_{i,\epsilon,0}(x) q^{-\epsilon e_i\cdot\varphi} \right)^{-1},\quad  w_0=1,\quad
w_{\varphi-\al,\varphi}= -\sum_{i,\epsilon}h_{i,\epsilon,\al}(x) q^{-\epsilon e_i\cdot(\varphi-\al)}$$
for any $\varphi\in Q_+^*$, $\al \in U^*\setminus \{0\}$. The weight of a path $p$ is a product over its vertex and edge sets $v(p),e(p)$:
$$w(p)= \prod_{ \varphi\in  v(p)} w_\varphi  \times \prod_{(\varphi-\al,\varphi)\in e(p)} w_{\varphi-\al,\varphi} ,
$$
and the coefficient $\bar c_\beta(x)$ is the partition function of all such paths
$$ {\bar c}_{\beta}(x)=\sum_{{\rm paths}\, p\ {\rm in}\ Q_+^*\atop {\rm from}\ 0\to\beta} w(p) .$$
A similar argument leads to a formulation of the Macdonald eigenvalue equation universal solution $P(x;s)$ in terms of a path model on $Q_+$.

The path model for the Pieri solution simplifies drastically in the $q$-Whittaker limit
as the Hamiltonians are directly polynomials of the $\lL^{-\al_i^*}$, giving rise to a small set of steps $U^*\setminus\{0\}$. For instance, by inspection of (\ref{Hamil}-\ref{Hamil2}), we find that for all $\g$, $w_\varphi^{-1}=\sum_{i=1}^N\sum_{\epsilon=\pm1} x_i^{\epsilon}(1-q^{-\epsilon\beta\cdot e_i})$.

\begin{example}
In the case $\g=A_4^{(2)}$, the $q$-Whittaker limit of the Pieri equation is
\begin{eqnarray*}
0&=&\left(T_1+\left(1-\frac{\lL_2}{\lL_1}\right) (T_2+T_1^{-1})+\left(1-\frac{1}{\lL_2}\right) T_2^{-1} -\frac{1}{\lL_2}-(x_1+x_1^{-1}+x_2+x_2^{-1})\right)\\
& \times& x_1^{\lambda_1}x_2^{\lambda_2} \sum_{n,m\geq 0}
{\tilde c}_{n,m}(x) \left(\frac{\lL_2}{\lL_1}\right)^n\left(\frac{1}{\lL_2}\right)^m
\end{eqnarray*}
expressed in terms of the variables $\lL^{-\al_1}=\frac{\lL_2}{\lL_1},\lL^{-\al_2}=\frac{1}{\lL_2}$,
$\al_i$ the simple roots of $B_2$, and  with the normalization ${\tilde c}_{0,0}(x) =1$.
The is equivalent to the recursion relation
$${\tilde c}_{n,m}(x)=\frac{q^{n-1}(q^{-m}x_2+x_1^{-1}){\tilde c}_{n-1,m}(x)+(q^{m-n-1}x_2^{-1}+1){\tilde c}_{n,m-1}(x)}{x_1(q^{-n}-1)+x_1^{-1}(q^n-1)+x_2(q^{n-m}-1)+x_2^{-1}(q^{m-n}-1)},$$
which has the form ${\tilde c}_{n,m}(x)=a_{n,m}{\tilde c}_{n-1,m}(x)+b_{n,m}{\tilde c}_{n,m-1}(x)$, the solution is
$${\tilde c}_{n,m}(x) =\sum_{{\rm paths}\, p: (0,0)\to (n,m)} \prod_{{\rm steps}\,   s\in p} w_s \ ,$$
where the is sum over all the paths on $\Z_+^2$ with steps $(1,0)$ and $(0,1)$, from the origin to $(n,m)$, of the product of step weights $w_s$ where 
$w_s=a_{i,j}$ for a horizontal step $(i-1,j)\to (i,j)$ and $w_s=b_{i,j}$ for a vertical step $(i,j-1)\to (i,j)$.
\end{example}

\subsection{Universal solutions and $q$-Whittaker functions}\label{sec:whitak}
We have used the terminology $q$-Whittaker functions for the $t\to\infty$ limits of Koornwinder-Macdonald polynomials. Strictly speaking,
$\Pi_\lambda^{(\g)}(x)$ is a class-$1$ $q$-Whittaker function for the quantum universal enveloping algebra of a simple Lie algebra only in the cases where the Pieri operators are known to be $q$-deformed quantum Toda Hamiltonians.
These are related to the quantum Q-system of types
$\g=A_{N-1}^{(1)},D_N^{(1)},D_{N+1}^{(2)},A_{2N-1}^{(2)}$, whose conserved quantities are the relativistic Toda Hamiltonians,
of types $A_{N-1},D_N,B_N,C_N$, respectively (see Remark \ref{todarem}). 
By analogy, we call all the eigenfunctions $\Pi_\lambda^{(\g)}(x)$ $q$-Whittaker functions for all $\g$ in this paper.
These share a number of properties.

In the $q$-Whittaker limit, the universal series solutions of the $\g$-Macdonald eigenvalue equations and the Pieri equations (\ref{draone}-\ref{dratwo}) of Section \ref{sec:universal},
when specialized to $\lambda$ a $\g$-partition, reduce respectively to the analogues of 
class-1 $q$-Whittaker functions $\Pi_\lambda^{(\g)}(x)=\Pi^{(\g)}(x;q^\lambda)$, and the analogues of fundamental
$q$-Whittaker functions
$\Psi_\lambda^{(\g)}(x)=\mathrm K^{(\g)}(x;q^\lambda)$. The former is a Weyl-symmetric polynomial in $x$, whereas the latter is a non-symmetric series solution of the relativistic Toda equation with prescribed leading term. It is associated with the highest weight Verma module $V_\mu$, where $\mu$ is obtained from the variable $x=q^{\rho+\mu}$.
In the theory of Whittaker functions, the fundamental Whittaker function is convergent only in a particular Weyl chamber. The class-1 function is a linear combination of fundamental Whittaker functions, regular in all the Weyl chambers.
It is  obtained via a symmetrization over the Weyl group action on the fundamental solutions with suitable coefficients.

A subtlety arises in the study of convergence of the above series. In the definition of $\bar\Delta^{(\g)}(x)$ as an infinite product, it is assumed that $|q|<1$ for convergence.
On the other hand, the series $\mathrm K^{(\g)}(q^\lambda;x)$ or $\Psi_\lambda^{(\g)}(x)$ is well-behaved for $|q|>1$.
Obviously $\Pi_\lambda^{(\g)}(x)$, being a polynomial, makes sense for both $|q|>1$ and $|q|<1$. 
Note that the infinite product
$$\tilde \Delta^{(\g)}=\frac{1}{\prod_{n=0}^\infty \prod_{\al\in R_+} (1-q^{-n} x^{-\al})} $$
is another solution of \eqref{eyh}, convergent for $|q|>1$ instead of $|q|<1$. (To see this, note that conjugating 
$\Gamma_i^{\pm 1}$ with $\bar \Delta^{(\g)}(x)$ or $\tilde \Delta^{(\g)}(x)$ yields the same result.). The remarkable fact is that as a convergent series for $|q|>1$, 
$\tilde \Delta^{(\g)}(x)\, \Psi_\lambda^{(\g)}(x)$ is not equal to the polynomial class-$1$ $q$-Whittaker function $\Pi_\lambda^{(\g)}(x)$. We conjecture\footnote{Such a symmetrization formula exists relating classical fundamental and class-$1$ Whittaker functions.} that for general $\g$ 
it requires a symmetrization over the Weyl group, valid for $|q|>1$:
\begin{equation} \label{conjsym}
\Pi_\lambda^{(\g)}(x) =\sum_{w\in W} \tilde \Delta^{(\g)}(wx)\, \Psi_\lambda^{(\g)}(wx).
\end{equation}

We end up with two characterizations of $\mathrm K(\lL;x)$: (1) as a series with $|q|<1$ in the variables $\lL^{-\al_i}$, equal to $\Pi^{(\g)}(x;\lL)/\bar \Delta^{(\g)}(x)$
(by Theorem \ref{classmodwhit}) and (2) as a series with $|q|>1$ equal to $ \Psi_\lambda^{(\g)}(x)$ when $\lambda$ is a $\g$-partition. The polynomials $\Pi_\lambda^{(\g)}(x)$ are expressed in terms of both, but with very different formulas for $|q|<1$ and $|q|>1$.

\begin{example}
Let us illustrate the above in the simplest case of $A_1^{(1)}$. Denoting by $u=x_2/x_1$,
$v=\lL_2/\lL_1$,
\begin{eqnarray*}p(u;v)&:=&x_1^{-\lambda_1}x_2^{-\lambda_2}\,\Pi(x;\lL)=\sum_{n=0}^\infty  u^n \,\prod_{i=0}^{n-1} 
\frac{1-v q^i}{1-q^{-i-1}}, \\
k(v;u)&:=&x_1^{-\lambda_1}x_2^{-\lambda_2}\,\mathrm K(\lL;x)=\sum_{n=0}^\infty v^n\,
\frac{q^{n^2} u^n}{\prod_{i=1}^{n}(1-q^i)(1- u q^i)} .\end{eqnarray*}
With $\bar \Delta(u)=(q u;q)_{\infty}$ and $\tilde \Delta(u)=1/(u;q^{-1})_{\infty}$, 
\begin{eqnarray*}
p(u;v)&=& \bar \Delta(u)\, k(v;u) \qquad (|q|<1),\\
\Pi_{\lambda_1,\lambda_2}(x)&=& x_1^{\lambda_1}x_2^{\lambda_2}\, \tilde \Delta(u)\,  k(q^{\lambda_2-\lambda_1};u)+ 
x_2^{\lambda_1}x_1^{\lambda_2}\,\tilde \Delta(u^{-1})\,  k(q^{\lambda_2-\lambda_1};u^{-1})\quad (|q|>1),\\
&=& x_1^{\lambda_1}x_2^{\lambda_2}\,p(u;q^{\lambda_2-\lambda_1})\quad (|q|<1),
\end{eqnarray*}
where the second and third lines hold for any integers $\lambda_1\geq \lambda_2\geq 0$.
\end{example}

\subsection{Summary/Perspectives}\label{sec:summary}

We have proved the Macdonald operator-quantum Q-sys\-tem conjectures, which state that suitably defined,
$\tau_+$-translated $\g$-Macdonald difference operators, in the $q$-Whittaker limit, obey the $\g$-type quantum Q-system 
relations, and may as such be considered as cluster variables in a suitable quantum cluster algebra in all cases except $A_{2N}^{(2)}$.
The proofs cover the cases $\g=A_{N-1}^{(1)},D_N^{(1)},B_N^{(1)}, C_N^{(1)}, A_{2N-1}^{(2)},A_{2N}^{(2)},D_{N+1}^{(2)}$, and rely strongly on the duality between the Koornwinder/Macdonald eigenvalue equations and the associated Pieri rules. 

We have proved that the conserved quantities of the $\g$-quantum Q-systems are the $\g$-Pieri operators in the $q$-Whittaker limit,
which in a number of cases can themselves be identified with known $q$-difference Toda Hamiltonians. Our construction provides explicit expressions for all (higher) Hamiltonians as well. It would be interesting to explore the combinatorial content of these, possibly in the language of cluster integrable systems
\cite{GoncharovKenyon}. 

We have constructed time-translation operators $g(\lL)$ for all $\g$ that commute with the corresponding $q$-difference Toda Hamiltonians. More generally, in the spirit of \cite{SS}, it would be desirable to construct commuting Baxter $Q$-operators $Q(u;\lL)$ 
that coincide with $g(\lL)^{-1}$ at $u=1$, but are in general quantum-dilogarithmic generating functions for the Hamiltonians. In all cases but 
$A_{2N}^{(2)}$, such a construction should exist in terms of mutations of the corresponding quantum cluster algebra.

As another by-product of our proofs, we have unearthed a remarkable structure involving three more
companion theories $B_N^{(1)\,'}$, $A_{2N-1}^{(2)\,'}$ and $A_{2N}^{(2)\,'}$ and their associated quantum
Q-systems. It would be interesting to further investigate their quantum Laurent property, as well as their combinatorial content, in particular the meaning of the associated fermionic sums which in the known cases provide q-multiplicities of decompositions of KR modules onto irreducibles (see \cite{krKR,DFK15}).

The quantum Q-system relates to the $q$-Whittaker limit of Macdonald-Koornwinder theory, but it is natural to consider the extension of some of our results to finite $t$. The time-shifted Macdonald operators $\mathcal D^{(\g)}_{a;n}(x;q,t)$ by iterated conjugation by $\gamma(x) $ were considered in \cite{DFKqt} in type $A_{N-1}^{(1)}$, and are generators of the spherical DAHA or, in the limit $N\to\infty$, the quantum toroidal algebra of $\mathfrak{gl}_1$ or the elliptic Hall algebra.
In the Koornwinder we again have elements of the corresponding spherical DAHA but the limit $N\to\infty$ is still to be understood. 

Moreover, in type $A_{N-1}^{(1)}$ a $t$-deformed analogue of the time translation operator $g(\lL)$ was defined in \cite{Langmann} (see Proposition 4.2),
and conjectured to act diagonally on the so-called Non-stationary Ruijsenaars function, introduced in \cite{ShiraishiLaumon} as the (conjectural) universal series solution to the elliptic Ruijsenaars operator eigenvalue equation, also related to the geometry of affine Laumon spaces. We expect this operator to play the same role for the finite $t$ macdonald case as $g(\Lambda)$ for the q-Whittaker limit $t\to\infty$, i.e. to be the (Macdonald) Fourier transform of the Gaussian for finite $t$. 
We believe there should exist a cluster algebra formulation of this operator.
It would also be extremely interesting to investigate this operator for general Koornwinder theory.

\appendix

\section{Derivation of the $\g$-Macdonald operators}\label{appendixA}
\label{appA}

We combine several constructions \cite{macdoroot,vandiej,Rains} of commuting difference operators corresponding to the affine algebras in Table \ref{tableone}. The goal is to construct an appropriate set of $N$ commuting operators for each $\g$, with  eigenvalues proportional to the symmetric functions in Table \ref{chartable}, which form a basis for the space spanned by the irreducible fundamental characters of the Lie algebras $R$.
This choice of $\g$-Macdonald difference operators is designed such that their $q$-Whittaker limits satisfy the type $\g$ quantum Q-systems.

\subsection{Macdonald's operators}
\label{appmacsec}

For each affine algebra $\g$ in Table \ref{tableone}, except for the case of $A_{2n}^{(2)}$, and for each minuscule co-weight of $S$, Macdonald defines a  difference operator with eigenvalue which is a fundamental character of $R^*$ \cite{macdoroot}.

Let $\{e_i\}_{1\leq i\leq N}$ be the standard basis of $\R^N$ with the standard inner product $(\cdot,\cdot)$.  For the set of variables $x=(x_1,...,x_N)$, we denote $x^v=x_1^{v_1}\cdots x_N^{v_N}$ for any $v=\sum_i v_i e_i$.
\begin{table}

\renewcommand{\arraystretch}{1}
\begin{equation*}
\begin{array}{|c|l|l|}
\hline
\hbox{Algebra} & \hbox{Positive roots $R_+$} & \hbox{fundamental weights $\omega_a$}\\ \hline
B_N & e_i,\ i\in[1,N]; & 
\displaystyle{\omega_i=\sum_{k=1}^i e_k,\ i\in[1,N-1]};\\
& e_i\pm e_j,\ 1\leq i< j\leq N&\displaystyle{  \omega_N=\frac{1}{2}\sum_{k=1}^N e_k}\\ \hline
C_N &e_i\pm e_j,\ 1\leq i< j\leq N; \ 2e_i, \ i\in[1,N] &\displaystyle{\omega_i=\sum_{k=1}^i e_k,\ i\in[1,N] }\\ \hline
D_N & e_i\pm e_j, \ 1\leq i<j\leq N& \displaystyle{
\omega_i=\sum_{k=1}^i e_k,\  i\in[1,N-2]; }\\
&& \displaystyle{\omega_N=\frac{1}{2}\sum_{k=1}^N e_k,\ \omega_{N-1}=\omega_N-e_N}\\ \hline
\end{array}
\end{equation*}
\caption{Positive roots and fundamental weights of the finite dimensional algebras of types BCD.}\label{rootsandweights}
\end{table}

There is a surjective map ${}_*: { R}\to { S}$, $\al\mapsto \al_*=\al/u_\al\in S$, for some real $u_\al$. 
In the case where $R\neq S$, $u_\al=\frac{(\al,\al)}{2}$, so that $u_\al=1$ in all cases but for the short roots of type $B_N$ or the long roots of type $C_N$, in which case it is equal to $\frac12$ or 2, respectively.

Let $\pi=\sum \pi_i e_i$, and define\footnote{We choose Macdonald's parameters $t_\al=t$ for all $\alpha$,
independently of the length $\al$. This allows us to obtain the dual $q$-Whittaker limit by simply taking $t\to\infty$.}
\begin{equation}\label{macphi}
\Phi_\pi:=\prod_{\al\in { R}_+\atop (\pi,\al_*)=1}\frac{1-t x^\al}{1-x^\al} , \qquad \Gamma_\pi=\prod_i \Gamma_i^{\pi_i}.
\end{equation}
For each minuscule weight $\pi$ of ${S}^\vee=\{ 2\frac{\al}{(\al,\al)}, \, \al\in S\}$ , i.e. a weight such that $(\pi,\al_*)\in \{0, \pm1\}$
for all $\al\in { R}$, there is a Macdonald difference operator ${\mathcal E}_\pi$ which acts on functions $f(x)$ by the symmetrization
\begin{equation}\label{macmac}{\mathcal E}_\pi f=\sum_{w\in W}w\left(\Phi_\pi \Gamma_\pi f \right).
\end{equation}
We list below the explicit formulas for each case treated in \cite{macdoroot}.
The construction refers to the positive roots and fundamental weights for the simple Lie groups of types $BCD$ in Table \ref{rootsandweights}.

\subsubsection{Macdonald operators for $D_N^{(1)}$}
\label{DNopsec}
Here, $R=S$ is the root system of type $D_N$. 
There are three minuscule weights: $\omega_1=e_1$, $\omega_{N-1}$
and $\omega_N$. 
Equation \eqref{macphi} becomes
\begin{eqnarray*}
\Phi_{\omega_1}&=& \prod_{j=2}^N \frac{1-t x_1 x_j}{1-x_1x_j}\frac{tx_1-x_j}{x_1-x_j} ,\ \
\Phi_{\omega_{N-1}}=  \prod_{1\leq i<j\leq N-1} \frac{1-t x_i x_j}{1-x_ix_j} \prod_{i=1}^{N-1}\frac{t x_i-x_N}{x_i-x_N},\\
\Phi_{\omega_N}&=&    \prod_{1\leq i<j\leq N} \frac{1-t x_i x_j}{1-x_ix_j}.
\end{eqnarray*}
The Weyl group of $D_N$ acts on the set $(x_1,x_2,...,x_N)$ by permutations of the indices and 
inversions of an even number of variables. The three Macdonald operators are
\begin{eqnarray}
{\mathcal E}_{\omega_1}^{(D_N^{(1)})}&=&\sum_{\epsilon=\pm 1}\sum_{i=1}^N \prod_{j\neq i} \frac{1-t x_i^\epsilon x_j}{1-x_i^\epsilon x_j}
\frac{t x_i^\epsilon-x_j}{x_i^\epsilon-x_j} \Gamma_i^{\epsilon} ,\label{MD1}\\
{\mathcal E}_{\omega_{N-1}}^{(D_N^{(1)})}&=&  \sum_{\epsilon_1,...,\epsilon_N=\pm 1\atop \epsilon_1\epsilon_2\cdots \epsilon_N=-1}  
\prod_{1\leq i<j\leq N} \frac{1-t x_i^{\epsilon_i} x_j^{\epsilon_j}}{1-x_i^{\epsilon_i}x_j^{\epsilon_j}} 
\prod_{i=1}^N \Gamma_i^{\frac{\epsilon_i}{2}}, \label{MD2}\\
{\mathcal E}_{\omega_N}^{(D_N^{(1)})}&=&  
\sum_{\epsilon_1,...,\epsilon_N=\pm 1\atop \epsilon_1\epsilon_2\cdots \epsilon_N=1}  \prod_{1\leq i<j\leq N} \frac{1-t x_i^{\epsilon_i} x_j^{\epsilon_j}}{1-x_i^{\epsilon_i}x_j^{\epsilon_j}} \prod_{i=1}^N \Gamma_i^{\frac{\epsilon_i}{2}}.\label{MD3}
\end{eqnarray}
These will be identified below as  $\D_a^{(D_N^{(1)})}(x;q,t)={\mathcal E}_{\omega_{a}}^{(D_N^{(1)})}$ with $a=1,N-1,N$.

\subsubsection{Macdonald operator for $B_N^{(1)}$ }
\label{BNopsec}
Here, $R=S$ is the root system of $B_N$
There is a unique minuscule weight of $S^\vee=C_N$,
$\pi=e_1=\omega_1$, with 
$$\Phi_{\omega_1}=\frac{1-t x_1}{1-x_1}\prod_{j=2}^N \frac{1-t x_1 x_j}{1-x_1x_j}\frac{t x_1-x_j}{x_1-x_j}.
$$
The Weyl group $W\simeq S_N \ltimes \Z_2$ is generated by all permutations and inversions of the variables in the set $x=(x_1,x_2,...,x_N)$, so the corresponding difference operator is
$$
{\mathcal E}_{\omega_1}^{(B_N^{(1)})}=\sum_{\epsilon=\pm 1}
\sum_{i=1}^N \frac{1-t x_i^{\epsilon}}{1-x_i^{\epsilon}}
\prod_{j\neq i} \frac{1-t x_i^{\epsilon} x_j}{1-x_i^{\epsilon}x_j}\frac{t x_i^{\epsilon} -x_j}{x_i^{\epsilon}-x_j}
\Gamma_i^{\epsilon}.
$$
This will be identified as  $\D_1^{(B_N^{(1)})}(x;q,t)={\mathcal E}_{\omega_1}^{(B_N^{(1)})}$.

\subsubsection{Macdonald operator for $\g=C_N^{(1)}$ }
\label{CNopsec}
Here, $R=S$ is the root system of type $C_N$. 
There is a unique minuscule weight of $S^\vee=B_N$, $\pi=\frac{1}{2}\sum_{i=1}^N e_i=\omega_N$,
with
$$
\Phi_{\omega_N}=\prod_{i=1}^N \frac{1-t x_i^2}{1-x_i^2} \prod_{1\leq i<j\leq N}\frac{1-t x_i x_j}{1-x_ix_j}.
$$
The Weyl group is is the same as for type $B_N$, resulting in the difference operator
\begin{equation}\label{CNN}
{\mathcal E}_{\omega_N}^{(C_N^{(1)})}=\sum_{\epsilon_1,...,\epsilon_N=\pm 1}\prod_{i=1}^N
\frac{1-t x_i^{2\epsilon_i}}{1-x_i^{2\epsilon_i}} 
 \prod_{1\leq i<j\leq N}  \frac{1-t x_i^{\epsilon_i}x_j^{\epsilon_j}}{1-x_i^{\epsilon_i}x_j^{\epsilon_j}}
\prod_{i=1}^N \Gamma_i^{\frac{\epsilon_i}{2}}.
\end{equation}
This will be identified as $\D_N^{(C_N^{(1)})}(x;q,t)={\mathcal E}_{\omega_N}^{(C_N^{(1)})}$.

\subsubsection{Macdonald operator for $\g=A_{2N-1}^{(2)}$ }
\label{A2Nmopsec} Here, $(R,S)=(C_N,B_N)$.
The map $*: R\to S$ is given by $(e_i\pm e_j)_*=e_i\pm e_j$, and $(2e_i)_*=e_i$.
There is a unique minuscule weight of $S^\vee=C_N$, $\pi=e_1=\omega_1$, and
$$
\Phi_{\omega_1}=\frac{1-t x_1^2}{1-x_1^2} \prod_{j=2}^{N}\frac{1-t x_1 x_j}{1-x_1x_j}\frac{t x_1-x_j}{x_1-x_j}.
$$
Summing over the Weyl group of type $C_N$,
$$
{\mathcal E}_{\omega_1}^{(A_{2N-1}^{(2)})}=\sum_{\epsilon=\pm 1}
\sum_{i=1}^N \frac{1-t x_i^{2\epsilon}}{1-x_i^{2\epsilon}}
\prod_{j\neq i} \frac{1-t x_i^{\epsilon} x_j}{1-x_i^{\epsilon}x_j}\frac{t x_i^{\epsilon} -x_j}{x_i^{\epsilon}-x_j}
\Gamma_i^{\epsilon}.
$$
This will be identified as $\D_1^{(A_{2N-1}^{(2)})}(x;q,t)={\mathcal E}_{\omega_1}^{(A_{2N-1}^{(2)})}$.

\begin{remark}\label{atworem}
The algebra $A_{2N-1}^{(2)}$ is obtained from $A_{2N-1}^{(1)}$ by a folding procedure using the natural $\Z_2$ automorphism. Remarkably,
this extends to the difference operators as follows. Consider the specialization $\tau$ of $x=(x_1,x_2,...,x_{2N})$ obtained by setting $x_{2N+1-i}=x_i^{-1}$, $i=1,2,...,N$,
and accordingly $\Gamma_{2N+1-i}=\Gamma_i^{-1}$. We have
$$\tau\big(\D_1^{(A_{2N-1}^{(1)})}(x;q,t)\big)=\D_1^{(A_{2N-1}^{(2)})}(x;q,t). $$
However, the $A_{2N-1}^{(1)}$-Macdonald polynomials specialized via $\tau$ have a non-trivial decomposition onto the basis of 
$A_{2N-1}^{(2)}$-Macdonald polynomials.
\end{remark}

\subsubsection{Macdonald operator for $\g=D_{N+1}^{(2)}$ }
\label{D2Nopsec} Here, $(R,S) = (B_N,C_N)$.
The map $*: R\to S$ is given by $(e_i\pm e_j)_*=e_i\pm e_j$, and $(e_i)_*=2e_i$. There is a unique minuscule weight $\pi=\frac{1}{2}\sum_{i=1}^N e_i=\omega_N$ of type $S^\vee=B_N$, so that
$$
\Phi_{\omega_N}=\prod_{i=1}^N \frac{1-t x_i}{1-x_i} \prod_{1\leq i<j\leq N}\frac{1-t x_i x_j}{1-x_ix_j}.
$$
Summing over the Weyl group of type $C_N$ gives
\begin{equation}\label{D2NN}
{\mathcal E}_{\omega_N}^{(D_{N+1}^{(2)})}= \sum_{\epsilon_1,...,\epsilon_N=\pm 1} 
\prod_{i=1}^N \frac{1-t x_i^{\epsilon_i}}{1-x_i^{\epsilon_i}}
\prod_{1\leq i<j\leq N} \frac{1-t x_i^{\epsilon_i} x_j^{\epsilon_j}}{1-x_i^{\epsilon_i} x_j^{\epsilon_j}} \prod_{i=1}^N \Gamma_i^{\epsilon_i/2}.
\end{equation}
This will be identified as $\D_N^{(D_{N+1}^{(2)})}(x;q,t)={\mathcal E}_{\omega_N}^{(D_{N+1}^{(2)})}$.

\subsection{Higher order Koornwinder-Macdonald difference operators}
For generic parameters $(a,b,c,d,q,t)$, Koornwinder defined the first order $q$-difference operator whose eigenfunctions are the Koornwinder polynomials, invariant under the Weyl group of type $C$.
Consequently, van Diejen \cite{vandiej} defined a commuting family of higher order difference operators with the same eigenfunctions. We recall this construction in \ref{vandisec}. 
Using the spectrum of these operators, we construct in  \ref{combsec} 
linear combinations of these, such that their eigenvalues are proportional to 
elementary symmetric functions ${\hat e}_a(s)$. 
We add to this certain higher order operators due to Rains \cite{Rains}.
Upon specialization of the parameters (a,b,c,d), we combine this in Sect. \ref{usumacdo},
with Macdonald's construction of Section \ref{appmacsec}.

\subsubsection{van Diejen's higher order Koornwinder difference operators}
\label{vandisec}

\begin{defn}\label{defVandiej}
The van Diejen operator of order $m\in[1,N]$ is
 \begin{equation}\label{vandiealpha}
 {\mathcal V}_m^{(a,b,c,d)}:=\sum_{J\subset [1,N],\ |J|=m\atop \epsilon_j=\pm 1 \ \forall \ j\in J}\sum_{s=1}^m (-1)^{s-1}
 \sum_{\emptyset\subsetneq J_1\subsetneq \cdots \subsetneq J_s=J}
 \prod_{r=1}^s V_{\{x\},\{\epsilon\};J_r \setminus J_{r-1};K_r}^{(a,b,c,d)} (\prod_{j\in J_1} \Gamma_j^{\epsilon_j}-1),
\end{equation}
where $J_0=\emptyset$, $K_r=[1,N]\setminus J_r$ and
\begin{eqnarray}V_{\{x\},\{\epsilon\};J;K}^{(a,b,c,d)}&:=&
\prod_{i\in J} \frac{(1-a x_i^{\epsilon_i})(1- b x_i^{\epsilon_i})(1-c x_i^{\epsilon_i})(1-d x_i^{\epsilon_i})}{(1- x_i^{2\epsilon_i})(1- qx_i^{2\epsilon_i})} \nonumber \\
&&\times 
\prod_{{i<j}\atop {i,j\in J}}\frac{1-t x_i^{\epsilon_i}x_j^{\epsilon_j}}{1-x_i^{\epsilon_i}x_j^{\epsilon_j}}\frac{1-qt x_i^{\epsilon_i}x_j^{\epsilon_j}}{1-qx_i^{\epsilon_i}x_j^{\epsilon_j}} 
\prod_{i\in J\atop j\in K}  \frac{1-t x_i^{\epsilon_i}x_j}{1-x_i^{\epsilon_i}x_j}
 \frac{t x_i^{\epsilon_i}-x_j}{x_i^{\epsilon_i}-x_j},\label{defV}
\end{eqnarray}
where $J,K\subset [1,N]$ are such that $J\cap K=\emptyset$.
\end{defn}

When $m=1$, the van Diejen operator is the Koornwinder operators of Equation \eqref{koorn}.
\begin{example} 
Define
\begin{equation}\label{defAabcd}
A^{(a,b,c,d)}(x)=\frac{(1-ax)(1-bx)(1-cx)(1-dx)}{(1-x^2)(1-q x^2)}.\end{equation}
Then
\begin{eqnarray*}
{\mathcal V}_2^{(a,b,c,d)}&=&\sum_{1\leq i_1<i_2\leq N\atop \epsilon_1,\epsilon_2=\pm 1}\prod_{\ell=1}^2\left(
A^{(a,b,c,d)}(x_{i_\ell}^{\epsilon_\ell})
\prod_{k\neq i_1,i_2} \frac{1-t x_{i_\ell}^{\epsilon_\ell}x_k}{1-x_{i_\ell}^{\epsilon_\ell}x_k}
 \frac{t x_{i_\ell}^{\epsilon_\ell}-x_k}{x_{i_\ell}^{\epsilon_\ell}-x_k}\right)\\
&&\times \left\{ 
\frac{1-t x_{i_1}^{\epsilon_1}x_{i_2}^{\epsilon_2}}{1-x_{i_1}^{\epsilon_1}x_{i_2}^{\epsilon_2}}\frac{1-q t x_{i_1}^{\epsilon_1}x_{i_2}^{\epsilon_2}}{1-q x_{i_1}^{\epsilon_1}x_{i_2}^{\epsilon_2}}
(\Gamma_{i_1}^{\epsilon_1}\Gamma_{i_2}^{\epsilon_2}-1) -  \frac{1-t x_{i_1}^{\epsilon_1}x_{i_2}}{1-x_{i_1}^{\epsilon_1}x_{i_2}}
 \frac{t x_{i_1}^{\epsilon_1}-x_{i_2}}{x_{i_1}^{\epsilon_1}-x_{i_2}}(\Gamma_{i_1}^{\epsilon_1}-1)\right. \\
 && \qquad\qquad\qquad\left. - \frac{1-t x_{i_2}^{\epsilon_2}x_{i_1}}{1-x_{i_2}^{\epsilon_2}x_{i_1}}
 \frac{t x_{i_2}^{\epsilon_2}-x_{i_1}}{x_{i_2}^{\epsilon_2}-x_{i_1}} (\Gamma_{i_2}^{\epsilon_2}-1) \right\},
\end{eqnarray*}
where the sum over $s$ in \eqref{vandiealpha} decomposes into three terms, with $s=1, J_1=J=\{i_1,i_2\}$, 
$s=2, J_1=\{i_1\}, J_2=J=\{i_1,i_2\}$ and $s=2, J_1=\{i_2\}, J_2=J=\{i_1,i_2\}$.

\end{example}

The operators \eqref{vandiealpha} form a commuting family of difference operators
with common eigenfunctions being the Koornwinder polynomials.
We  give a description of their eigenvalues. 
Let $\sigma =\sqrt{\frac{abcd}{q}}$. 
Recall the elementary and complete symmetric functions $e_k(x)=s_{1^k}(x)$
and $h_k(x)=s_{m}(x)$.

\begin{defn}\label{ABCDeigen}
For arbitrary $\lambda_1,...,\lambda_N$, and $k,m\in [1,N]$, we define the collections of variables
\begin{eqnarray*}
u^{(k)}&:=&\{\sigma t^{k-i}\}_{1\leq i\leq k},\\
v&:=&\{ s_i+s_i^{-1}\}_{1\leq i\leq N},
\\
w_{(m)}&:=&\{ \sigma t^{N-i} +\sigma^{-1}t^{-(N-i)} \}_{m\leq i\leq N},
\end{eqnarray*}
with $s_i=\sigma t^{N-i} q^{\lambda_i}$ as usual,
and the functions
\begin{eqnarray}
d_{m}^{(k)}&:=&\sigma^m\,t^{m(k-\frac{m+1}{2})}\, {\hat e}_m(u^{(k)}),\label{Evalzero}\\
d_{\lambda;m}&:=& \sigma^m\,t^{m(N-\frac{m+1}{2})}\, {\hat e}_m(s),\label{Eval}\\
f_{\lambda;m}&:=& \sigma^m\,t^{m(N-\frac{m+1}{2})}\, \sum_{j=0}^m (-1)^j\, e_{m-j}(v)\, h_{j}(w_{(m)}) .\label{Fval} 
\end{eqnarray}
\end{defn}
The spectral theorem for van Diejen operators is
\begin{thm}\label{speckorthm}\cite{vandiej}
The (monic) symmetric Koornwinder polynomials $P_\lambda^{(a,b,c,d)}(x)$ satisfy
\begin{equation}\label{korevalpha}
{\mathcal V}_m^{(a,b,c,d)}\, P_\lambda^{(a,b,c,d)}(x) = f_{\lambda;m}\, P_\lambda^{(a,b,c,d)}(x) .
\end{equation}
\end{thm}

\subsubsection{Koornwinder-Macdonald operators}\label{MacKorsec}
\label{combsec}

Using the spectral theorem \ref{korevalpha}, we can construct  linear combinations of van Diejen's operators with eigenvalues  equal to $d_{\lambda;m}$. To do this we prove two combinatorial lemmas about symmetric functions.
 Given a set of variables $x=\{x_1,...,x_N\}$, we define associated set
${\tilde x}=\{x_i+x_i^{-1}, i\in[1,N]\}$,  and if $\beta\leq N$, define $x^{[\beta]}=\{x_1,...,x_{\beta}\}$ and $\tilde x^{[\beta]} = \{x_i+x_i^{-1}, i\in [1,\beta]\}$. In particular, $x^{[0]}=\emptyset$.

\begin{lemma}\label{lemeh}
For all $r\geq0$ and $n\geq 1$, 
\begin{equation}\label{theta}\theta_{r,n}(u):=\sum_{\ell=0}^{r}\, 
(-1)^\ell e_{r-\ell}(u^{[n-\ell]}) h_\ell(u^{[n-\ell+1]})=\delta_{r,0}.
\end{equation}
\end{lemma}
\begin{proof}
When $r=0$ the sum is trivially equal to 1. We consider $r>0$.
We define two types of integer configurations. A fermionic configuration on $\{1,...,p\}$ with $j$ particles, denoted by $F$, is a set of $j$ distinct integers in the set $[1,p]$. A bosonic configuration on $\{1,...,p'\}$ with $j'$ particles, denoted by $B$, is a sequence of $j'$ integers in $[1,p']$ which are not necessarily distinct. We assign a weight $w_F=\prod_{i\in F} u_i$ to a fermionic configuration, and a weight $w_B=\prod_{i\in B} (-u_i)$ to a bosonic configuration. Moreover, to the pair $(F,B)$, we assign a weight $w_{F,B}=w_F w_B$. The partition function of $j$ fermions on $[1,p]$ is $e_j(u^{[p]})$, and the partition function of $j'$ bosons on $[1,p']$ is $h_{j'}(-u^{[p']})=(-1)^{j'}h_{j'}(u^{[p']})$. 

We define the following set of pairs of fermionic and bosonic configurations:
$$\mathcal S_{n,r}=\left\{(F,B) \vert\, t_F:=\max(F)\leq n-|B|, t_B:=\max(B) \leq n-|B|+1, |F|+|B|=r\right\}$$
where $|X|$ is the cardinality of the set $X$, and if $|X|=0$ we define $\max(X)=0$.
The identity \eqref{theta} is an identity for the partition function:
$$
Z_{n,r} := \sum_{(F,B)\in \mathcal S_{n,r}} w_{F,B} = \delta_{r,0}.
$$

To prove this, for any $r\geq1$ we construct a fixed point-free involution $\Phi$ on the set $\mathcal S_{n,r}$, such that $\Phi(w_F w_B) = - w_F w_B$.
If such an involution exists, the partition function for any $r>0$ vanishes:
$$\sum_{(F,B)\in \mathcal S_{n,r}} w_{(F,B)}=\frac{1}{2}\sum_{(F,B)\in \mathcal S_{n,r}} (w_{(F,B)}+w_{\Phi(F,B)})=0\ .$$

\begin{figure}
\includegraphics[width=.9\textwidth]{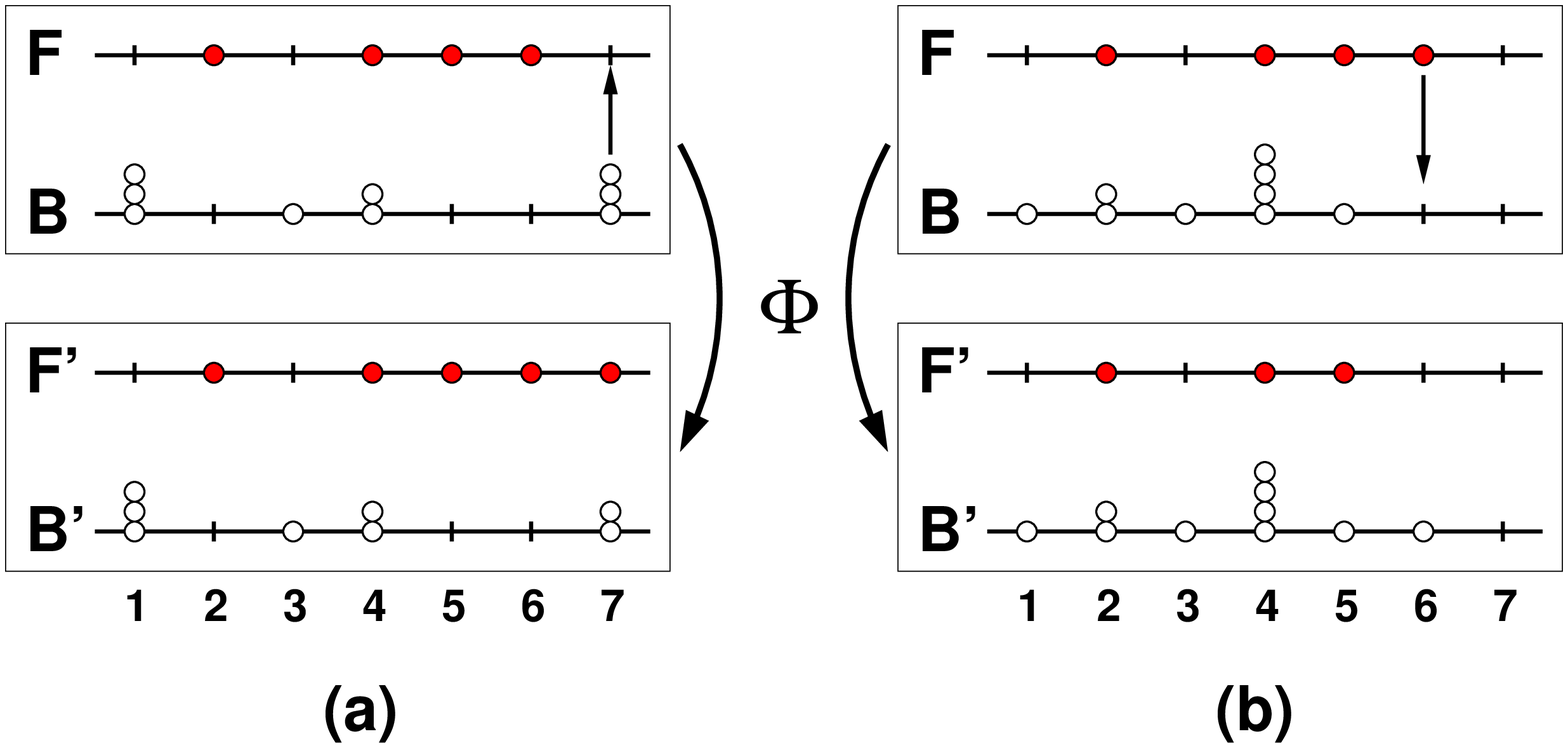}
\caption{An illustration of the involution $\Phi$. Case (a) has $t_F=6<t_B=7$, hence we move the topmost rightmost bosonic particle to a fermionic particle at position $t_F'=t_B$. Case (b) has $t_F=6\geq t_B=5$, hence we move the fermionic particle to a bosonic one at position $t_{B'}=t_F$.}\label{PhiPic}
\end{figure}

The involution $\Phi$ is illustrated in Figure \ref{PhiPic} in the language of particles, namely by considering $F,B$ as the sets of integer coordinates of particles along the integer line: e.g. in the case of Fig. \ref{PhiPic}(a) we have $F=\{2,4,5,6\}$ and $B=\{1,1,1,3,4,4,7,7,7\}$. Let $(F,B)\in\mathcal S_{n,r}$.
The map $\Phi$ acts by moving one particle between $F$ and $B$, thus preserving $|F|+|B|=r$ and reversing the sign of the weight. It is defined as $\Phi(F,B) = (F',B')$, where
\begin{itemize}
\item[(a)] If $t_B>t_F$:
{ $F'=F\cup t_B$, $B'=B\setminus t_B$. Since $|B'|=|B|-1$, $t_{F'} = t_B \leq n-|B|+1 = n-|B'|$ and $t_{B'} \leq t_B \leq n-|B|+1 = n-|B'| < n-|B'|+1$.  Therefore, $(F',B')\in\mathcal S_{n,r}$,  while $w_{F',B'}=-w_{F,B}$.
}
\item[(b)] If $t_B\leq t_F$:
{ $F' = F \setminus t_F$ and $B'=B\cup t_F$. Then $t_{F'} \leq t_{F}-1 \leq n-|B|-1 = n-|B'|$ and 
$t_{B'} = t_F \leq n-|B| = n-|B'|+1$, so that $(F',B')\in\mathcal S_{n,r}$, and $w_{F',B'}=-w_{F,B}$.
}
\end{itemize}
The map $\Phi$ is clearly an involution. When $r>0$, $\Phi$ has no fixed points,
since one can always move a particle. The Lemma follows.
\end{proof}

\begin{lemma}\label{finlutcomb}
There is an identity on symmetric functions:
\begin{equation}\label{genid}
{\hat e}_m(x)=\sum_{j=0}^m {\hat e}_{m-j}({ y}^{[N-j]}) 
\sum_{k=0}^{j} (-1)^k e_{j-k}({\tilde x}) h_{k}({\tilde  y}^{[N-j+1]}).
\end{equation}
\end{lemma}
\begin{proof}
We rewrite \eqref{genid} using the generating function ${\hat E}(z;x)$ of \eqref{Chatdef} as
\begin{equation}\label{etilde}{\hat E}(z;x)\vert_{z^m}=\sum_{j=0}^m {\hat E}(z;{y}^{[N-j]})
\vert_{z^{m-j}} F(z;{\tilde x},{\tilde y}^{[N-j+1]})\vert_{z^{j}},\end{equation}
where $f(z)\vert_{z^m}$ is the coefficient of $z^m$ in the series expansion of $f(z)$ around $0$, and
$${\hat E}(z;{y}^{[\beta]})=\prod_{i=1}^{\beta}(1+{\tilde y}_i z+z^2), \quad 
F(z;{\tilde x},{\tilde y}^{[\beta]})=\frac{\prod_{i=1}^N 1+z\tilde{x}_i}{\prod_{i=1}^{\beta} 1+z\tilde{y}_i} \ .$$
One can further decompose
$$\prod_{i=1}^{N-j} (1+z {\tilde y}_i+z^2)=(1+z^2)^{N-j} {\hat E}(\frac{z}{1+z^2};\tilde{y}) =\sum_{k=0}^{N-j} z^k\,(1+z^2)^{N-j-k} \,e_k(\tilde{y}^{[N-j]}),$$
and
$$\frac{\prod_{i=1}^N (1+z {\tilde x}_i)}{\prod_{i=1}^{N-j+1} (1+z {\tilde y}_i)}=\prod_{i=1}^N (1+z {\tilde x}_i) \sum_{\ell\geq 0} (-1)^\ell z^\ell h_\ell(\tilde{y}^{[N-j+1]}).$$
The right hand side of  \eqref{etilde} is therefore
\begin{eqnarray}
&&\sum_{j=0}^m  \sum_{k,\ell} (1+z^2)^{N-j-k} \big\vert_{z^{m-j-k}}\, \prod_{i=1}^N(1+z {\tilde x}_i)\vert_{z^{j-\ell}}\, (-1)^\ell e_k(\tilde{y}^{[N-j]})
h_\ell({\tilde  y}^{[N-j+1]})\nonumber \\
&=&\sum_{r=0}^m  \sum_{k,\ell \geq 0\atop k+\ell=r}\, \sum_{j=\ell}^{m-k} (1+z^2)^{N-j-k} \big\vert_{z^{m-j-k}}\, 
\prod_{i=1}^N(1+z {\tilde x}_i)\vert_{z^{j-\ell}}\, 
(-1)^\ell e_k(\tilde{y}^{[N-j]}) h_\ell(\tilde{y}^{[N-j+1]}),\nonumber \\
&=&\sum_{r=0}^m  \sum_{j=0}^{m-r} (1+z^2)^{N-j-r} \big\vert_{z^{m-j-r}}\, 
\prod_{i=1}^N(1+z {\tilde x}_i)\vert_{z^{j}}\, \sum_{\ell =0}^{r}\, 
(-1)^\ell e_{r-\ell}(\tilde{y}^{[N-j-\ell]}) h_\ell(\tilde{y}^{[N-j-\ell+1]}),\nonumber \\
&&\label{presquefini}
\end{eqnarray}
where we changed variables $k\mapsto r-\ell$ and $j\mapsto j-\ell$.
Finally, using Lemma \ref{lemeh} for the collection $u={\tilde y}$ and $n=N-j$, the expression above drastically simplifies into
$$\sum_{j=0}^{m} (1+z^2)^{N-j} \big\vert_{z^{m-j}}\, 
\prod_{i=1}^N(1+z {\tilde x}_i)\vert_{z^{j}}=\prod_{i=1}^N (1+z {\tilde x}_i+z^2)\big\vert_{z^m} ={\hat e}_m(x),$$
which implies \eqref{etilde}. 
\end{proof}

\begin{defn}\label{malphadef} We define the Koornwinder-Macdonald difference operators
as 
$${\mathcal D}_m^{(a,b,c,d)}(x):= \sum_{j=0}^m d_{j}^{(N-m+j)}\, {\mathcal V}_{m-j}^{(a,b,c,d)},\qquad  \quad
(m=1,2,...,N),$$
in terms of the van Diejen operators of Sect.\ref{vandisec}, with $d_{j}^{(k)}$ as in \eqref{Evalzero}.
\end{defn}

\begin{thm}\label{ABCDmacdopol}
The Koornwinder-Macdonald polynomials are common eigenfunctions of  ${\mathcal D}_m$ with eigenvalue $d_{\lambda;m}$ defined in \eqref{Eval}:
\begin{equation}\label{evalgen}
 {\mathcal D}_m^{(a,b,c,d)}(x;q,t)\, P_\lambda^{(a,b,c,d)}(x)= d_{\lambda;m}\, P_\lambda^{(a,b,c,d)}(x) .
 \end{equation}
\end{thm}
\begin{proof}
Equation \eqref{evalgen} follows from Theorem \ref{speckorthm} and the relation
\begin{equation}\label{tobeproved} d_{\lambda;m} =\sum_{j=0}^m d_{j}^{(N-m+j)}\, f_{\lambda;m-j} .
\end{equation}
Equation \eqref{tobeproved} is obtained by specializing Lemma \ref{finlutcomb} to the variables
$x_i=\sigma t^{N-i}q^{\lambda_i}$ and $y_i=\sigma t^{N-i}$, and 
noting that the prefactors in (\ref{Eval}-\ref{Fval}) amount to an overall factor of $\sigma^m t^{m(N-\frac{m+1}{2})}$ 
on both sides of the equation.
\end{proof}

\subsubsection{Rains operators}
\label{rainsec}

The Rains operator $\widehat{\mathcal D}_N^{(a,b,c,d)}(x;q,t)$ defined in  (\ref{rains},\ref{rainsop}) has eigenvalues \eqref{eigenrains}.
This operator is not linearly independent of the K-M operators, as is seen from the following Lemma.
\begin{lemma}\label{rainstokor}
For generic $(a,b,c,d)$ the following relation holds between the Koornwinder-Macdonald operators and the Rains operator:
\begin{equation}
\widehat{\mathcal D}_N^{(a,b,c,d)}(x;q,t)={\mathcal D}_{N}^{(a,b,c,d)}(x;q,t)+\sum_{m=1}^N (-1)^m\left(a^mb^m +\frac{c^md^m}{q^m}\right)\,t^{m(m-1)/2} \, 
{\mathcal D}_{N-m}^{(a,b,c,d)}(x;q,t).
\end{equation}
\end{lemma}
\begin{proof}
This is a consequence of the relation between the eigenvalues:
\begin{equation}\label{eigenrelations}
\hat{d}_{\lambda;N}^{(a,b,c,d)}=d_{\lambda,N}+\sum_{m=1}^N (-1)^m\left(a^mb^m +\frac{c^md^m}{q^m}\right)\,t^{m(m-1)/2} \,d_{\lambda;N-m} .
\end{equation}
Using the notation $\sigma=\sqrt{abcd/q}$, $u=ab/\sigma$, $u^{-1}=c d/q\sigma$ and
$s_i=\sigma q^{\lambda_i}t^{N-i}$, as well as the expressions \eqref{Eval} for $d_{\lambda,m}$ and \eqref{eigenrains} for $\hat{d}_{\lambda;N}$, Equation \eqref{eigenrelations} is
\begin{eqnarray*} 
&&q^{-|\lambda|} \prod_{i=1}^N (1-u s_i)(1-u^{-1} s_i)=\sigma^N\, t^{N(N-1)/2}\left\{ \hat{e}_N(s)+\sum_{m=1}^N (-1)^m (u^m+u^{-m}) \hat{e}_{N-m}(s) \right\}\\
&&\ \ =\frac{(-1)^N \sigma^N\, t^{N(N-1)/2}}{u^N} \sum_{m=0}^{2N} (-1)^m u^m \hat{e}_m(s)
=\frac{(-1)^N \sigma^N\, t^{N(N-1)/2}}{u^N}\prod_{i=1}^N(1-u s_i)(1-u s_i^{-1}),
\end{eqnarray*}
where we have used \eqref{Chatdef} in the last line.
The equality follows from $\prod_{i=1}^N s_i =q^{|\lambda|} \sigma^N\, t^{N(N-1)/2}$.
\end{proof}

\subsubsection{Koornwinder-Macdonald operators and Macdonald's operators of Section \ref{appmacsec}}
\label{usumacdo}

The first Koornwinder-Macdonald operator in Equation of Definition \ref{malphadef} is
\begin{eqnarray}
{\mathcal D}_1^{(a,b,c,d)}(x;q,t)&=&{\mathcal V}_1^{(a,b,c,d)}(x;q,t)+\frac{1-t^N}{1-t}\left( 1+ \frac{abcd}{q}t^{N-1}\right),
\nonumber \\
&=&
\K_1^{(a,b,c,d)}(x;q,t)+\frac{1-t^N}{1-t}\left( 1+ \frac{abcd}{q}t^{N-1}\right).\label{MaKop}
\end{eqnarray}

To relate these with some of the operators of Section \ref{appmacsec}, we have the following Lemma.

\begin{lemma}\label{constlem}
The first order  Koornwinder-Macdonald operator can be written as
\begin{equation}
\qquad {\mathcal D}_1^{(a,b,c,d)}(x;q,t)= 
\varphi^{(a,b,c,d)}(x)+
\sum_{i=1}^N\sum_{\epsilon=\pm 1}  \Phi_{i,\epsilon}^{(a,b,c,d)}(x) \Gamma_i^{\epsilon},\nonumber 
\end{equation}
where
\begin{eqnarray}
\varphi^{(a,b,c,d)}(x)&=&\sum_{\epsilon=\pm 1} \frac{\big(1-\frac{\epsilon}{ q^{1/2}}a\big)
\big(1-\frac{\epsilon }{q^{1/2}}b\big)
\big(1-\frac{\epsilon }{q^{1/2}}c\big)\big(1- \frac{\epsilon }{q^{1/2}}d\big)}{2(1-t)(1-q^{-1}t)}\nonumber \\
&&\qquad \qquad\qquad \times 
\left\{ \prod_{i=1}^N  
\frac{1-\frac{\epsilon\, t }{q^{1/2}}x_i}{1-\frac{\epsilon }{q^{1/2}}x_i}\,
\frac{1-\frac{\epsilon\, t }{q^{1/2}}x_i^{-1}}{1-\frac{\epsilon }{q^{1/2}}x_i^{-1}}
-t^N\right\} .\label{valphi}
\end{eqnarray}
\end{lemma}
\begin{proof}
We have:
\begin{equation}\label{phiexp}
 \varphi^{(a,b,c,d)}(x)=\frac{1-t^N}{1-t}\left( 1+ \frac{abcd}{q}t^{N-1}\right)-\sum_{i=1}^N\sum_{\epsilon=\pm 1}
\Phi_{i,\epsilon}^{(a,b,c,d)}(x) .
\end{equation}
Using the simple fraction decomposition of the function
\begin{equation}\label{deftheta}
\theta(z):= 
\frac{\big(1-\frac{z}{a}\big)\big(1-\frac{z}{b}\big)\big(1-\frac{z}{c}\big)\big(1-\frac{z}{d}\big)}{\big(1-\frac{z^2}{q}\big)\big(1-\frac{z^2}{t}\big)} \, \prod_{i=1}^N\frac{1-\frac{z}{t} x_i}{1-z x_i}\, \frac{1-\frac{z}{t} x_i^{-1}}{1-z x_i^{-1}} .
\end{equation}
We find
$$\theta(z)=\frac{q}{abcd\,t^{2N-1}}+\sum_{\epsilon=\pm 1}\left\{ \sum_{i=1}^N \frac{A_{i,\epsilon}}{1-z x_i^{\epsilon}} +\frac{B_\epsilon}{1-\epsilon \frac{z}{q^{1/2}}}+\frac{C_\epsilon}{1-\epsilon \frac{z}{t^{1/2}}}\right\},$$
where
\begin{eqnarray*}
A_{i,\epsilon}&=&\frac{1-t}{t^{2N}}\,\frac{\big(1-\frac{x_i^{-\epsilon}}{a}\big)\big(1-\frac{x_i^{-\epsilon}}{b}\big)\big(1-\frac{x_i^{-\epsilon}}{c}\big)\big(1-\frac{x_i^{-\epsilon}}{d}\big)}{\big(1-\frac{x_i^{-2\epsilon}}{q}\big)\big(1-\frac{x_i^{-2\epsilon}}{t}\big)} \,
\prod_{j\neq i} \frac{t x_i^{\epsilon}-x_j}{x_i^{\epsilon}-x_j}\, \frac{t x_i^{\epsilon}x_j-1}{x_i^{\epsilon}x_j-1},\\
&=&
\frac{q(1-t)}{abcd\, t^{2N-1}}\, \Phi_{i,\epsilon}^{(a,b,c,d)}(x) ,
\end{eqnarray*}
and
\begin{eqnarray*}
B_\epsilon&=&\frac{\big(1-\epsilon\frac{a}{q^{1/2}}\big)\big(1-\epsilon\frac{b}{q^{1/2}}\big)\big(1-\epsilon\frac{c}{q^{1/2}}\big)\big(1-\epsilon\frac{d}{q^{1/2}}\big)}{2\,abcd\, t^{2N-1}\,\big(1-\frac{t}{q}\big)}\prod_{i=1}^N 
\frac{1-\frac{\epsilon\, t }{q^{1/2}}x_i}{1-\frac{\epsilon }{q^{1/2}}x_i}\,
\frac{1-\frac{\epsilon\, t }{q^{1/2}}x_i^{-1}}{1-\frac{\epsilon }{q^{1/2}}x_i^{-1}},\\
C_\epsilon&=&q\frac{\big(1-\epsilon\frac{a}{t^{1/2}}\big)\big(1-\epsilon\frac{b}{t^{1/2}}\big)\big(1-\epsilon\frac{c}{t^{1/2}}\big)\big(1-\epsilon\frac{d}{t^{1/2}}\big)}{2\,abcd\, t^{N-1}\,\big(1-\frac{q}{t}\big)}.
\end{eqnarray*}
By definition \eqref{deftheta}, $\theta(0)=1$, therefore
\begin{equation}\label{interm}
\sum_{\epsilon=\pm1}\left\{ \sum_{i=1}^N \Phi_{i,\epsilon}^{(a,b,c,d)}(x)+\frac{abcd\, t^{2N-1}}{q(1-t)}\big(B_\epsilon+C_{\epsilon}\big)\right\}=\frac{\frac{abcd}{q}\, t^{2N-1}-1}{1-t}.
\end{equation}
Using
\begin{eqnarray*}&&\sum_{\epsilon=\pm 1} \left\{\frac{t}{2}\big(1-\epsilon\frac{a}{t^{1/2}}\big)\big(1-\epsilon\frac{b}{t^{1/2}}\big)\big(1-\epsilon\frac{c}{t^{1/2}}\big)\big(1-\epsilon\frac{d}{t^{1/2}}\big)
\right.\\
&&\left. \qquad \qquad 
- \frac{q}{2}\big(1-\epsilon\frac{a}{q^{1/2}}\big)\big(1-\epsilon\frac{b}{q^{1/2}}\big)\big(1-\epsilon\frac{c}{q^{1/2}}\big)\big(1-\epsilon\frac{d}{q^{1/2}}\big)\right\}=\left(1-\frac{q}{t}\right)\left(t-\frac{abcd}{q}\right),
\end{eqnarray*}
we have:
\begin{eqnarray*}
&&\sum_{\epsilon=\pm 1}\frac{abcd\, t^{2N-1}}{q(1-t)}\,C_{\epsilon}=\frac{t^N}{2(1-t)\big(1-\frac{q}{t}\big)}
\sum_{\epsilon=\pm 1} \big(1-\epsilon\frac{a}{t^{1/2}}\big)\big(1-\epsilon\frac{b}{t^{1/2}}\big)\big(1-\epsilon\frac{c}{t^{1/2}}\big)\big(1-\epsilon\frac{d}{t^{1/2}}\big),\\
&&=\frac{t^N- \frac{abcd}{q}t^{N-1}}{1-t}-\frac{t^N}{2(1-t)\big(1-\frac{t}{q}\big)}\sum_{\epsilon=\pm 1} \big(1-\epsilon\frac{a}{q^{1/2}}\big)\big(1-\epsilon\frac{b}{q^{1/2}}\big)\big(1-\epsilon\frac{c}{q^{1/2}}\big)\big(1-\epsilon\frac{d}{q^{1/2}}\big).
\end{eqnarray*}
The Lemma follows by substituting this into \eqref{interm}, and using the result to reexpress $\varphi^{(a,b,c,d)}(x)$
as given by \eqref{phiexp}.
\end{proof}

When  $c=q^{1/2}$ and $d=-q^{1/2}$,  $\varphi^{(a,b,q^{1/2},-q^{1/2})}(x)=0$. Using Table \ref{korspec}, we obtain the following.
\begin{cor}\label{lesmemes}
In the cases $\g=D_N^{(1)},B_N^{(1)}, A_{2N-1}^{(2)}$ the first Macdonald operator is equal to the first specialized Koornwinder-Macdonald operator:
\begin{equation}\label{simpD}
\D_1^{(\g)}(x;q,t)=\sum_{i=1}^N\sum_{\epsilon=\pm 1} \Phi_{i,\epsilon}^{(\g)}(x) \Gamma_i^{\epsilon}={\mathcal E}_1^{(\g)}(x;q,t).
\end{equation}
\end{cor}

By direct inspection, the specialized Rains operators of \eqref{rains} and \eqref{rainsop} can also be expressed in terms Macdonald's operators.
\begin{lemma}\label{lesautres}
In the cases $\g=D_{N}^{(1)}, C_N^{(1)}, D_{N+1}^{(2)}$ we have the identifications
\begin{equation}
\renewcommand{\arraystretch}{1.5}
\begin{array}{ll}
D_{N}^{(1)}:&{\mathcal E}_{\omega_N}^{(D_N^{(1)})}(x;q,t) +{\mathcal E}_{\omega_{N-1}}^{(D_N^{(1)})}(x;q,t)= {\mathcal R}_N^{(1,-1)}(x;q,t),\\ &
 \widehat{\mathcal D}_N^{(1,-1,q^{1/2},-q^{1/2})}=\left({\mathcal R}_N^{(1,-1)}\right)^2;\\
C_N^{(1)}:&{\mathcal E}_{\omega_N}^{(C_N^{(1)})}(x;q,t)={\mathcal R}_N^{(t^{1/2},-t^{1/2})}(x;q,t),\\
&\widehat{\mathcal D}_N^{(t^{1/2},-t^{1/2},t^{1/2}q^{1/2},-t^{1/2}q^{1/2})}=\left({\mathcal R}_N^{(t^{1/2},-t^{1/2})}\right)^2;\\
D_{N+1}^{(2)}:&{\mathcal E}_{\omega_N}^{(D_{N+1}^{(2)})}(x;q,t)={\mathcal R}_N^{(t,-1)}(x;q,t),\\ 
&\widehat{\mathcal D}_N^{(t,-1,t q^{1/2},-q^{1/2})}=\left({\mathcal R}_N^{(t,-1)}\right)^2.
\end{array}
\end{equation}
\end{lemma}

\subsection{$\g$-Macdonald operators and their eigenvalues}
\label{gmacsec}

For each $\g$, we have chosen a list of $N$ linearly independent commuting difference operators (see Def. \ref{gmacdodef}) using the operators in the preceding sections of this Appendix. These have the property that in the $q$-Whittaker limit, they and their time-translates provide solutions of the various quantum Q-systems.
We first explore redundancies between the various definitions to justify Def. \ref{gmacdodef}. Then we describe the eigenvalues of the $\g$-Macdonald operators, and conclude with a remark on the choice leading to Def. \ref{gmacdodef}.

\subsubsection{Redundancies}

Specializing the parameters $(a,b,c,d)$ as in Table \ref{tableone} in the difference operators of Theorem \ref{ABCDmacdopol} gives 
a list of $N$ commuting difference operators for each $\g$.
As noted in Corollary \ref{lesmemes}, some of these are Macdonald's operators listed in Section \ref{appmacsec}.
In addition, there are non-linear relations\footnote{These relations are proved using the same argument as in the proof of Theorem \ref{ABCDmacdopol}, using identities between eigenvalues. Some are a direct consequence of the identifications with Rains operators from Theorem \ref{lesautres}.}:
\begin{eqnarray*}
{\mathcal E}_{\omega_N}^{(D_N^{(1)})}\,{\mathcal E}_{\omega_{N-1}}^{(D_N^{(1)})}&=&\sum_{\al=0}^{\lfloor \frac{N-1}{2} \rfloor} t^{\al(2\al+1)}\, {\mathcal D}^{(a,b,c,d)}_{N-2\al-1},\\
({\mathcal E}_{\omega_N}^{(D_N^{(1)})})^2+({\mathcal E}_{\omega_{N-1}}^{(D_N^{(1)})})^2&=&{\mathcal D}_{N}^{(a,b,c,d)}
+2\sum_{\al=0}^{\lfloor \frac{N}{2} \rfloor} t^{\al(2\al-1)}\,  {\mathcal D}^{(a,b,c,d)}_{N-2\al},
\end{eqnarray*}
for the $D_N^{(1)}$ specialization. 
Similarly, for the $C_N^{(1)}$ and $D_{N+1}^{(2)}$ specializations,
$$\left({\mathcal E}_{\omega_N}^{(\g)}\right)^2 ={\mathcal D}_N^{(a,b,c,d)}+2\sum_{\al=1}^N t^{\frac{\al(\al+1)}{2}} \, {\mathcal D}^{(a,b,c,d)}_{N-\al}.$$
In those cases, the Macdonald operators carry more information than the specialization of the Koornwinder-Macdonald operators ${\mathcal D}_m^{(a,b,c,d)}$. This justifies the choice in Def. \ref{gmacdodef}.

\subsubsection{The eigenvalues of $\g$-Macdonald operators}
\label{listsec}

\begin{defn}
If $m\leq N_\g$,
let $d_{\lambda,m}^{(\g)}$ to be the specialization of Equation \ref{Eval} to the parameters corresponding to $\g$ in Table \ref{tableone}:
\begin{equation}\label{newevaldef}
d_{\lambda,m }^{(\g)}:=d_{\lambda,m}=t^{m(N+\xi_\g-\frac{m+1}{2})}\, \hat e_m(s), \quad m\in[1,N_\g],
\end{equation}
where $s=t^{\rho^{(\g)}} q^\lambda$, i.e. $s_i=q^{\lambda_i} t^{N-i+\xi_\g}$.
Otherwise, define
\begin{equation}\label{evalex}
\renewcommand{\arraystretch}{1.5}
\begin{array}{ll}
\g=D_N^{(1)}:& d_{\lambda;\beta}^{(D_{N}^{(1)})}=t^{\frac{N(N-1)}{4}}\,{\hat e}_\beta^{(D_{N})}(s), \qquad  \beta=N-1,N ,\\
\g=C_N^{(1)}, D_{N+1}^{(2)}:&  d_{\lambda;N}^{(\g)}=t^{\frac{N(N+1)}{4}}\,{\hat e}_N^{(B_N)}(s),
\end{array}
\end{equation}
where
\begin{eqnarray}\label{restofehat}
{\hat e}_\beta^{(D_{N})}(x)&:=& \sum_{\epsilon_1,...,\epsilon_N=\pm1 \atop \prod \epsilon_i=2(\beta-N)+1} \prod_{i=1}^N x_i^{\epsilon_i/2}\qquad (\beta=N-1,N),\nonumber \\
{\hat e}_N^{(B_N)}(x)&:=& \prod_{i=1}^N \frac{1+x_i}{\sqrt{x_i}}={\hat e}_{N-1}^{(D_{N})}(x)+{\hat e}_N^{(D_{N})}(x).
\end{eqnarray}
\end{defn}
The functions $\hat e_m^{(R^*)}(s)$ have the property that their dominant monomial (in powers of $t$) is
$s^{\omega_m^*}$, where $\omega_m^*$ are the fundamental weights of $R^*$.

\begin{thm}\label{specmacG}
The $\g$-specialized Koornwinder polynomials are common eigenvectors of the $\g$-Macdonald operators, satisfying the eigenvalue equation
\begin{equation}\label{geneval}
 {\mathcal D}^{(\g)}_m(x)\, P_\lambda^{(\g)}= d_{\lambda;m}^{(\g)}\, P_\lambda^{(\g)} , \qquad m\in[1,N],
 \end{equation}
 with
 $d_{\lambda;m}^{(\g)}=\theta_m^{(\g)}\, \hat{e}_m^{(R^*)}(s)$
 ($\hat{e}_m^{(R^*)}= \hat{e}_m$ for $m\leq N_\g$)
 and 
\begin{equation}
 \theta_m^{(\g)}=t^{m(N+\xi_\g-\frac{m+1}{2})},\qquad m\leq N_\g \label{defthone} ,
 \end{equation}
 \begin{equation}\label{defthtwo}
 \begin{array}{c}
 \theta_\beta^{(D_{N}^{(1)})}=t^{\frac{N(N-1)}{4}}, \qquad \beta=N-1,N,\\
 \theta_N^{(\g)}=t^{\frac{N(N+1)}{4}}, \qquad \g=C_N^{(1)}, D_{N+1}^{(2)}.
\end{array}
\end{equation}
\end{thm}

\subsubsection{Remark about our choice of Macdonald operators}\label{aboutsec}
We made the choice of $\g$-Macdonald operators for simplicity. A more natural choice is to choose eigenvalues proportional to the fundamental characters of $R^*$. This leads to a more complicated choice of the difference operators, but in the $q$-Whittaker limit, they have the same limit as our $\g$-Macdonald operators.

For Macdonald's operators in Section \ref{appmacsec}, 
the eigenvalue equation is
$$ {\mathcal D}_m^{(\g)}(x)\, P_\lambda^{(\g)}(x)= \theta_m^{(\g)}\,s_{\omega_m}^{(R^*)}(s)\, P_\lambda^{(\g)}(x) ,$$
where the Schur functions $s_{\omega_m}^{(R^*)}$ are the fundamental characters of $R^*$ with highest weight $\omega_m^*$.
We may choose the set of difference operators $\tilde{\mathcal D}_m^{(\g)}(x)$, with
 eigenvalues proportional to the fundamental characters of $R^*$ for all $m$ as above.
These are related to ${\mathcal D}_m^{(\g)}$ via a triangular change of basis, as can be seen from the relation between the fundamental characters and $e_m^{(R)}$:
\begin{eqnarray*}
D_N: \quad s_{\omega_m}^{(D_N)}&=&\hat{e}_m^{(D_N)} ,\qquad m\in[1,N],\\
B_N: \quad 
s_{\omega_m}^{(C_N)}&=&\hat{e}_m-\hat{e}_{m-2} , \qquad m\in[1,N],\\
C_N: \quad 
s_{\omega_m}^{(B_N)}&=&\hat{e}_m+\hat{e}_{m-1}  \qquad m\in[1,N-1],\\
s_{\omega_N}^{(B_N)}&=&\hat{e}_N^{(B_N)}.
\end{eqnarray*}
Therefore
$\tilde{\mathcal D}^{(\g)}$ are
\begin{eqnarray}
\g=D_N^{(1)}: \quad \tilde{\mathcal D}^{(\g)}_m&:=& {\mathcal D}^{(\g)}_m \qquad (m=1,2,...,N),\nonumber \\
\g=B_N^{(1)},A_{2N-1}^{(2)},A_{2N}^{(2)}: \quad \tilde{\mathcal D}^{(\g)}_m&:=&{\mathcal D}^{(\g)}_m
-(1-\delta_{m,1})\frac{\theta_m^{(\g)}}{\theta_{m-2}^{(\g)}}\, {\mathcal D}^{(\g)}_{m-2} \quad (m=1,2,...,N),\nonumber \\
\g=C_N^{(1)},D_{N+1}^{(2)}: \quad \tilde{\mathcal D}^{(\g)}_m&:=&{\mathcal D}^{(\g)}_m
+\frac{\theta_m^{(\g)}}{\theta_{m-1}^{(\g)}}\, {\mathcal D}^{(\g)}_{m-1}\quad (m=1,2,...,N-1),\nonumber\\
\tilde{\mathcal D}^{(\g)}_N&:=&{\mathcal D}^{(\g)}_N ,\label{tildeDdef}
\end{eqnarray}
with the convention that ${\mathcal D}^{(\g)}_0=1$.
This choice guarantees the following:
\begin{thm}
For all $m=1,2,...,N$ and all $\g$,
\begin{equation}\label{spectilde}
\tilde{\mathcal D}^{(\g)}_m\, P_\lambda(x)=\theta_m^{(\g)}\, s_{\omega_m}^{(R^*)}(s)\, P_\lambda(x), \qquad s=q^\lambda \,t^{\rho^{(\g)}}.
\end{equation}
\end{thm}

\subsection{$q$-Whittaker limit of the $\g$-Macdonald operators}
\label{appqwhit}
The $q$-Whittaker limit corresponds to sending $t\to \infty$. 
The quantity $A^{(a,b,c,d)}(x)$ of \eqref{defAabcd}, which appears in the van Diejen operators \eqref{vandiealpha}, tends to $\sigma = t^{\xi_\g}$ under the specialization of Table \ref{tableone}.
All terms in \eqref{vandiealpha} have the same leading behavior as that with
$s=1$ and $J=J_1=\{1,2,...,m\}$, namely a sum of three contributions from \eqref{defV}:
$$ m \xi_\g + m(m-1) + 2 m(N-m)=m(2N-m-1+\xi_\g). $$
Therefore, ${\mathcal V}_m^{(\g)} \sim (\theta_m^{(\g)})^2\, V_m^{(\g)}$, where the difference operator
$V_m^{(\g)}$ is independent of $t$. Similarly ${\mathcal D}_m^{(\g)}\sim (\theta_m^{(\g)})^2\, D_m^{(\g)}$
where $D_m^{(\g)}$ independent of $t$, for $m=1,2,...,N_\g$. By inspection, 
we find the same leading behavior for $N_\g<m\leq N$, using \eqref{defthtwo}. This leads to the definitions:
\begin{equation}\label{DlimW}D_m^{(\g)}(x):=\lim_{t\to \infty} (\theta_m^{(\g)})^{-2} \,{\mathcal D}_m^{(\g)}(x),\ \ V_m^{(\g)}:=\lim_{t\to\infty}  (\theta_m^{(\g)})^{-2}\, {\mathcal V}_m^{(\g)}\end{equation}
with $\theta_m^{(\g)}$ as in (\ref{defthone}-\ref{defthtwo}).

\begin{remark}\label{limDrem}
The limit $t\to \infty$ 
of ${\mathcal D}^{(\g)}_m$ is the {\it same} as that of $\tilde{\mathcal D}^{(\g)}_m$ of Sect. \ref{aboutsec}:
Using ${\mathcal D}^{(\g)}_m\simeq (\theta_m^{(\g)})^{2}\, {D}^{(\g)}_m$ at leading order in $t$,
the statement follows immediately from \eqref{tildeDdef} by noting that for $m\leq N_\g$:
\begin{eqnarray*}
\lim_{t\to\infty} \frac{\theta_{m-1}^{(\g)}}{\theta_{m}^{(\g)}}&=&0 \quad (m\geq 1;\g=C_N^{(1)},D_{N+1}^{(2)}),\\
\lim_{t\to\infty} \frac{\theta_{m-2}^{(\g)}}{\theta_{m}^{(\g)}}&=&0 \quad (m\geq 2;\g=B_N^{(1)},A_{2N-1}^{(2)},A_{2N}^{(2)}).
\end{eqnarray*}
\end{remark}

For any $\g$, define
$\Pi_\lambda^{(\g)}=\lim_{t\to \infty} t^{-\rho^*\cdot\lambda} P_\lambda^{(\g)}$.
We refer to these as $q$-Whittaker functions, although they coincide with the usual definition only for the 
case of $\g=D_N^{(1)}$ and the twisted algebras in Table \ref{tableone}. The spectral theorem for $D_m^{(\g)}$ is as follows.
\begin{thm}\label{qwithakthm}
The functions $\Pi^{(\g)}_\lambda$ are common eigenfunctions of $D_m^{(\g)}$, with
$$ D_{m}^{(\g)}(x)\, \Pi^{(\g)}_\lambda(x) =\Lambda^{\omega_m^*} \,  \Pi^{(\g)}_\lambda(x).$$
\end{thm}
\begin{proof}
The eigenvalues are  the limits
$$\delta_{\lambda;m}^{(\g)}=\lim_{t\to \infty} (\theta_m^{(\g)})^{-1} {\hat e}_m^{(R^*)}(s), \qquad s_i=q^{\lambda_i} t^{N+\xi_\g-i}.$$
The limit  extracts the leading monomial of the $s$ variables in the symmetric function, which is the term involving $s_1,s_2,...,s_m$ with positive powers.
For $m=1,2,...,N_\g$ this gives, using \eqref{newevaldef} \eqref{defthone} and \eqref{Eval}:
$\delta_{\lambda;m}^{(\g)}=\lim_{t\to \infty} t^{-m(N+\xi_g-\frac{m+1}{2})} {\hat e}_m(s)=q^{\lambda_1+\lambda_2+\cdots+\lambda_m}$.
By inspection, using \eqref{evalex}, \eqref{defthtwo} and \eqref{restofehat}, 
\begin{eqnarray*}
&&D_N^{(1)}: \quad \delta_{\lambda;N}^{(D_{N}^{(1)})}=q^{\frac{1}{2}(\lambda_1+\lambda_2+\cdots+\lambda_{N-1}+\lambda_N)},\ 
\delta_{\lambda;N-1}^{(D_{N}^{(1)})}=q^{\frac{1}{2}(\lambda_1+\lambda_2+\cdots+\lambda_{N-1}-\lambda_N)},\\
&&C_N^{(1)}: \quad \delta_{\lambda;N}^{(C_N^{(1)})}=q^{\frac{1}{2}(\lambda_1+\lambda_2+\cdots+\lambda_{N-1}+\lambda_N)},\\
&&D_{N+1}^{(2)}: \delta_{\lambda;N}^{(D_{N+1}^{(2)})} =q^{\frac{1}{2}(\lambda_1+\lambda_2+\cdots+\lambda_{N-1}+\lambda_N)}.
\end{eqnarray*}
The result can be uniformly written as $\Lambda^{\omega_m^*}=q^{\omega_m^*\cdot\lambda}$,  $\omega_m^*$ the fundamental weights of $R^*$.
\end{proof}

\begin{remark} 
Using the convention $V_0^{(\g)}=1$, Def. \ref{malphadef} implies
\begin{equation}\label{vanD} D_{m}^{(\g)}= \sum_{j=0}^m V_j^{(\g)},  \qquad m\leq N_\g.\end{equation}
\end{remark}

\section{Pieri rules}\label{appB}
In this section, we list some of the Pieri operators implementing the first Pieri rule for Koornwinder and $\g$-Macdonald polynomials. Their $q$-Whittaker limit gives $q$-difference Toda Hamiltonians for certain root systems.

\subsection{First Pieri rule for Koornwinder polynomials}
\label{appkorpieri}

The first Pieri operator for generic $(a,b,c,d)$ is the $q$-difference operator $\cH_1^{(a,b,c,d)}(s)$ of Theorem \eqref{kortopieri}. 

\begin{thm}\label{korpier}
The first Pieri operator for Koornwinder polynomials is
\begin{eqnarray*}
\cH_1^{(a,b,c,d)}(s;q,t)&=&a^{-1}t^{1-N}\varphi^{(a^*,b^*,c^*,d^*)}(s)+
\sum_{i=1}^N\left[ \left(\prod_{j=1}^{i-1}\frac{s_i-t^{-1}s_j }{s_i-s_j}\frac{q s_i-t s_j}{q s_i-s_j}\right)\, T_i\right.\nonumber \\
&&+
\frac{\prod_{u\in\{a^*,b^*,c^*,d^*\}}(1-u^{-1}s_i)(q- u s_i)}{(1-s_i^2)(q-s_i^2)^2(q^2-s_i^2)}
\nonumber \\
 &&\times
 \left.
 \left(
 \prod_{j=i+1}^{N} \frac{t^{-1}s_i-s_j}{s_i-s_j}\frac{t s_i-q s_j}{s_i-q s_j} 
 \prod_{j\neq i} \frac{q-t s_i s_j }{q-s_i s_j}\frac{1-t^{-1} s_i s_j}{1-s_i s_j}\right)T_i^{-1}
 \right],
\end{eqnarray*}
where 
\begin{eqnarray*}\varphi^{(a^*,b^*,c^*,d^*)}(x)&=&\sum_{\epsilon=\pm 1} \frac{\big(1-\frac{\epsilon}{ q^{1/2}}a^*\big)
\big(1-\frac{\epsilon }{q^{1/2}}b^*\big)
\big(1-\frac{\epsilon }{q^{1/2}}c^*\big)\big(1- \frac{\epsilon }{q^{1/2}}d^*\big)}{2(1-t)(1-q^{-1}t)}\nonumber \\
&&\qquad \qquad\qquad \times 
\left\{ \prod_{i=1}^N  
\frac{1-\frac{\epsilon\, t }{q^{1/2}}x_i}{1-\frac{\epsilon }{q^{1/2}}x_i}\,
\frac{1-\frac{\epsilon\, t }{q^{1/2}}x_i^{-1}}{1-\frac{\epsilon }{q^{1/2}}x_i^{-1}}
-t^N\right\},
\end{eqnarray*}
and $(a^*,b^*,c^*,d^*)$ are the dual Koornwinder parameters \eqref{dua}.
\end{thm}
\begin{proof}
We use the formula \eqref{PieriOpK} with $(a,b,c,d)\to (a^*,b^*,c^*,d^*)$. For simplicity, we work with the values $(a,b,c,d)$ and interchange them with their duals at the end.
We need the following.
\begin{lemma}
\begin{eqnarray*}
&&\Delta^{(a,b,c,d)}(x)^{-1} \Gamma_i \Delta^{(a,b,c,d)}(x)\\
&&\qquad =\frac{(1-\frac{1}{x_i^2})(1-\frac{1}{q x_i^2})}{(1-\frac{1}{ax_i})(1-\frac{1}{bx_i})(1-\frac{1}{cx_i})(1-\frac{1}{dx_i})}\prod_{j=1}^{i-1} \frac{1-\frac{q x_i}{t x_j}}{1-\frac{q x_i}{x_j}} \prod_{j=i+1}^N\frac{1-\frac{x_j}{x_i}}{1-\frac{x_j}{t x_i}} \prod_{j\neq i} \frac{1-\frac{1}{x_ix_j}}{1-\frac{1}{tx_ix_j}}\, \Gamma_i,\\
&&\Delta^{(a,b,c,d)}(x)^{-1} \Gamma_i^{-1} \Delta^{(a,b,c,d)}(x)\\
&&\qquad =\frac{(1-\frac{q}{ax_i})(1-\frac{q}{bx_i})(1-\frac{q}{cx_i})(1-\frac{q}{dx_i})}{(1-\frac{q}{x_i^2})(1-\frac{q^2}{x_i^2})}\prod_{j=1}^{i-1}\frac{1-\frac{x_i}{ x_j}}{1-\frac{x_i}{t x_j}} \prod_{j=i+1}^N\frac{1-\frac{qx_j}{ tx_i}}{1-\frac{qx_j}{x_i}} \prod_{j\neq i} \frac{1-\frac{q}{t x_ix_j}}{1-\frac{q}{x_ix_j}}\, \Gamma_i^{-1}.
\end{eqnarray*}
\end{lemma}
Using Lemma \ref{constlem}, the conjugation of the ${\mathcal D}_1^{(a,b,c,d)}$ operator boils down to that of the terms $\Phi_{i,\epsilon}^{(a,b,c,d)} \Gamma_i^\epsilon$. We have:
\begin{eqnarray*} 
&&\Delta^{(a,b,c,d)}(x)^{-1} \Phi_{i,+}^{(a,b,c,d)} \Gamma_i\Delta^{(a,b,c,d)}(x)=\frac{a b c d}{q} t^{2N-2} \left( \prod_{j=1}^{i-1} \frac{1-\frac{x_j}{tx_i}}{1-\frac{x_j}{x_i}}\frac{1-\frac{qx_i}{t x_j}}{1-\frac{qx_i}{x_j}}\right) \Gamma_i,\\
&&\Delta^{(a,b,c,d)}(x)^{-1} \Phi_{i,-}^{(a,b,c,d)} \Gamma_i^{-1}\Delta^{(a,b,c,d)}(x)\\
&&\qquad =t^{2N-2} \frac{(1-\frac{a}{x_i})(1-\frac{q}{ax_i})(1-\frac{b}{x_i})(1-\frac{q}{bx_i})(1-\frac{c}{x_i})(1-\frac{q}{cx_i})(1-\frac{d}{x_i})(1-\frac{q}{dx_i})}{(1-\frac{1}{x_i^2})(1-\frac{q}{x_i^2})^2(1-\frac{q^2}{x_i^2})}\\
&&\qquad \qquad \quad \times\, \left(\prod_{j=i+1}^{N} \frac{1-\frac{x_i}{t x_j}}{1-\frac{x_i}{ x_j}}\frac{1-\frac{qx_j}{t x_i}}{1-\frac{qx_j}{x_i}}\prod_{j\neq i} \frac{1-\frac{q}{t x_ix_j}}{1-\frac{q}{x_ix_j}}\frac{1-\frac{x_ix_j}{t}}{1-x_ix_j}\right)\Gamma_i^{-1}.
\end{eqnarray*}
We finally use the formula \eqref{PieriOpK} for $m=1$ and substitute the above results with the change 
$(a,b,c,d)\to (a^*,b^*,c^*,d^*)$, and the Theorem follows.
\end{proof}

\subsection{First Pieri rules for $\g$}\label{firpierapp}

We may now specialize the result of Theorem \ref{korpier} to $\g$, according to Table \ref{korspec}. We first treat the cases where $\g=D_N^{(1)},C_N^{(1)},A_{2N-1}^{(2)}$ for which the constant term $\varphi^{(\g^*)}(x)$ vanishes\footnote{ Here we show the explicit dependence on $\lL_i$ and $t$ (rather than $s_i=\lL_i\, t^{\xi_\g+N-i}$), as $\xi_\g$ varies with $\g$. This makes the limit $t\to\infty$ easier to follow.}:
\begin{eqnarray*}
\cH_1^{(D_N^{(1)})}&=&\sum_{i=1}^N\left\{\prod_{j<i} \frac{t^{i-j-1} \lL_j -\lL_i}{t^{i-j}\lL_j-\lL_i}\frac{t^{i-j+1} \lL_j -q \lL_i}{t^{i-j}\lL_j-q \lL_i} \, T_i \right. \\
&&\left. +\prod_{j>i} \frac{t^{j-i+1} \lL_i-q\lL_j}{t^{j-i}\lL_i-q\lL_j}\frac{t^{j-i-1} \lL_i-\lL_j}{t^{j-i}\lL_i-\lL_j} 
\prod_{j\neq i}\frac{1-t^{2N-i-j-1}\lL_i\lL_j}{1-t^{2N-i-j}\lL_i\lL_j} \frac{t^{2N-i-j+1}\lL_i\lL_j-q}{t^{2N-i-j}\lL_i \lL_j-q}\, T_i^{-1}\right\},\\
\cH_1^{(C_N^{(1)})}&=&\sum_{i=1}^N \left\{ \prod_{j<i} \frac{t^{j-i+1} \lL_i-\lL_j}{t^{j-i} \lL_i-\lL_j}\frac{\lL_j-qt^{j-i-1}\lL_i}{\lL_j-qt^{j-i}\lL_i} \,T_i \right. \\
&&+ \frac{1 -t^{N-i}\lL_i}{1-t^{N+1-i}\lL_i}\frac{t^{N+2-i}\lL_i-q}{t^{N+1-i}\lL_i-q} 
\prod_{j>i}\frac{t^{j-i+1} \lL_i-q\lL_j}{t^{j-i} \lL_i-q\lL_j}\frac{\lL_j-t^{j-i-1}\lL_i}{\lL_j-t^{j-i}\lL_i}\\
&&\qquad \times \left.
\prod_{j\neq i} \frac{1 -t^{2N+1-i-j}\lL_i \lL_j}{1-t^{2N+2-i-j}\lL_i\lL_j} \frac{t^{2N+3-i-j}\lL_i\lL_j-q}{t^{2N+2-i-j}\lL_i\lL_j-q}\, T_i^{-1}  \right\}.
\end{eqnarray*}
\begin{eqnarray*}
\cH_1^{(A_{2N-1}^{(2)})}&=&\sum_{i=1}^N\left\{ 
\prod_{j<i} \frac{t^{j-i+1}\Lambda_i -\Lambda_j}{t^{j-i}\Lambda_i -\Lambda_j} \frac{\lL_j-q t^{j-i-1}\lL_i}{\lL_j-q t^{j-i}\lL_i}\, T_i\right. \\
&&\left. +\frac{t^{2N-2i}\Lambda_i^2-1}{t^{2N+1-2i}\Lambda_i^2-1}
\frac{t^{2N+2-2i}\Lambda_i^2-q^2}{t^{2N+1-2i}\Lambda_i^2-q^2}
\prod_{j>i} 
\frac{t^{j-i+1}\Lambda_i-q \Lambda_j}{t^{j-i}\Lambda_i-q \Lambda_j}\frac{t^{j-i-1}\Lambda_i-\Lambda_j}{t^{j-i}\Lambda_i-\Lambda_j} \right. \\
&&\left. \times 
\prod_{j\neq i} 
\frac{t^{2N-i-j}\Lambda_i\Lambda_j-1}{t^{2N+1-i-j}\Lambda_i\Lambda_j-1}
\frac{t^{2N+2-i-j}\Lambda_i\Lambda_j-q}{t^{2N+1-i-j}\Lambda_i\Lambda_j-q}\, T_i^{-1}\right\}.
\end{eqnarray*}

We now treat the cases $\g=B_N^{(1)},D_{N+1}^{(2)},A_{2N}^{(2)}$ for which $\varphi^{(\g^*)}$ has a non-trivial contribution. We find:
\begin{eqnarray*}
&&\cH_1^{(B_N^{(1)})}= G^{(B_N^{(1)})}(\Lambda)+\sum_{i=1}^N\left\{ 
\prod_{j< i} \frac{t^{j-i+1} \lL_i-\lL_j}{t^{j-i} \lL_i-\lL_j}\frac{\lL_j-qt^{j-i-1}\lL_i}{ \lL_j-qt^{j-i}\lL_i}\,T_i\right.\\
&&+\frac{1 -t^{2N-2i}\lL_i^2}{1-t^{2N+1-2i}\lL_i^2}\,\frac{q -t^{2N-2i}\lL_i^2}{q -t^{2N+1-2i}\lL_i^2}\,\frac{t^{2N+2-2i}\lL_i^2-q^2}{t^{2N+1-2i}\lL_i^2-q^2}\,\frac{t^{2N+2-2i}\lL_i^2-q}{t^{2N+1-2i}\lL_i^2-q}\\
&&\times \left.\prod_{j>i} \frac{t^{j-i+1} \lL_i-q\lL_j}{t^{j-i} \lL_i-q\lL_j}\frac{\lL_j-t^{j-i-1}\lL_i}{\lL_j-t^{j-i}\lL_i} \prod_{j\neq i} \frac{1 -t^{2N-i-j}\lL_i \lL_j}{1-t^{2N+1-i-j}\lL_i \lL_j} \frac{t^{2N+2-i-j}\lL_i \lL_j-q}{t^{2N+1-i-j}\lL_i \lL_j-q} T_i^{-1} \right\},
\end{eqnarray*}
where
$$
G^{(B_N^{(1)})}(\lL)=-1+\frac{1}{2}\sum_{\epsilon=\pm 1}
\prod_{i=1}^N \frac{t^{N+\frac{3}{2}-i}\Lambda_i -\epsilon q^{1/2}}{t^{N+\frac{1}{2}-i}\Lambda_i -\epsilon q^{1/2}}\, \frac{t^{N-\frac{1}{2}-i}\Lambda_i -\epsilon q^{-1/2}}{t^{N+\frac{1}{2}-i}\Lambda_i -\epsilon q^{-1/2}}.
$$
\begin{eqnarray*}
&&\cH_1^{(D_{N+1}^{(2)})}=G^{(D_{N+1}^{(2)})}(\lL)+\sum_{i=1}^N\left\{ \prod_{j<i} \frac{t^{j-i+1}\lL_i-\lL_j}{t^{j-i}\lL_i-\lL_j}
\frac{qt^{j-i-1}\lL_i-\lL_j}{qt^{j-i}\lL_i-\lL_j}\, T_i\right.\\
&&+\frac{1-t^{N-i}\Lambda_i}{1-t^{N+1-i}\Lambda_i} \frac{q^{1/2}-t^{N-i}\Lambda_i}{q^{1/2}-t^{N+1-i}\Lambda_i}
 \frac{q-t^{N+2-i}\Lambda_i}{q-t^{N+1-i}\Lambda_i} \frac{q^{1/2}-t^{N+2-i}\Lambda_i}{q^{1/2}-t^{N+1-i}\Lambda_i}\\
 &&\left. \times
\prod_{j\neq i} \frac{1-t^{2N+1-i-j}\lL_i\lL_j}{1-t^{2N+2-i-j}\lL_i\lL_j} 
\frac{q-t^{2N+3-i-j}\lL_i\lL_j}{q-t^{2N+2-i-j}\lL_i\lL_j}
\prod_{j>i}\frac{ t^{j-i-1}\lL_i-\lL_j}{ t^{j-i}\lL_i-\lL_j} \frac{t^{j-i+1}\lL_i-q\lL_j}{t^{j-i}\lL_i-q\lL_j}
\, T_i^{-1}\right\},
\end{eqnarray*}
where 
$$
G^{(D_{N+1}^{(2)})}(\lL)=\frac{(1-t q^{-1/2})(1+q^{-1/2})}{1-t q^{-1}}\left\{ 
\prod_{i=1}^N \frac{q^{1/2}- \lL_it^{N+2-i}}{q^{1/2}- \lL_it^{N+1-i}}\frac{1-q^{1/2}\lL_it^{N-i}}{1-q^{1/2}\lL_it^{N+1-i}}-1\right\},
$$
and finally
\begin{eqnarray*}
&&\cH_1^{(A_{2N}^{(2)})}=G^{(A_{2N}^{(2)})}(\lL)+\sum_{i=1}^N\left\{ \prod_{j<i} \frac{t^{j-i+1}\lL_i-\lL_j}{t^{j-i}\lL_i-\lL_j}
\frac{qt^{j-i-1}\lL_i-\lL_j}{qt^{j-i}\lL_i-\lL_j}\, T_i\right.\\
&&+\frac{1-t^{N-i}\Lambda_i}{1-t^{N+1-i}\Lambda_i} \frac{q-t^{2N-2i+1}\Lambda_i^2}{q-t^{2N+2-2i}\Lambda_i^2}
 \frac{q^2-t^{2N+3-2i}\Lambda_i^2}{q^2-t^{2N+2-2i}\Lambda_i^2} \frac{q-t^{N+2-i}\Lambda_i}{q-t^{N+1-i}\Lambda_i}\\
 &&\left. \times
 \prod_{j>i}\frac{ t^{j-i-1}\lL_i-\lL_j}{ t^{j-i}\lL_i-\lL_j} \frac{t^{j-i+1}\lL_i-q\lL_j}{t^{j-i}\lL_i-q\lL_j}
\prod_{j\neq i} \frac{1-t^{2N+1-i-j}\lL_i\lL_j}{1-t^{2N+2-i-j}\lL_i\lL_j} 
\frac{q^2-t^{2N+3-i-j}\lL_i\lL_j}{q^2-t^{2N+2-i-j}\lL_i\lL_j}
\, T_i^{-1}\right\},
\end{eqnarray*}
where
\begin{eqnarray*}
G^{(A_{2N}^{(2)})}(\lL)&=& \sum_{\epsilon=\pm 1} \frac{(1-\epsilon t q^{-1/2})(1+\epsilon q^{-1/2})}{2(1-t q^{-1})} \left\{ \prod_{i=1}^N \frac{q^{1/2}-\epsilon \lL_it^{N+2-i}}{q^{1/2}-\epsilon \lL_it^{N+1-i}}\frac{\epsilon-q^{1/2}\lL_it^{N-i}}{\epsilon-q^{1/2}\lL_it^{N+1-i}}-1\right\}.
\end{eqnarray*}

\begin{remark}\label{remlimops}
When we consider the $q$-Whittaker limit $t\to \infty$, the constant pieces $G^{(\g)}(\lL)$ in the latter three cases $\g=B_N^{(1)},D_{N+1}^{(2)},A_{2N}^{(2)}$ provide an extra constant term
in the limiting Hamiltonians $H_1^{(\g)}$. These read respectively:
$$\lim_{t\to\infty} G^{(B_N^{(1)})}(\Lambda)=\frac{q^{-1}}{\lL_{N-1}\lL_N}-\frac{1+q^{-1}}{\lL_N^2} ,\ \lim_{t\to\infty} G^{(D_{N+1}^{(2)})}(\Lambda)=-\frac{1+q^{-1/2}}{\lL_N} ,\ \lim_{t\to\infty} G^{(A_{2N}^{(2)})}(\Lambda)=-\frac{1}{\lL_N} .
$$
\end{remark}

\section{Proof of Theorem \ref{lemone}
}\label{appC}
\label{appCtwo}

Here, we provide the proof of  Theorem \ref{lemone}, that the time translation operators $g^{(\g)}$ of Equation \eqref{variousg}
commute with the limiting first Pieri operator/Toda Hamiltonian  $H_1^{(\g)}$ of Equation \eqref{Hamil}:
\begin{equation} \label{commute} g^{(\g)} H_1^{(\g)}= H_1^{(\g)}g^{(\g)} .\end{equation}

We divide each Hamiltonian into bulk and boundary pieces, study how $g$ commutes with each,
and then sum up the contributions to prove the commutation \eqref{commute}.

\begin{lemma}
For all $\g$, 
\begin{eqnarray}
g_T\, \left(1-\lL^{-\al_a}\right) &=&   \left(1-q \lL^{-\al_a}\right) T^{-\al_a} g_T \quad (a=0,1,...,N-1),\label{lemc1}\\
g_\lL\,  \left(1-\lL^{-\al_a}\right) T_a&=& \left(1-q^{-1}\lL^{-\al_{a+1}}\right) T_{a+1}\, g_\lL \quad (a=1,2,...,N-2),\label{lemc2}\\
g_\lL\, \left(1-\lL^{-\al_a}\right) T_a^{-1}&=& \left(1-q^{-1}\lL^{-\al_{a-1}}\right)T_a^{-1}\, g_\lL\quad (a=1,2,...,N-1) ,\label{lemc3}
\end{eqnarray}
with the convention $\al_0=0$.
\end{lemma}

\subsection{Bulk and boundary terms}
Each of the first Pieri  operators \eqref{Hamil} is the sum of a bulk and  boundary piece, where the bulk piece has the same form for all $\g$:
$$H_1^{[m]}:=\sum_{a=1}^m \left(1-\frac{\lL_a}{\lL_{a-1}}\right)\, T_a+\left(1-\frac{\lL_{a+1}}{\lL_{a}}\right)\, \frac{1}{T_a},$$
with the convention that $\frac{1}{\lL_0}=0$, for suitable value of $m= m_\g$: $m_\g=N-2$ for $\g=D_N^{(1)},B_N^{(1)},A_{2N}^{(2)}$, and $m_\g=N-1$ for all other cases. 

The operators $g^{(\g)}$ of \eqref{variousg} also have some common structures. For 
$\g=B_N^{(1)},A_{2N-1}^{(2)},D_{N+1}^{(2)}$, 
$g=g_T^{\frac{1}{2}} \al g_T^{\frac{1}{2}} g_\lL \beta $ where $\al,\beta$ are functions of $\lL_N$ only. For
$\g=D_N^{(1)}$  $g=g_T g_\lL \beta$ with $\beta$  a function of $\lL_{N-1}\lL_N$ only. 
The cases $\g=C_N^{(1)},A_{2N}^{(2)}$ are different, and have a $g$ operator of the form $g=g_Tg_\lL \al g_Tg_\lL\beta$, where $\al,\beta$ are functions of $\lL_N$ only.

\begin{lemma}\label{lembulk}
For all $\g$ and $m\leq m_\g$, and $t_1=2$ for $\g=C_N^{(1)},A_{2N}^{(2)}$, $t_1=1$ otherwise,
$$
g^{(\g)}\, H_1^{[m]} \, (g^{(\g)})^{-1} =(g_T\,g_\lL)^{t_1} \, H_1^{[m]} \, (g_T\,g_\lL)^{-t_1} .
$$
\end{lemma}
\begin{proof}
The only subtleties arise in the cases $\g=B_N^{(1)},A_{2N}^{(2)}$. When $\g=B_N^{(1)}$, with $g=g_T^{\frac{1}{2}} \al g_T^{\frac{1}{2}} g_\lL \beta$, the action of $g_T^{\frac{1}{2}} g_\lL \beta$ on $H_1^{[m]}$ creates a term proportional to $T_{m+1}$,
as a consequence of 
\begin{equation}\label{stepzero}
g_T\, g_\lL\, \left(1-\frac{\lL_a}{\lL_{a-1}}\right)\, T_a =\left(T_a-\frac{\lL_{a+1}}{\lL_a}\,T_{a+1}\right)\, g_Tg_\lL
\end{equation}
for $a=m$.
This term commutes with $\al$ only if $m+1\leq N-1$, hence we set $m_\g=N-2$. 
Similarly, when $\g=A_{2N}^{(2)}$, with $g=g_Tg_\lL \al g_Tg_\lL\beta$, the action of 
$g_Tg_\lL\beta$ on $H_1^{[m]}$ creates a term proportional to $T_{m+1}$ which commutes with $\al$ only if $m+1\leq N-1$, hence we set $m_\g=N-2$ as well. 
\end{proof}
\begin{lemma}\label{bulklem}
We have the commutation relations for $m=1,2,...,N-1$:
$$g_T\,g_\lL \, H_1^{[m]} =\left\{H_1^{[m]} +\frac{\lL_{m+1}}{\lL_m}\,\left(\frac{1}{T_m}- T_{m+1}\right)\right\}\, g_T \,g_\lL .$$
\end{lemma}
\begin{proof}
We simply sum over the relevant values of $a$ the commutation relations \eqref{stepzero}, as well as the following:
\begin{equation}\label{comtinv} g_T\, g_\lL\, \left(1-\frac{\lL_{a+1}}{\lL_{a}}\right)\, \frac{1}{T_a} =\left( \frac{1}{T_a}-\frac{\lL_{a}}{\lL_{a-1}}\, \frac{1}{T_{a-1}}\right)\, g_Tg_\lL ,
\end{equation}
obtained by combining (\ref{lemc1}-\ref{lemc3}).
\end{proof}

The cases $\g=C_N^{(1)}$ and $A_{2N}^{(1)}$ are special, as we will have to use Lemma \ref{bulklem} {\it twice}. This is deferred to Section~\ref{CAsecs} below.

For all the other cases, we apply Lemmas \ref{lembulk} and \ref{bulklem} to commute $g$ through the relevant bulk contributions to $H_1^{(\g)}$ corresponding to $a=1,2,..,m_\g$. The remaining terms of $H_1^{(\g)}$ are the ``boundary" terms, corresponding to $m_\g<a\leq N$ and possibly to constant terms independent of $T_a^{\pm1}$. We address them individually in the following Sections \ref{D1sec} through \ref{D2sec}. In all cases, the Lemma \ref{lemone} follows from summing all contributions
from the bulk $a=1,2,...,m_\g$ and the boundary $a=m_\g+1,...,N$ plus constant terms.

\subsection{Case $\g=D_{N}^{(1)}$}\label{D1sec} Recall that $g=g_Tg_\lL\prod_{n=0}^\infty \left(1-\frac{q^n}{\lL_{N-1}\lL_N}\right)^{-1}$, and that $m_\g=N-2$.
For $a=N-1$, we have:
\begin{eqnarray*}
&&g\left(1-\frac{\lL_{N-1}}{\lL_{N-2}}\right) T_{N-1}= \left( T_{N-1} -\frac{\Lambda_{N}}{\Lambda_{N-1}}\, T_{N}-\frac{1}{\Lambda_{N-1}\Lambda_{N}} T_{N}^{-1}+\frac{1}{\Lambda_{N-1}^2} T_{N-1}^{-1}\right)\, g, \\
&&g\, \left(1-\frac{\Lambda_{N}}{\Lambda_{N-1}}\right)\left(1-\frac{1}{\Lambda_{N-1}\Lambda_{N}}\right) T_{N-1}^{-1}=
\left(T_{N-1}^{-1}-\frac{\Lambda_{N-1}}{\Lambda_{N-2}} \,T_{N-2}^{-1}\right)\, g,
\end{eqnarray*}
and for $a=N$:
\begin{eqnarray*}
&&g\, \left(1-\frac{\Lambda_{N}}{\Lambda_{N-1}}\right) T_{N}=\left( T_{N} -\frac{1}{\Lambda_{N-1}\Lambda_{N}} T_{N-1}^{-1}\right)\, g,\\
&&g\,\left(1-\frac{1}{\Lambda_{N-1}\Lambda_{N}}\right)\, T_N^{-1}=\left(T_{N}^{-1}-\frac{\Lambda_{N}}{\Lambda_{N-1}}\, T_{N-1}^{-1}\right)\, g.
\end{eqnarray*}
Eq. \eqref{commute} follows by summing all bulk and boundary contributions.

\subsection{Case $\g=B_{N}^{(1)}$}\label{B1sec} Recall that $g=g_1\,g_2$, $g_1=g_T^{\frac{1}{2}} \al$, $g_2= g_1 g_\lL$ with $\al= \prod_{n=0}^\infty \left(1-\frac{q^n}{\lL_N^2}\right)^{-1}$, and that $m_\g=N-2$.
The terms $a=N-1,N$ and the constant term read
\begin{eqnarray*}
&&g\, \left(1-\frac{\Lambda_{N}}{\Lambda_{N-1}}\right) \,T_{N-1}^{-1}=T_{N-1}^{-1}\, g_T^{\frac{1}{2}} \,\al g_T^{\frac{1}{2}} \, 
 \left(1-\frac{\Lambda_{N-1}}{\Lambda_{N-2}}\right) \al \,g_{\Lambda}=\left(T_{N-1}^{-1}-\frac{\Lambda_{N-1}}{\Lambda_{N-2}}T_{N-2}^{-1}\right)\, g,\\
&&g\, \left(1-\frac{\Lambda_{N-1}}{\Lambda_{N-2}}\right) \,T_{N-1}
=T_{N-1}\, g_1 g_T^{\frac{1}{2}}\, \left(1-\frac{\Lambda_{N}}{\Lambda_{N-1}}\right) \,\al g_{\Lambda}
=g_1\, \left(T_{N-1}-q^{-1}\frac{\Lambda_{N}}{\Lambda_{N-1}}T_N^{\frac{1}{2}} T_{N-1}^{\frac{1}{2}}\right) \, g_2,\\
&&g\, \left(1-\frac{\Lambda_{N}}{\Lambda_{N-1}}\right) \,T_N=g_{1}\, 
\left(1-\frac{1}{\Lambda_{N}^2}T_N^{-1}\right)\left(1-q^{-2}\frac{1}{\Lambda_{N}^2}T_N^{-1}\right)\, T_N\,g_{2}, \\
&&g\,\left(1-\frac{1}{\Lambda_{N}^2}\right)\left(1-q^{2}\frac{1}{\Lambda_{N}^2}\right)\,T_N^{-1}
=g_1\, \left(T_N^{-1}-q^{-1}\frac{\Lambda_{N}}{\Lambda_{N-1}}T_{N-1}^{-\frac{1}{2}}T_{N}^{-\frac{1}{2}}\right)\, g, \\
&&g\, \left(\frac{q^{-2}}{\Lambda_{N-1}\Lambda_N}-\frac{1+q^{-2}}{\Lambda_{N}^2}\right)=
 g_{1} \left(\frac{q^{-1}}{\Lambda_{N-1}\Lambda_N}T_{N-1}^{-\frac{1}{2}}T_{N}^{-\frac{1}{2}}-q^2\frac{1+q^{-2}}{\Lambda_{N}^2}T_N^{-1}\right)\,g_{2} .
\end{eqnarray*}
Summing these five terms gives
\begin{eqnarray*}
&&g\left\{\left(1-\frac{\Lambda_{N}}{\Lambda_{N-1}}\right) \,T_{N-1}^{-1}+\left(1-\frac{\Lambda_{N-1}}{\Lambda_{N-2}}\right) \,T_{N-1}+ \left(1-\frac{\Lambda_{N}}{\Lambda_{N-1}}\right) \,T_N \right.\\
&&\qquad\qquad +\left.
\left(1-\frac{1}{\Lambda_{N}^2}\right)\left(1-q^{2}\frac{1}{\Lambda_{N}^2}\right)\,T_N^{-1}+
\frac{q^{-2}}{\Lambda_{N-1}\Lambda_N}-\frac{1+q^{-2}}{\Lambda_{N}^2} \right\}\\
&&=\left\{ T_{N-1}+\left(1-\frac{\Lambda_N}{\Lambda_{N-1}}\right)T_N-\frac{\lL_{N-1}}{\lL_{N-2}}T_{N-2}^{-1}+\left(1-\frac{\Lambda_N}{\Lambda_{N-1}}\right)T_{N-1}^{-1}\right. \\
&&\qquad\qquad  +\left. \left(1-\frac{1}{\Lambda_{N}^2}\right)\left(1-q^{2}\frac{1}{\Lambda_{N}^2}\right)\,T_N^{-1}+\frac{q^{-2}}{\Lambda_{N-1}\Lambda_N}-\frac{1+q^{-2}}{\Lambda_N^2}
 \right\}\, g.
\end{eqnarray*}
Eq. \eqref{commute} follows by summing all bulk and boundary contributions.

\subsection{Case $\g=A_{2N-1}^{(2)}$}\label{A2osec} Recall that $g=g_Tg_\lL\al$, with $\al=\prod_{n=0}^\infty \left(1-\frac{q^{2n}}{\lL_N^2}\right)^{-1}$, and $m_\g=N-1$. For the two boundary terms for $a=N$ we have:
\begin{eqnarray*}
g\,\left(1-\frac{\Lambda_N}{\Lambda_{N-1}}\right)T_N&=&g_T\, \left(1-q^{-2}\frac{1}{\Lambda_N^2}\right) T_N\, g_\lL\,\al=\left(T_N-\frac{1}{\Lambda_N^2}T_{N}^{-1}\right)\,g,\\
g\,\left(1-\frac{1}{\Lambda_{N}^2}\right)T_N^{-1}\, g^{-1}&=&g_T\,\left(1-q^{-1}\frac{\Lambda_{N}}{\Lambda_{N-1}}\right)T_N^{-1} \,g_\lL\,\al=
\left(T_N^{-1}-\frac{\Lambda_{N}}{\Lambda_{N-1}}T_{N-1}^{-1}\right)\, g.
\end{eqnarray*}
Eq. \eqref{commute} follows by summing all bulk and boundary contributions.

\subsection{Case $\g=D_{N+1}^{(2)}$}\label{D2sec} Recall that $g=g_Tg_\lL\al$, with $\al=\prod_{n=0}^\infty \left(1-\frac{q^{\frac{n}{2}}}{\lL_N}\right)^{-1}$, and $m_\g=N-1$. For the two boundary terms for $a=N$ and the constant term we have:
\begin{eqnarray*}
&&g\,\left(1-\frac{\Lambda_N}{\Lambda_{N-1}}\right)T_N
=\left(T_N-\frac{1+q^{-\frac{1}{2}}}{\Lambda_N}+\frac{q}{\Lambda_N^{2}}T_N^{-1}\right)\, g,\\
&&g\,\left(1-\frac{1}{\Lambda_{N}}\right)\left(1-\frac{q^\frac{1}{2}}{\Lambda_{N}}\right)T_N^{-1}=
g_T\,\left(1-q^{-1}\frac{\Lambda_{N}}{\Lambda_{N-1}}\right)T_N^{-1} \, g_\lL\,\al=
\left(T_N^{-1}-\frac{\Lambda_{N}}{\Lambda_{N-1}}T_{N-1}^{-1}\right)\, g,\\
&&g\,\left(-\frac{1+q^{-\frac{1}{2}}}{\Lambda_{N}}\right)= -\frac{1+q^\frac{1}{2}}{\Lambda_{N}}T_N^{-1}\, g.
\end{eqnarray*}
Eq. \eqref{commute} follows by summing all bulk and boundary contributions.

\subsection{Cases $\g=C_{N}^{(1)},A_{2N}^{(2)}$} \label{CAsecs}
Recall the general structure of the operator $g=g_1 g_2$ with $g_1=g_T g_\lL \al $ and $g_2=g_T g_\lL \beta $, where $\al,\beta$ are functions of $\lL_N$ only. We note that $g_2$ is the {\it same} for both cases, with $\beta=\prod_{n=0}^\infty \left(1-\frac{q^n}{\lL_N}\right)^{-1}$. We note also that $H_1^{[N-1]}$ is the same for both cases, and Lemma \ref{bulklem} gives the commutation of $g_2$ through $H_1^{[N-1]}$:
\begin{equation}
\label{bulkCA}
g_2\,  H_1^{[N-1]}= \left\{ H_1^{[N-1]} +\frac{\lL_{N}}{\lL_{N-1}}\,\left(\frac{1}{T_{N-1}}- T_{N}\right)\right\}\, g_2.
\end{equation}
We now need to add the boundary and constant terms.
For $\g=C_N^{(1)}$ we have for $a=N$ (and no constant term):
\begin{eqnarray*}
g_{2}\, \left(1-\frac{\Lambda_{N}}{\Lambda_{N-1}}\right) T_N=
\left(T_N-\frac{q^{-\frac{1}{2}}}{\Lambda_{N}}\right)\, g_{2},\ \
g_{2}\, \left(1-\frac{1}{\Lambda_{N}}\right) T_N^{-1}= 
\left(T_N^{-1}-\frac{\Lambda_{N}}{\Lambda_{N-1}}T_{N-1}^{-1}\right)\, g_{2}.
\end{eqnarray*}
For $\g=A_{2N}^{(2)}$ we have for $a=N$ and the constant term:
\begin{eqnarray*}
g_2\, \left( 1- \frac{\lL_N}{\lL_{N-1}}\right)T_N &=&g_T\,\left(1- \frac{q^{-1}}{\lL_N}\right)T_N\,g_\lL\,\beta
=\left(T_N-\frac{q^{-\frac{1}{2}}}{\lL_N}\right)\, g_2,\\
g_2\, \left( 1- \frac{1}{\lL_N}\right)T_N^{-1}&=&\left(T_N^{-1}-\frac{\lL_N}{\lL_{N-1}}T_{N-1}^{-1}\right)\, g_2,\\
g_2\, \left( 1- \frac{1}{\lL_N}\right) &=&\left(1-  \frac{q^\frac{1}{2}}{\lL_N}T_N^{-1}\right)\, g_2.
\end{eqnarray*}
Adding these to the bulk contributions \eqref{bulkCA}, we deduce the following commutations:
\begin{eqnarray}
\g=C_N^{(1)}: && g_2 H_1=\left\{ H_1+ \frac{1}{\lL_N}(T_N^{-1}-q^{-\frac{1}{2}})\right\} g_2,\label{g2C}\\
\g=A_{2N}^{(2)}:&& g_2 H_1=\left\{ H_1+ \frac{1-q^\frac{1}{2}}{\lL_N}(T_N^{-1}-q^{-\frac{1}{2}})\right\} g_2.\label{g2A}
\end{eqnarray}
We must now commute $g_1$ through this. For both cases, we split again $H_1$ into a bulk piece $H_1^{[N-1]}$
and boundary pieces and constant terms. For the bulk contribution, we use again Lemma \ref{bulklem} to write in both cases
\begin{equation}
\label{bulk2CA}
g_1\,  H_1^{[N-1]}= \left\{ H_1^{[N-1]} +\frac{\lL_{N}}{\lL_{N-1}}\,\left(\frac{1}{T_{N-1}}- T_{N}\right)\right\}\, g_1.
\end{equation}
Then for $\g=C_N^{(1)}$ and $a=N$:
\begin{eqnarray*}
g_{1}\, \left(1-\frac{\Lambda_{N}}{\Lambda_{N-1}}\right) T_N&=&
T_N\, g_{1},\\
g_{1}\, \left(1-\frac{1}{\Lambda_{N}}\right) T_N^{-1}&=& \left(1-\frac{q^\frac{1}{2}}{\Lambda_{N}}T_N^{-1}\right) \left(T_N^{-1}-\frac{\Lambda_{N}}{\Lambda_{N-1}}T_{N-1}^{-1}\right)\, g_{1},
\end{eqnarray*}
and for $\g=A_{2N}^{(2)}$, $a=N$ and the constant term:
\begin{eqnarray*}
g_1\, \left( 1- \frac{\lL_N}{\lL_{N-1}}\right)T_N&=&g_T\,\left(1- \frac{q^{-\frac{1}{2}}}{\lL_N}\right)T_N\, g_\lL\,\al
=\left(T_N-\frac{1}{\lL_N}\right)\, g_1,\\
g_1\, \left( 1- \frac{q^\frac{1}{2}}{\lL_N}\right)T_N^{-1}&=&\left(T_N^{-1}-\frac{\lL_N}{\lL_{N-1}}T_{N-1}^{-1}\right)\, g_1,\\
g_1\, \left( 1- \frac{q^{-\frac{1}{2}}}{\lL_N}\right) &=&\left(1- \frac{1}{\lL_N}T_N^{-1}\right)\, g_1.
\end{eqnarray*}
Summing these with the bulk contributions \eqref{g2C}and \eqref{g2A} respectively, we arrive at
\begin{eqnarray}
\g=C_N^{(1)}: && g_1 \,\left\{ H_1+ \frac{1}{\lL_N}(T_N^{-1}-q^{-\frac{1}{2}})\right\}=H_1\, g_1,\label{gfin2C}\\
\g=A_{2N}^{(2)}:&& g_1\,\left\{ H_1+ \frac{1-q^{\frac{1}{2}}}{\lL_N}(T_N^{-1}-q^{-\frac{1}{2}})\right\}=H_1\, g_1.\label{gfin2A}
\end{eqnarray}
Finally, Lemma \ref{lemone} follows for $\g=C_N^{(1)},A_{2N}^{(2)}$ by combining the commutation relations 
(\ref{g2C}-\ref{g2A}) and (\ref{gfin2C}-\ref{gfin2A}). In all cases, Eq. \eqref{commute} follows by summing all bulk and boundary contributions.

\bibliographystyle{alpha}

\newcommand{\etalchar}[1]{$^{#1}$}
\def\cprime{$'$}

\end{document}